%% file: main.tex
\documentclass[11pt,a4paper,reqno]{amsart}
\pdfoutput=1
\usepackage{amsmath,amsthm,latexsym,amssymb,wasysym,mathtools}
\usepackage[margin=1.0in,tmargin=0.9in,bmargin=0.80in]{geometry}
\usepackage{hyperref}
\usepackage{bbm}
\usepackage{xargs}
\usepackage{enumerate}
\usepackage{enumitem}
\setlist[enumerate]{label={\upshape(\roman*)}}
\usepackage[textsize=footnotesize]{todonotes}
\usepackage[flushleft]{threeparttable}
\usepackage{multirow}
\usepackage{longtable}
\usepackage{times}
\usepackage{microtype}
\usepackage[numbers]{natbib}
\usepackage{xcolor}
\usepackage{tabularx}
\usepackage{ltablex}
\usepackage{booktabs}
\theoremstyle{plain}
\newtheorem{theorem}{Theorem}
\newtheorem{lemma}{Lemma}
\newtheorem{proposition}{Proposition}
\newtheorem{corollary}{Corollary}
\usepackage{natbib}
\theoremstyle{definition}

\theoremstyle{remark}
\newtheorem{remark}{Remark}

\newcommand{\R}{\mathbb{R}}
\newcommand{\E}{\mathbb{E}}
\newcommand{\KL}{\mathrm{KL}}
\newcommand{\W}{\mathcal{W}_2}

\newcounter{appendixsection} % independent counter

\usepackage{tikz}
\usetikzlibrary{arrows.meta,positioning,calc}
\usepackage{float}
\DeclareMathOperator{\Law}{Law}
\usepackage{float}
\usepackage{graphicx}
\usepackage{subcaption} 
\usepackage{caption}
\title[]%
{Error estimates for tamed Euler and Randomized Euler schemes for SDEs with locally Lipschitz drift with applications to non-logconcave sampling and optimization}
\author[I.~Lytras]{Iosif Lytras$^{\dagger}$}
\address{Archimedes, Athena Research Center, Greece and National Technical University of Athens, Athens,Greece.}
\email{i.lytras@athenarc.gr}

\author[A.~Ntousis]{Angelos Ntousis$^{\ast}$}
\address{$^{\ast}$ Archimedes, Athena Research Center, Greece}
\email{aggelosntousis02@gmail.com}

\begin{document}
\begin{abstract}
In this paper, we study the numerical discretization of stochastic differential equations with locally Lipschitz, super-linearly growing drift, and the resulting implications 
for sampling from non-log-concave distributions satisfying a logarithmic Sobolev inequality. 
In this regime, the classical Euler--Maruyama scheme underlying the unadjusted Langevin 
algorithm (ULA) is known to be unstable. We analyze the KL-accelerated tamed unadjusted 
Langevin algorithm (kTULA) and introduce a new tamed randomized midpoint scheme, termed 
tRLMC. Building on the shifted-composition approach of \cite{chewi2024local}, we develop 
two new local-error frameworks that yield finite-time, non-asymptotic error estimates 
against the underlying SDE---in KL divergence for kTULA, and in total variation for tRLMC 
---valid for general locally Lipschitz drift. Specializing these frameworks to the sampling 
problem under a logarithmic Sobolev inequality, we obtain a near-optimal 
$\widetilde{O}(\varepsilon^{-1/2})$ iteration complexity for kTULA in KL divergence, 
with corresponding guarantees in total variation and Wasserstein distance. We further 
establish, for the first time, a non-asymptotic guarantee in total variation for a tamed 
randomized Langevin scheme under super-linear drift growth, together with the corresponding 
Wasserstein-distance bound, both with $\widetilde{O}(\varepsilon^{-1})$ complexity for 
tRLMC. As a consequence, both schemes yield non-asymptotic  bounds for a
non-convex excess-risk optimization problem.
\end{abstract}
\maketitle

%----------------------------------------------------------------------

%----------------------------------------------------------------------

\input{body/Introduction}

\input{body/Preliminaries}

\input{body/Tamed_schemes_main_results}

\input{body/Framework}

%\input{body/Cross-regurality for RMLMC}

%----------------------------------------------------------------------
\input{Lemmas}

%----------------------------------------------------------------------
\input{Proofs}

\bibliographystyle{plainnat}
\bibliography{references} 
\end{document}

%% file: body/Introduction.tex
\section{Introduction}

The numerical discretization of stochastic differential equations (SDEs) with super-linearly growing drift has attracted considerable attention in the numerical analysis of stochastic dynamics. Beyond its intrinsic interest, this analysis underpins a wide range of sampling algorithms used in computational statistics, Bayesian inference, and large-scale optimization. The general SDE of interest in this work takes the form
\begin{equation}
  dX_t = -h(X_t)\,dt + \sigma\,dB_t,
  \tag{SDE}\label{eq:SDE}
\end{equation}
where $h : \mathbb{R}^d \to \mathbb{R}^d$ is a locally Lipschitz drift that may grow super-linearly at infinity, $\sigma > 0$ is a constant volatility, and $(B_t)_{t\ge 0}$ is a canonical Wiener process in $\mathbb{R}^d$ with unit covariance. A canonical application consists of sampling from a target distribution $\pi$ on $\mathbb{R}^d$, typically expressed in Gibbs form as $\pi(x) \propto e^{-\beta u(x)}$ for some potential function $u : \mathbb{R}^d \to \mathbb{R}$. Under suitable assumptions on $u$, this distribution arises as the invariant measure of the Langevin SDE
\begin{equation}
  dX_t = -\nabla u(X_t)\,dt + \sqrt{\tfrac{2}{\beta}}\,dB_t,
  \tag{LSDE}\label{eq:LSDE}
\end{equation}
which is precisely \eqref{eq:SDE} specialized to $h := \nabla u$ and $\sigma := \sqrt{2/\beta}$.
Motivated by this fundamental property of \eqref{eq:LSDE}, one of the most widely used
algorithmic schemes for sampling from $\pi$ is the
\emph{unadjusted Langevin algorithm} (ULA), given recursively by
\begin{equation}
  \theta^{\mathrm{ULA}}_{0} := \theta_0, 
  \qquad
  \theta^{\mathrm{ULA}}_{n+1}
  =
  \theta^{\mathrm{ULA}}_{n}
  - \lambda\, h\!\left(\theta^{\mathrm{ULA}}_{n}\right)
  + \sqrt{\tfrac{2\lambda}{\beta}}\,\xi_{n+1},
  \qquad n\in\mathbb{N}_0,
  \tag{ULA}\label{eq:ULA}
\end{equation}
where $\theta_n \in \mathbb{R}^d$, $n=1,2,\ldots$, denotes the state variable of the algorithm,
$(\xi_n)_{n\ge1}$ is an independent and identically distributed (i.i.d.) sequence of
standard $d$-dimensional random vectors with unit covariance,
$\lambda>0$ is a step-size parameter, and $h := \nabla u$ denotes the gradient of the
potential $u$.
The guiding idea behind \eqref{eq:ULA} is that it may be viewed as an Euler--Maruyama
discretization of \eqref{eq:LSDE}, so that, for sufficiently large $n$ and sufficiently small
$\lambda$, $\theta^{\mathrm{ULA}}_n$ is distributed approximately according to the invariant measure of \eqref{eq:LSDE}, which is precisely the target
distribution $\pi$.

This connection has motivated a substantial literature devoted to establishing
non-asymptotic convergence guarantees for \eqref{eq:ULA} in various probability metrics,
most notably Wasserstein distances, total variation, and information-theoretic divergences
such as the Kullback--Leibler (KL) divergence.
Much of this literature focuses on the case where $\pi$ is log-concave and the gradient
$\nabla u$ is globally Lipschitz, corresponding respectively to convexity and smoothness
of the potential $u$; see, e.g.,
\cite{dalalyan,unadjusted,aew} and references therein.

Beyond the globally Lipschitz log-concave setting, a substantial body of work has sought to relax convexity and smoothness assumptions on \(u\), notably through dissipativity, convexity-at-infinity, and related structural conditions; see, for instance,
\cite{dalalyan2016theoreticalguaranteesapproximatesampling,majka2018non,chewi2024analysislangevinmontecarlo,raginsky}.
A complementary line of research, initiated by the insights of Vempala and Wibisono \cite{vempala2019rapid}, establishes convergence under isoperimetric inequalities, such as Poincaré and logarithmic Sobolev inequalities; see
\cite{erdogdu2021convergence,erdogdu2022convergence,chewi2024analysislangevinmontecarlo,mousavi2023towards}.

This work examines the numerical discretization of \eqref{eq:SDE} in a regime that lies at the intersection of two well-studied directions: relaxing global Lipschitz assumptions on the drift, and---in the sampling specialization---weakening convexity via isoperimetric inequalities. While each of these settings has been analyzed extensively in isolation, rigorous non-asymptotic guarantees in their combination---particularly in the presence of super-linear drift growth---remain limited in the literature. Crucially, the numerical-analytic obstacles arising from super-linear drift are not specific to the gradient case $h = \nabla u$: explosion of moments, divergence of Euler--Maruyama in mean-square sense, and breakdown of standard local-error frameworks all occur at the level of \eqref{eq:SDE}, prompting a treatment that is general at the level of the SDE and only later specialized to sampling. This motivates the following two-fold question:
\begin{center}
\emph{
How does one derive finite-time error estimates for numerical discretizations of SDEs with super-linearly growing locally Lipschitz drift, and what guarantees can such estimates yield for sampling and optimization in the absence of log-concavity and global smoothness?
}
\end{center}

This scenario presents substantial challenges. As noted in several prior studies, both \eqref{eq:ULA}
and its stochastic variants, such as SGLD, may become unstable in such regimes. In particular, when $h$
grows super-linearly, the associated Euler--Maruyama scheme---which forms the backbone of
\eqref{eq:ULA}---can fail dramatically. A pivotal result in~\cite{erdogdu2022convergence} showed that, in such cases,
the numerical approximation may diverge from the true SDE solution in mean-square sense, even over
finite time horizons. This phenomenon is directly linked to the explosion of moments in the discretized
process, highlighting why standard schemes may break down when applied to approximate \eqref{eq:SDE}
with super-linearly growing drift $h$.

Addressing this issue requires a fundamentally different approach---one that revisits the way in which
\eqref{eq:ULA} is designed as a numerical discretization of \eqref{eq:LSDE}, and leverages insights from the
theory of numerical methods for SDEs. In this vein, a promising class of techniques---known as tamed Euler
schemes---emerged in~\cite{hutzenthaler2012,tamed-euler,SabanisAoAP}. These methods modify the drift $h$ in order to ensure
the stability of \eqref{eq:ULA} even under super-linear growth. Specifically, they replace the original
coefficient $h$ with a modified version $h_\lambda$, depending on the stepsize $\lambda$, which is constructed
to satisfy two essential properties:
\begin{itemize}
\item[(P1)] The tamed coefficient $h_\lambda$ grows at most linearly, i.e.,
$h_\lambda(\theta)=O(\|\theta\|)$ as $\|\theta\|\to\infty$.
\item[(P2)] $h_\lambda$ converges pointwise to the original coefficient $h$ as $\lambda\to0$.
\end{itemize}
These properties ensure both the stability of the numerical scheme and consistency with the original
dynamics as the stepsize vanishes.
\par

\medskip
\textbf{Tamed Euler algorithms.}
In the context of Langevin-based sampling, taming techniques have been applied under strong
dissipativity or convexity assumptions \cite{tula,johnston2023kinetic,lytras2025contr}, in
stochastic-gradient settings \cite{TUSLA,lim2021non}, under a ``convexity at infinity''
assumption \cite{neufeld2024nonasymptoticconvergenceboundsmodified}, and more recently under functional inequalities
\cite{lytras2024tamed,lytras2025taming,lytras2025ktula}. In parallel, truncated or
projected schemes have been examined under logarithmic Sobolev assumptions
\cite{yang2025non}.

In this work, we build on the KL-accelerated tamed unadjusted Langevin algorithm (kTULA) introduced in~\cite{lytras2025ktula}, which can be viewed as a modification of TULA designed to ensure stability under super-linear drift growth. The key idea underlying our analysis is that taming induces state-dependent perturbations in the dynamics, which fall outside the scope of existing KL frameworks. By explicitly controlling these perturbations, we establish non-asymptotic finite-time error estimates for \ref{eq:kTULA} in KL divergence under polynomial smoothness and dissipativity of the drift; when specialized to the sampling setting $h = \nabla u$ under a logarithmic Sobolev inequality, this yields improved convergence rates toward $\pi_\beta$ even in the presence of super-linear gradient growth.

\medskip
\textbf{Tamed randomized Euler algorithms.}
At the same time, an alternative line of work has focused on randomized discretization schemes for Langevin dynamics, where randomness is introduced directly at the level of the numerical integrator, for instance through randomized time steps or intermediate evaluations of the drift. Such methods have been shown to improve stability properties, reduce discretization bias, and, in certain regimes, achieve improved convergence rates compared to their deterministic counterparts.

A representative example, in the Langevin case $h = \nabla u$, is the randomized midpoint Langevin algorithm, given by
\[
\begin{aligned}
Y_{n+1}^\tau
&=
Y_n-\tau_{n+1}\nabla u(Y_n)\,\lambda+\sqrt{2}\,\Delta W_{n+1}^\tau,
\qquad Y_0=x_0,\\
Y_{n+1}
&=
Y_n-\nabla u(Y_{n+1}^\tau)\,\lambda+\sqrt{2}\,\Delta W_{n+1},
\qquad n\in\mathbb{N}_0,
\end{aligned}
\]
where $(\tau_n)_{n\in\mathbb{N}}$ are i.i.d.\ uniform random variables on $[0,1]$.

Originally introduced as numerical methods for SDEs with irregular coefficients, such randomized schemes were later adapted to sampling settings in \cite{shen2019randomized} and further developed in
\cite{cao2020complexity,yu2023langevin,chewi2024local,li2025convergence}.
In the context of sampling, randomized discretizations are particularly appealing, as they can better capture the underlying stochastic dynamics while mitigating the accumulation of discretization error.

However, existing analyses of these methods rely crucially on global Lipschitz assumptions on the gradient. In particular, it remains unclear whether the stability and bias-reduction properties of randomized schemes persist when the drift exhibits super-linear growth, as is typical in non-log-concave and high-dimensional learning problems.

Motivated by the developments in both the randomized Euler and tamed Euler literature, it is natural to investigate whether the stabilizing effect of taming can be combined with the robustness properties of randomized discretizations. This raises the question of whether one can construct discretization schemes that remain stable under super-linear drift growth while retaining the favorable bias properties of randomized methods.

In this work, we address this question by introducing and analyzing a tamed randomized Langevin Monte Carlo algorithm (tRLMC). We establish non-asymptotic finite-time error estimates for this method in total variation against the underlying SDE, and, in the sampling specialization, derive non-asymptotic convergence guarantees in non-log-concave regimes, including explicit bounds in Wasserstein distance and total variation, with further implications for optimization.

\medskip
\textbf{Our contributions.} Our results are organized in two layers: finite-time error estimates against the underlying \eqref{eq:SDE}, valid for general locally Lipschitz drift, and their specialization to the sampling problem under $h = \nabla u$ and a logarithmic Sobolev inequality. The two finite-time frameworks address complementary settings: the KL framework applies to the deterministic tamed Euler scheme, where a sharp cross-regularity estimate is available, whereas the TV framework applies to the randomized scheme, for which such a cross-regularity estimate is currently out of reach due to the randomization of the integration time.
\begin{enumerate}
    \item[(i)] \textbf{Finite-time KL error estimates for SDEs with locally Lipschitz drift.}
    Building on \cite{chewi2024local}, we develop a new local-error framework for controlling the KL divergence between a discretization scheme and the \eqref{eq:SDE} when the drift grows super-linearly. Applied to \eqref{eq:kTULA}, the framework yields a finite-time estimate of the form
    \[
        \KL\!\bigl(\mu\widehat{P}^{N}\,\big\|\,\nu P^{N}\bigr)
        \;\lesssim\;
        \tfrac{1}{\lambda}\,\mathcal{W}_2^2(\mu,\nu)
        \;+\;\bigl((N\lambda)\vee\log N\bigr)\,\lambda^{2},
    \]
    valid for general locally Lipschitz drift; see Theorem~\ref{thm:kl-local-kTULA}.

    \item[(ii)] \textbf{Finite-time TV error estimates for randomized schemes under super-linear drift.}
    By combining the taming methodology with a randomized midpoint method, we introduce a new discretization, termed \ref{eq:tamed-rLMC}, and develop a corresponding total variation local-error framework. This yields the first finite-time error estimate in total variation for a randomized scheme applied to SDEs with super-linearly growing locally Lipschitz drift; see Theorem~\ref{thm:tv-local-tRLMC}.

    \item[(iii)] \textbf{Near-optimal sampling guarantees for tamed Euler schemes in $\KL$ divergence.}
    Specializing the KL framework to the Langevin case $h = \nabla u$ under a logarithmic Sobolev inequality, we derive an explicit $\tilde{O}(\epsilon^{-\frac{1}{2}})$ complexity bound for \ref{eq:kTULA} to reach accuracy $\epsilon$ in KL divergence toward $\pi_\beta$. This further yields improved guarantees in Wasserstein and total variation distances.

    \item[(iv)] \textbf{Sampling guarantees for tamed randomized schemes.}
    Specializing the TV framework to the Langevin case, we show that \ref{eq:tamed-rLMC} achieves $\tilde{O}(\epsilon^{-1})$ complexity for sampling toward $\pi_\beta$ in $W_2$ and total variation distance under a logarithmic Sobolev inequality.

    \item[(v)] \textbf{Connection with optimization.}
    We apply our Wasserstein convergence results to obtain explicit non-asymptotic excess-risk bounds for tamed Langevin schemes in the non-log-concave setting. More precisely, we decompose the excess risk as:
    \[
        \underbrace{\E\!\bigl[u(\theta_n)\bigr]-\E_{\pi_\beta}\!\bigl[u(x)\bigr]}_{\text{sampling error}}
        +
        \underbrace{\E_{\pi_\beta}\!\bigl[u(x)\bigr]- u^*}_{\text{concentration for large }\beta},
    \]
    where $\theta_n$ denotes the $n$-th iterate of the algorithm. Thus, a solution to the sampling problem yields corresponding optimization guarantees for sufficiently large $\beta$.
\end{enumerate}

\medskip
\textbf{Organization of the paper.}
The remainder of the paper is organized as follows. Section~\ref{sec:preliminaries} introduces the necessary notation, the structural assumptions on the drift of \eqref{eq:SDE}, and the functional inequalities used in the sampling specialization. Section~\ref{sec:main-results} introduces the two tamed schemes \eqref{eq:kTULA} and \eqref{eq:tamed-rLMC} and presents the main finite-time error estimates against the underlying SDE in KL divergence and total variation, together with their specialization to sampling in KL, total variation, and Wasserstein distance, and the resulting implications for optimization. Section~\ref{sec:numerical-experiments} reports numerical experiments illustrating the stabilizing effect of taming both for sampling and for a nonlinear optimization problem. Section~\ref{sec:framework} presents the new local-error frameworks underlying our analysis, including the KL framework for \ref{eq:kTULA}, the corresponding TV framework for the randomized scheme, and the Wasserstein convergence argument for \ref{eq:tamed-rLMC}. Detailed proofs of all auxiliary results and of the main theorems are deferred to Appendices~\ref{app:drift-moments}-\ref{sec:proof-section}.

%% file: body/Preliminaries.tex
\section{Preliminaries and Blanket Assumptions}
\label{sec:preliminaries}

In this section, we provide the necessary groundwork for stating the proposed algorithmic
schemes and our main results. We adopt throughout a two-layer convention: structural assumptions and notation are formulated at the level of the general \eqref{eq:SDE} with locally Lipschitz drift $h$, while functional inequalities and the target measure $\pi_\beta$ are introduced only in the sampling specialization $h = \nabla u$.

\subsection{Notation}
\label{subsec:math-prelims}

Throughout, $\|\cdot\|$ denotes the Euclidean norm on $\mathbb{R}^d$ and 
$\langle\cdot,\cdot\rangle$ its associated inner product. For a sufficiently smooth function $f:\R^d\to\R$, we write $\nabla f$, $\nabla^2 f$ and $\Delta f$
for its gradient, Hessian matrix, and Laplacian respectively. For a continuously differentiable map $F : \R^d \to \R^d$, we write $J(F)$ for its Jacobian. In the gradient case $F = \nabla u$, one has $J(\nabla u) = \nabla^2 u$.
For a measurable map $T$, we write $\Law(T)$ for its induced distribution.
For a probability measure $\mu$ and Markov kernel $P$, the pushforward measure $\mu P$ is
\[
  (\mu P)(A) := \int_{\mathbb{R}^d} P(x,A)\,\mu(dx),
  \qquad 
  \mu P^n := \mu \underbrace{P\cdots P}_{n\ \text{times}}.
\]

\medskip
\noindent\textbf{SDE semigroup and one-step kernel.}
Let $(P_t)_{t\ge 0}$ denote the Markov semigroup of~\eqref{eq:SDE}, i.e.,
$P_t(x,\cdot)=\Law(X_t^x)$ where $X_0^x=x$.
Fix a step size $\lambda>0$ and define the one-step diffusion kernel
\[
  P := P_\lambda.
\]
In the sampling specialization $h = \nabla u$, $(P_t)_{t\ge 0}$ coincides with the Langevin semigroup of~\eqref{eq:LSDE}.

\medskip
\noindent\textbf{2-Wasserstein distance.}
Let $\mathcal{P}_2(\R^d)$ denote the space of probability measures with finite second moment.
For $\mu,\nu\in\mathcal{P}_2(\R^d)$, the $2$-Wasserstein distance is defined by
\begin{equation}\label{eq:w2}
  \W(\mu,\nu)
  :=
  \left(
    \inf_{\gamma\in\mathcal{C}(\mu,\nu)}
    \int_{\R^d\times\R^d} \|x-y\|^2\,\gamma(dx,dy)
  \right)^{1/2},
\end{equation}
where $\mathcal{C}(\mu,\nu)$ denotes the set of couplings of $\mu$ and $\nu$, that is, probability measures on $\R^d\times\R^d$ with marginals $\mu$ and $\nu$. Equivalently, $(X,Y)$ is a coupling of $\mu$ and $\nu$ if $\Law(X)=\mu$ and $\Law(Y)=\nu$.

\medskip
\noindent\textbf{Kullback--Leibler divergence.}
Let $\mathcal{P}(\R^d)$ denote the space of probability measures on $\R^d$.
For $\mu,\nu\in\mathcal{P}(\R^d)$, the Kullback--Leibler divergence is defined by
\begin{equation}\label{eq:kl-def}
 \KL(\mu\|\nu)
 :=
 \int_{\R^d} \log\!\left(\frac{d\mu}{d\nu}\right)\,d\mu,
\end{equation}
whenever $\mu$ is absolutely continuous with respect to $\nu$, and $\KL(\mu\|\nu):=+\infty$ otherwise.

\medskip
\noindent
We rely on the standard chain rule and data processing inequality for relative entropy, which yield the marginal bound \(\KL(\mu^Y\|\nu^Y) \le \KL(\mu^{X,Y}\|\nu^{X,Y})\). A key refinement for our purposes is the \emph{shifted chain rule} \cite{chewi2024local}. By introducing an auxiliary variable \(X'\), it provides the flexible bound
\begin{equation}\label{eq:shifted-chain-rule}
 \KL(\mu^Y\|\nu^Y)
 \le
 \KL(\mu^{X'}\|\nu^X)
 +
 \inf_{\gamma\in\mathcal{C}(\mu^X,\mu^{X'})}
 \int
 \KL\!\bigl(\mu^{Y|X=x}\,\big\|\,\nu^{Y|X=x'}\bigr)\,\gamma(dx,dx').
\end{equation}
This freedom to couple \(X\) and \(X'\) is central to the shifted comparison arguments developed in Section~\ref{sec:framework}.

\medskip
\noindent\textbf{Rényi divergence.}
The Rényi divergence of order $q > 1$ between two probability measures $\mu,\nu$ is defined by
\begin{equation}\label{eq:renyi-def}
 \mathsf{R}_q(\mu\|\nu)
 :=
 \frac{1}{q-1} \log \int \left(\frac{d\mu}{d\nu}\right)^q\,d\nu,
\end{equation}
whenever $\mu \ll \nu$, and $+\infty$ otherwise. In the limit as $q \to 1$, this recovers the Kullback--Leibler divergence, while the case $q=2$ yields $\mathsf{R}_2 = \log(1 + \chi^2)$, where $\chi^2$ is the chi-squared divergence.

\medskip
\noindent\textbf{Infinitesimal generator.}
We denote by $\mathcal{L}$ the infinitesimal generator of~\eqref{eq:SDE},
\[
  \mathcal{L}f \;:=\; \tfrac{\sigma^{2}}{2}\,\Delta f \;-\; \langle h,\nabla f\rangle,
\]
which, in the sampling specialization $h =\nabla u$ and $\sigma =\sqrt{\frac{2}{\beta}}$, reduces to the Langevin generator $\mathcal{L}f = \tfrac{1}{\beta}\Delta f - \langle \nabla u,\nabla f\rangle$.

\subsection{Assumptions on the drift}
\label{subsec:assumptions}

We impose the following structural conditions on the drift $h$ of~\eqref{eq:SDE}.

\begin{enumerate}[label=(A\arabic*), leftmargin=*]

\item \textbf{Polynomial Jacobian growth.}\label{ass:PJG}
There exist constants $L>0$ and $\ell\ge0$ such that
\[
  \max\bigl\{\|h(x)\|,\ \|J(h)(x)\|\bigr\}
  \le L\bigl(1+\|x\|^{2\ell}\bigr),
  \qquad \forall x\in\mathbb{R}^d.
  \tag*{\textup{(PJG)}}
\]

\item \textbf{Polynomial Lipschitz continuity.}\label{ass:PLC}
There exist constants $L'>0$ and $\ell'\ge0$ such that
\[
  \|h(x)-h(y)\|
  \le
  L'\bigl(1+\|x\|+\|y\|\bigr)^{\ell'}\|x-y\|,
  \qquad \forall x,y\in\mathbb{R}^d.
  \tag*{\textup{(PLC)}}
\]

\item \textbf{One-sided Lipschitz continuity.}\label{ass:OSL}
There exists a constant $K'\in\mathbb{R}$ such that
\[
  \langle h(x)-h(y),\,x-y\rangle
  \ge -K'\|x-y\|^2,
  \qquad \forall x,y\in\mathbb{R}^d.
  \tag*{\textup{(OSL)}}
\]

\item \textbf{Dissipativity.}\label{ass:D}
There exist constants $a,b>0$ such that
\[
  \langle h(x),x\rangle \ge a\|x\|^2 - b,
  \qquad \forall x\in\mathbb{R}^d.
  \tag*{\textup{(D)}}
\]
\item \textbf{Polynomial Jacobian Lipschitz continuity.}\label{ass:JLC}
There exist constants $L''>0$ and $\ell''\ge0$ such that
\[
  \|J(h)(x)-J(h)(y)\|
  \le
  L''(1+\|x\|+\|y\|)^{\ell''}\|x-y\|,
  \qquad \forall x,y\in\mathbb{R}^d.
  \tag*{\textup{(JLC)}}
\]
\end{enumerate}

Assumption~\ref{ass:PJG} allows the drift and its Jacobian to grow polynomially at infinity, thereby encompassing a wide class of drift coefficients with super-linear growth that fall outside the scope of the standard global Lipschitz framework. In the sampling specialization $h = \nabla u$, it covers non-convex potentials for which classical convergence analyses of \eqref{eq:ULA} are not directly applicable.

Assumption~\ref{ass:PLC} ensures that the drift is locally Lipschitz continuous,
with a Lipschitz modulus that may itself grow polynomially.
This condition is substantially weaker than global Lipschitz continuity
and is compatible with super-linear drifts,
while still being sufficient to establish well-posedness of~\eqref{eq:SDE}
and stability of the numerical schemes considered in this work.

Assumption~\ref{ass:OSL} is a one-sided Lipschitz condition, which plays a key role in controlling
the propagation of errors between coupled trajectories of~\eqref{eq:SDE}.
It is strictly weaker than monotonicity or strong convexity,
and is particularly well suited to the analysis of non-convex dynamics.

Assumption~\ref{ass:D} is a dissipativity condition that ensures coercive behavior of the drift
at infinity.
This assumption is central to our analysis, as it allows us to derive uniform-in-time
moment bounds for both~\eqref{eq:SDE} and its discretizations, and is a natural coercivity-type condition in applications.

Assumption~\ref{ass:JLC} is imposed only in the analysis of the \eqref{eq:kTULA}
scheme. It provides the additional Jacobian regularity---namely, local Lipschitz continuity of $J(h)$ with polynomial modulus---required to obtain the improved weak error rate underlying our KL convergence bounds. In the gradient case $h = \nabla u$, since $J(h) = \nabla^2 u$, this reduces to a polynomial-Lipschitz Hessian condition on the potential. The assumption is not needed for the randomized scheme \eqref{eq:tamed-rLMC}, whose
bias-reduction mechanism does not rely on second-order smoothness of the drift.

\subsection{Sampling and functional inequalities}
\label{subsec:functional-ineq}

We now turn to the specialization of~\eqref{eq:SDE} that underlies the sampling application. Throughout this subsection and whenever sampling guarantees are stated, we assume that $h = \nabla u$ for some $u : \R^d \to \R$ satisfying $\int_{\R^d} e^{-\beta u(x)}\,dx < \infty$ for any $\beta > 0$, and we set $\sigma = \sqrt{\frac{2}{\beta}}$. Under this specialization, \eqref{eq:SDE} reduces to~\eqref{eq:LSDE}, which admits as invariant measure the Gibbs distribution
\[
  \pi_\beta(A) \;:=\; \frac{\int_A e^{-\beta u(x)}\,dx}{\int_{\R^d} e^{-\beta u(x)}\,dx},
  \qquad A \in \mathcal{B}(\R^d).
\]
In addition to the structural assumptions \hyperref[ass:PJG]{(A1)}--\hyperref[ass:JLC]{(A5)} on $h = \nabla u$, we impose a functional inequality on $\pi_\beta$.

\begin{enumerate}[label=(A\arabic*), leftmargin=*, resume]
\item \textbf{Logarithmic Sobolev inequality (LSI).}\label{ass:LSI}
The target distribution $\pi_\beta$ satisfies a logarithmic Sobolev inequality with constant $C_{\mathrm{LSI}}>0$,
meaning that for every probability measure $\nu\ll\pi_\beta$ with density $f:=\frac{d\nu}{d\pi_\beta}$,
\begin{equation}
\tag{LSI}
\label{eq:LSI}
 \KL(\nu\|\pi_\beta)
 =
 \int_{\R^d} f\log f\,d\pi_\beta
 \le
 \frac{1}{2C_{\mathrm{LSI}}}\,I_{\pi_\beta}(\nu),
\end{equation}
\noindent where $I_{\pi_\beta}(\nu)$ denotes the \emph{relative Fisher information} of $\nu$ with respect to $\pi_\beta$.
\end{enumerate}

\medskip
Assumption~\ref{ass:LSI} is widely used in the analysis of Langevin dynamics since it implies exponential
ergodicity in relative entropy. In particular, by the Bakry--\'Emery theorem \cite{bakry2006diffusions}, \eqref{eq:LSI} holds for
strongly convex potentials and is stable under bounded perturbations, Lipschitz mappings, and
convolutions.

\medskip
\noindent
Finally, \eqref{eq:LSI} implies Talagrand's transportation-cost inequality: for every $\nu\in\mathcal P_2(\R^d)$,
\begin{equation}\tag{TI}\label{eq:TI}
  \W(\nu,\pi_\beta)
  \;\le\;
  \sqrt{\frac{2}{C_{\mathrm{LSI}}}\,\KL(\nu\|\pi_\beta)}.
\end{equation}

%% file: body/Tamed_schemes_main_results.tex
\section{Tamed schemes and main results}
\label{sec:main-results}

This section introduces the two tamed discretization schemes analyzed in this work---\eqref{eq:kTULA} and \eqref{eq:tamed-rLMC}---and presents the corresponding main results. The results in this section are organized into two distinct layers. First, we establish finite-time error estimates against the underlying ~\eqref{eq:SDE}, which rely solely on the structural assumptions \hyperref[ass:PJG]{(A1)}--\hyperref[ass:JLC]{(A5)} for the general drift $h$. Second, we derive sampling and optimization guarantees toward $\pi_\beta$, which additionally invoke the gradient specialization $h=\nabla u$ and Assumption~\ref{ass:LSI}.

\subsection{Tamed discretizations}
\label{subsec:tamed-discretizations}

Throughout this subsection, we fix constants $a>0$ and $\ell\ge 0$, and define, for each $\lambda>0$, the tamed drift $h_\lambda:\R^d\to\R^d$ associated with $h$ by
\begin{equation}
h_\lambda(x)
:=
a x
+
\frac{h(x)-a x}
{\bigl(1+\lambda\|x\|^{2(\ell+1)}\bigr)^{1/2}}.
\label{eq:tamed-drift}
\end{equation}
The construction~\eqref{eq:tamed-drift} is well defined for any locally Lipschitz drift $h$ and does not require $h$ to be a gradient field.

\medskip
\noindent\textbf{The kTULA scheme.}
The KL-accelerated tamed unadjusted Langevin algorithm is given by
\begin{equation}
\widehat X_0^\lambda := X_0,
\qquad
\widehat X_{n+1}^\lambda
=
\widehat X_n^\lambda
-
\lambda\, h_\lambda(\widehat X_n^\lambda)
+
\sqrt{\frac{2\lambda}{\beta}}\,\xi_{n+1},
\qquad n\in\mathbb{N}_0,
\tag{kTULA}\label{eq:kTULA}
\end{equation}
where $X_0$ is an $\R^d$-valued random variable, $\lambda>0$ is the stepsize, $\beta>0$ is the inverse temperature, and
$(\xi_n)_{n\in\mathbb{N}_0}$ are i.i.d.\ standard $d$-dimensional Gaussian random vectors. We denote by $\pi_n^\lambda$ the density of $\Law(\widehat X_n^\lambda)$ for all $n\in\mathbb{N}_0$.

\begin{remark}
The design of the tamed coefficient $h_\lambda$ in~\eqref{eq:tamed-drift} follows from that of
mTULA~\cite{neufeld2024nonasymptoticconvergenceboundsmodified} and sTULA~\cite{lytras2025taming}. It allows us to derive several properties of $h_\lambda$, which are crucial to establish moment estimates and convergence results of
\eqref{eq:kTULA}. More precisely, by adopting the splitting trick originally used
in~\cite{lytras2025taming}, the tamed coefficient $h_\lambda$ satisfies a dissipativity condition, which enables contraction of the algorithm~\eqref{eq:kTULA} and therefore an
easier computation of the associated moment bounds. Moreover, dividing by the term
$\bigl(1+\lambda\|x\|^{2(\ell+1)}\bigr)^{1/2}$ in~\eqref{eq:tamed-drift} yields an
improved upper bound for the difference between $h_\lambda$ and $h$, compared
to~\cite{lytras2025taming}.
\end{remark}

\medskip
\noindent\textbf{Main results.}
Denote by
\begin{equation}\label{eq:lambda-max}
\lambda_{\max}^{\mathrm{kTULA}}
:=
\min\left\{
1,\;
\frac{1}{8a},\;
\frac{1}{(6L_0)^2}
\right\},
\end{equation}
where $L_0 := 2a + 4L + (\ell+1)(2L+a)$.

\medskip
\noindent
Throughout this subsection, for a fixed step size $\lambda>0$, we write $P:=P_\lambda$ for the one-step kernel of~\eqref{eq:SDE} run for time $\lambda$, and $\widehat P$ for the one-step kernel associated with \eqref{eq:kTULA}.

\subsubsection{Finite-time KL error estimate against the underlying SDE}
\label{subsubsec:ktula-general}

Our first result is a finite-time KL error estimate between \eqref{eq:kTULA} and~\eqref{eq:SDE}, valid for general locally Lipschitz drift.

\begin{theorem}[KL local-error bound for kTULA]
\label{thm:kl-local-kTULA}
Assume \hyperref[ass:PJG]{\textup{(A1)}}--\hyperref[ass:JLC]{\textup{(A5)}}.
Let $\lambda\in(0,\lambda_{\max}^{\mathrm{kTULA}}]$.
Then there exist constants $C_1,C>0$, depending at most polynomially on the dimension and independent of $N$ and $\lambda$, such that for all $\mu,\nu\in\mathcal P_2(\R^d)$ and all $N\in\mathbb N$ with $N\ge3$,
\begin{equation}\label{eq:kl-ktula-final}
\KL\!\bigl(\mu\,\widehat P^{N}\,\big\|\,\nu\,P^{N}\bigr)
\;\le\;
C_1\,\frac{1}{\lambda}\,\W^2(\mu,\nu)
\;+\;
C\,\bigl((N\lambda)\vee \log N\bigr)\lambda^2.
\end{equation}
\end{theorem}

\begin{remark}\label{rem:kl-kTULA-same-init}
In particular, for $\mu=\nu$ and $N\ge 3$, the Wasserstein term vanishes, and
\[
\KL\!\bigl(\mu\,\widehat P^{N}\,\big\|\,\mu\,P^{N}\bigr)
\;\le\;
C\,\bigl((N\lambda)\vee \log N\bigr)\lambda^2.
\]
\end{remark}

\subsubsection{Sampling specialization}
\label{subsubsec:ktula-sampling}

We now specialize Theorem~\ref{thm:kl-local-kTULA} to the sampling setting introduced in Section~\ref{subsec:functional-ineq}. From here on through the end of this subsection, we assume $h = \nabla u$, so that~\eqref{eq:LSDE} admits $\pi_\beta$ as its invariant measure, and we additionally impose the logarithmic Sobolev inequality~\ref{ass:LSI}. Under this specialization, the kernel $P$ becomes the one-step Langevin diffusion kernel of~\eqref{eq:LSDE}, and Theorem~\ref{thm:kl-local-kTULA} combined with the exponential ergodicity of $(P_t)_{t\ge 0}$ yields sampling guarantees toward $\pi_\beta$.

\begin{corollary}[KL sampling guarantee for kTULA under LSI]
\label{prop:kl-sampling-ktula}
Assume \hyperref[ass:PJG]{\textup{(A1)}}--\hyperref[ass:LSI]{\textup{(A6)}}, and let $\lambda\in(0,\lambda_{\max}^{\mathrm{kTULA}}]$.
Then there exist constants $C_1,C>0$ such that, for all $\mu\in\mathcal P_2(\R^d)$ and all $N\ge3$,
\begin{equation}\label{eq:kl-ktula-to-pi-general}
\KL\!\bigl(\mu\widehat P^N\,\big\|\,\pi_\beta\bigr)
\le
\KL\!\bigl(\mu\widehat P^N\,\big\|\,\mu P^N\bigr)
+
R_2\!\bigl(\mu P^N\,\big\|\,\pi_\beta\bigr),
\end{equation}
where $R_2(\cdot\|\cdot)$ denotes the R\'enyi divergence of order $2$. Consequently,
\begin{equation}\label{eq:kl-ktula-to-pi-sameinit}
\KL\!\bigl(\mu\widehat P^N\,\big\|\,\pi_\beta\bigr)
\le
C\,\bigl((N\lambda)\vee\log N\bigr)\lambda^2
+
e^{-2C_{\mathrm{LSI}}N\lambda}\,R_2\!\bigl(\mu\,\big\|\,\pi_\beta\bigr).
\end{equation}
\end{corollary}

\medskip
\noindent
We next deduce explicit mixing-time guarantees from this bound.

\begin{proposition}[Mixing time in KL for kTULA under LSI]
\label{prop:mixing-time-kl-ktula}
Assume \hyperref[ass:PJG]{\textup{(A1)}}--\hyperref[ass:LSI]{\textup{(A6)}}, and let $\lambda\in(0,\lambda_{\max}^{\mathrm{kTULA}}]$.
Fix $\mu\in\mathcal P_2(\R^d)$ and assume $R_0:=R_2(\mu\|\pi_\beta)<\infty$.
Then
\[
\KL\!\bigl(\mu\,\widehat P^{N}\,\big\|\,\pi_\beta\bigr)
\;\le\;
C\,\bigl((N\lambda)\vee \log N\bigr)\lambda^2
\;+\;
e^{-2C_{\mathrm{LSI}}N\lambda}\,R_0.
\]
In particular, for $\lambda \le \widetilde O(\sqrt{\varepsilon})$, it suffices to take $N \ge \widetilde{O}\!\left(\frac{1}{\sqrt{\varepsilon}}\right)$
to ensure
\[
\KL\!\bigl(\mu\,\widehat P^{N}\,\big\|\,\pi_\beta\bigr)\le \varepsilon.
\]
\end{proposition}

\medskip
\noindent
The KL bound further implies convergence in stronger metrics.

\begin{corollary}[TV mixing time via Pinsker]
\label{cor:tv-mixing-ktula}
Assume the conditions of Corollary~\ref{prop:kl-sampling-ktula} and let $\lambda\in(0,\lambda_{\max}^{\mathrm{kTULA}}]$. Fix $\mu\in\mathcal P_2(\R^d)$ and assume $R_0:=R_2(\mu\|\pi_\beta)<\infty$.
Then
\[
TV(\mu\widehat P^{N},\pi_\beta)
\;\le\;
C\,\lambda \sqrt{(N\lambda)\vee \log N}
\;+\;
e^{-C_{\mathrm{LSI}}N\lambda}\sqrt{R_0}.
\]
As a consequence, for $\lambda \le \widetilde O(\varepsilon)$, it suffices to take $N \ge \widetilde O(1/\varepsilon)$ to ensure
\[
TV(\mu\widehat P^{N},\pi_\beta)\le \varepsilon.
\]
\end{corollary}

\begin{corollary}[$W_2$ mixing time via Talagrand]
\label{cor:w2-mixing-ktula}
Assume the conditions of Corollary~\ref{prop:kl-sampling-ktula} and let $\lambda\in(0,\lambda_{\max}^{\mathrm{kTULA}}]$. Fix $\mu\in\mathcal P_2(\R^d)$ and assume $R_0:=R_2(\mu\|\pi_\beta)<\infty$.
Then there exist constants $C_{W_2},c_0>0$, where $C_{W_2}$ depends at most polynomially on the dimension and $c_0$ depends on $C_{\mathrm{LSI}}$, such that
\[
\W\!\bigl(\mu\widehat P^{N},\pi_\beta\bigr)
\;\le\;
C_{W_2}\!\left(
\lambda\sqrt{(N\lambda)\vee\log N}
+
e^{-c_0 N\lambda}
\right).
\]
As a consequence, for $\lambda \le \widetilde O(\varepsilon)$, it suffices to take $N \ge \widetilde O(1/\varepsilon)$ to guarantee
\[
\W\!\bigl(\mu\widehat P^{N},\pi_\beta\bigr)\le \varepsilon.
\]
\end{corollary}

\medskip
\noindent\textbf{Implications for optimization.}
When $\beta$ is sufficiently large, the Gibbs measure $\pi_\beta$ concentrates around the minimizers of $u$~\cite{hwang1980laplace}. As a result, \eqref{eq:kTULA} may also be used to address the optimization problem
\[
\operatorname*{minimize} \qquad \mathbb{R}^d \ni \theta \mapsto u(\theta).
\]

\begin{corollary}[Expected excess risk bound for kTULA]
\label{cor:excess-risk-ktula}
Assume \hyperref[ass:PJG]{\textup{(A1)}}--\hyperref[ass:LSI]{\textup{(A6)}}, and let $\lambda\in(0,\lambda_{\max}^{\mathrm{kTULA}}]$. Let $(\widehat X_n)_{n\ge0}$ be the \eqref{eq:kTULA} iterates. Fix $\mu:=\Law(\widehat X_0)$ such that $R_0:=R_2(\mu\|\pi_\beta)<\infty$.
Then there exist constants $C_0,C_1,C_2,C_3>0$, independent of $N$ and $\lambda$, such that for all $N\in\mathbb N$,
\begin{equation}\label{eq:excess-risk-ktula-statement}
\E\!\bigl[u(\widehat X_N)\bigr]-\inf_{\theta\in\R^d}u(\theta)
\;\le\;
C_1\,e^{-C_0N\lambda}
+
C_2\,\lambda\,\sqrt{(N\lambda)\vee\log N}
+
C_3\frac{\log \beta}{\beta}.
\end{equation}
\end{corollary}

\subsection{Tamed randomized discretization}
\label{subsec:tRLMC}

In addition to the fully explicit \eqref{eq:kTULA},
we also consider a randomized midpoint discretization equipped with our taming scheme.
This construction is designed to handle super-linearly growing drifts within
the randomized framework, while preserving stability and consistency
with the underlying ~\eqref{eq:SDE}.

Using this randomized midpoint construction, we define the
\emph{tamed randomized Langevin Monte Carlo} algorithm
as follows:
\begin{equation}
\label{eq:tamed-rLMC}
\tag{tRLMC}
\begin{aligned}
\bar Y_{n+1}^{\tau}
&=
\bar Y_n
-
\lambda\, h_\lambda(\bar Y_n)\,\tau_{n+1}
+
\sqrt{\tfrac{2}{\beta}}\,\Delta W_{n+1}^{\tau},
\qquad
\bar Y_0 = x_0,\\
\bar Y_{n+1}
&=
\bar Y_n
-
\lambda\, h_\lambda(\bar Y_{n+1}^{\tau})
+
\sqrt{\tfrac{2}{\beta}}\,\Delta W_{n+1},
\qquad n\in\mathbb{N}_0.
\end{aligned}
\end{equation}

Here, \((\tau_n)_{n\ge1}\) is an i.i.d.\ sequence of random variables uniformly distributed on \((0,1)\), independent of the Brownian motion \(W\), and
\[
\Delta W_{n+1}^{\tau}:=W_{t_n+\tau_{n+1}\lambda}-W_{t_n},
\qquad
\Delta W_{n+1}:=W_{t_{n+1}}-W_{t_n},
\]
where \(t_n=n\lambda\) for \(n\in\mathbb N_0\), and the tamed drift \(h_\lambda\) is defined by~\eqref{eq:tamed-drift}.

\medskip
\noindent\textbf{Main results.}
Denote by
\begin{equation}\label{eq:lambda-max-trlmc}
\lambda_{\max}^{\mathrm{tRLMC}}
:=
\min\Bigl\{
1,\;
\tfrac{1}{8a},\;
\tfrac{a}{2CL_0^2},\;
\bigl(\tfrac{pa}{4C_p}\bigr)^{2}
\Bigr\},
\end{equation}
where the constants are made explicit in the proof of Lemma~\ref{lem:moment-tamed-rLMC}.

\medskip
\noindent
Throughout this subsection, for a fixed step size \(\lambda>0\), we write \(P:=P_\lambda\) for the one-step kernel of~\eqref{eq:SDE} run for time $\lambda$, and \(\widetilde P\) for the one-step kernel associated with \eqref{eq:tamed-rLMC}.

\subsubsection{Finite-time TV error estimate against the underlying SDE}
\label{subsubsec:trlmc-general}

Our main finite-time result for \eqref{eq:tamed-rLMC} is a total-variation error estimate against~\eqref{eq:SDE}, valid for general locally Lipschitz drift.

\begin{theorem}[TV local-error bound for tRLMC]\label{thm:tv-local-tRLMC}
Assume \hyperref[ass:PJG]{\textup{(A1)}}--\hyperref[ass:D]{\textup{(A4)}} and let \(\lambda\in(0,\lambda_{\max}^{\mathrm{tRLMC}}]\). Then there exist constants \(C_1,C>0\), independent of \(N\) and \(\lambda\), such that for all \(\mu,\nu\in\mathcal P_2(\R^d)\) and all \(N\in\mathbb N\) with \(N\ge3\),
\begin{equation}
  TV\!\bigl(\mu\,\widetilde P^{N}\,\big\|\,\nu\,P^{N}\bigr)
  \;\le\;
  C_1\,\frac{1}{\sqrt{\lambda}}\,\W(\mu,\nu)
  \;+\;
  C\lambda \sqrt{\,\bigl((N\lambda)\vee \log N\bigr)} .
\end{equation}
\end{theorem}

Theorem~\ref{thm:tv-local-tRLMC} requires only \hyperref[ass:PJG]{(A1)}--\hyperref[ass:D]{(A4)}---in particular, it does not require the Jacobian regularity~\ref{ass:JLC} used in the KL framework---and is valid for general locally Lipschitz drift, with no gradient structure imposed on $h$. To the best of our knowledge, this is the first finite-time total-variation error estimate for a randomized discretization scheme applied to SDEs with super-linearly growing drift.

\subsubsection{Sampling specialization}
\label{subsubsec:trlmc-sampling}

We now specialize Theorem~\ref{thm:tv-local-tRLMC} to the sampling setting. From here on through the end of this subsection, we assume $h=\nabla u$, so that~\eqref{eq:LSDE} admits $\pi_\beta$ as its invariant measure, and we additionally impose Assumption~\ref{ass:LSI}.

\begin{corollary}[TV mixing time for tRLMC]
\label{cor:tv-mixing-trlmc}
Assume \hyperref[ass:PJG]{\textup{(A1)}}--\hyperref[ass:LSI]{\textup{(A6)}}, and
let \(\lambda\in(0,\lambda_{\max}^{\mathrm{tRLMC}}].\)
There holds
\[
TV(\mu_0 \widetilde P^N,\pi_\beta)
\;\le\;
C\lambda \sqrt{\,\bigl((N\lambda)\vee \log N\bigr)}
\;+\;
\sqrt{e^{-C_{\mathrm{LSI}}\lambda N} \KL(\mu_0\,\|\,\pi_\beta)}.
\]
As a result, for $\lambda\le \widetilde O({\varepsilon})$ one needs $N\ge \widetilde O\!\left(\frac{1}{\varepsilon}\right)$ to guarantee
\[
TV(\mu_0\widetilde P^N,\pi_\beta)\le \varepsilon.
\]
\end{corollary}

\medskip
\noindent
We also obtain convergence bounds in Wasserstein distance.

\begin{theorem}[Wasserstein convergence for tRLMC]
\label{thm:wass conv}
Assume \hyperref[ass:PJG]{\textup{(A1)}}--\hyperref[ass:LSI]{\textup{(A6)}}, and
let \(\lambda\in(0,\lambda_{\max}^{\mathrm{tRLMC}}].\)
There holds
\[
W_2(\mu_0\widetilde P^n,\pi_\beta)\;\le\; C_{W_2}\bigl(\lambda+e^{-c_0\lambda n}\bigr),
\]
where $C_{W_2}$ depends polynomially on the dimension and $c_0$ depends on $C_{\mathrm{LSI}}$.
As a result, for $\lambda\le \widetilde O({\varepsilon})$ one needs $N\ge \widetilde O\!\left(\frac{1}{\varepsilon}\right)$ to guarantee
\[
W_2(\mu_0\widetilde P^N,\pi_\beta)\le \varepsilon.
\]
\end{theorem}

\medskip
\noindent\textbf{Implications for optimization.}
As for \eqref{eq:kTULA}, we connect the Wasserstein convergence bound to the optimization problem $\min_{\theta\in\R^d} u(\theta)$ via the concentration of $\pi_\beta$ around the minimizers of $u$.

\begin{corollary}[Expected excess risk bound for tRLMC]
\label{cor:excess-risk-trlmc}
Assume \hyperref[ass:PJG]{\textup{(A1)}}--\hyperref[ass:LSI]{\textup{(A6)}}, and let $x_n$ be the $n$-th iterate of the tamed randomized Langevin Monte Carlo scheme~\eqref{eq:tamed-rLMC}.
There exist constants $C,C',c_0>0$, depending at most polynomially on the dimension, such that
\[
\E[u(x_n)]-u^*\;\le\; C\bigl(\lambda+e^{-c_0\lambda n}\bigr) + C'\frac{\log \beta}{\beta}.
\]
\end{corollary}

\subsection{Comparison with related literature}
\label{subsec:related-work}

We compare our results with the existing literature on Langevin-based sampling under non-convexity and local Lipschitz continuity of the gradient. We emphasize that Theorems~\ref{thm:kl-local-kTULA} and~\ref{thm:tv-local-tRLMC} are valid for general locally Lipschitz drift on~\eqref{eq:SDE} and are, to the best of our knowledge, the first finite-time error estimates of their kind under super-linear drift growth.

\medskip
\noindent\emph{Comparison in Wasserstein distance.}
Several recent works on taming and projection schemes provide convergence rates of $\mathcal{O}(\varepsilon^{-1})$ in Wasserstein distance under super-linear drift growth \cite{yang2025non,bao2024geometric}. Our analysis of \eqref{eq:kTULA} matches this rate under the same or weaker assumptions, and extends the analysis to KL divergence and total variation, both of which are stronger metrics.

\medskip
\noindent\emph{Comparison in KL divergence.}
For tamed schemes in KL divergence, the closest comparison is with \cite{lytras2024tamed,lytras2025taming,lytras2025ktula}, of which \cite{lytras2025ktula} provides the strongest known guarantees prior to the present work. The current article provides the best known complexity for convergence to the invariant measure under these assumptions. More specifically, ignoring logarithmic factors in $\varepsilon$, our algorithm converges in KL divergence after
\[
 n \;\ge\; \widetilde{\mathcal{O}}\!\bigl(d^{3(\ell+1)/2}\,\varepsilon^{-1/2}\bigr).
\]
On the other hand, \cite{lytras2025ktula} achieves a general rate of $\widetilde{\mathcal{O}}\!\bigl(d^{\frac{\ell+1}{\epsilon_h}} \varepsilon^{-\frac{1}{2-\epsilon_h}}\bigr)$ for $\epsilon_h \in \bigl(0,\frac{1}{2}\bigr]$, which explodes in the dimension as $\epsilon_h \to 0$. Specializing their framework to $\epsilon_h=\frac{1}{2}$ corresponding to our setting, that work achieves a complexity of $\widetilde{\mathcal{O}}\!\bigl(d^{2(\ell+1)}\,\varepsilon^{-\frac{2}{3}}\bigr)$ for KL accuracy $\varepsilon$.

\medskip
\noindent\emph{Comparison for randomized schemes.}
To the best of our knowledge, the only existing analysis of randomized Langevin schemes under super-linear drift growth is \cite{wang2025langevin}, which establishes Wasserstein convergence. Our analysis of \eqref{eq:tamed-rLMC} matches the Wasserstein rate of \cite{wang2025langevin} and additionally provides, for the first time in this setting, a non-asymptotic convergence guarantee in total variation, obtained through a new total variation local-error framework.

\section{Numerical Experiments}
\label{sec:numerical-experiments}

\subsection{Sampling experiments: stability and accuracy}
\label{subsec:sampling-experiments}

We first study the stability and accuracy of the proposed kTULA scheme \eqref{eq:kTULA} and the tamed randomized Langevin method \eqref{eq:tamed-rLMC}, and compare them with the classical Unadjusted Langevin Algorithm (ULA). As a benchmark we consider the double-well potential
\begin{equation}
U(x)=\tfrac14 x^4-\tfrac12 x^2,
\end{equation}
with gradient \(\nabla U(x)=x^3-x\). This example satisfies the dissipativity assumptions used in our analysis while exhibiting superlinear drift growth at infinity.

All experiments are performed in dimension \(d=100\) with inverse temperature \(\beta=1\). For each stepsize \(\lambda\in\{0.1,0.01,0.001\}\), ULA, kTULA, and tRLMC are run for \(N=2\times 10^5\) iterations. The initial condition is chosen deterministically as \(X_0=(200,0,\dots,0)\in\R^{100}\),
placing the system far from equilibrium and in a regime where the superlinear drift dominates. For kTULA and tRLMC, samples are collected after a burn-in period of \(5\times 10^4\) iterations. Each experiment is repeated independently \(30\) times, with all randomness generated from fixed seeds.

To assess stability, we record the \emph{explosion time} of ULA, defined as the first iteration at which numerical overflow occurs. To assess accuracy for kTULA and tRLMC, we estimate the second moment of the first coordinate and report the absolute error \(\bigl|\widehat{\E}[X_1^2]-\E_{\pi_\beta}[X_1^2]\bigr|\).
Here \(\pi_\beta\) denotes the invariant law with density proportional to \(\exp\!\bigl(-\beta\sum_{i=1}^d U(x_i)\bigr)\).
Since the coordinates are independent under the target distribution, \(\E_{\pi_\beta}[X_1^2]\) is the second moment of the associated one-dimensional marginal. In the experiments this reference value is computed numerically by trapezoidal quadrature on a dense grid over \([-4,4]\).

\begin{figure}[htbp]
    \centering
    \begin{subfigure}[b]{0.32\textwidth}
        \centering
        \includegraphics[width=\linewidth]{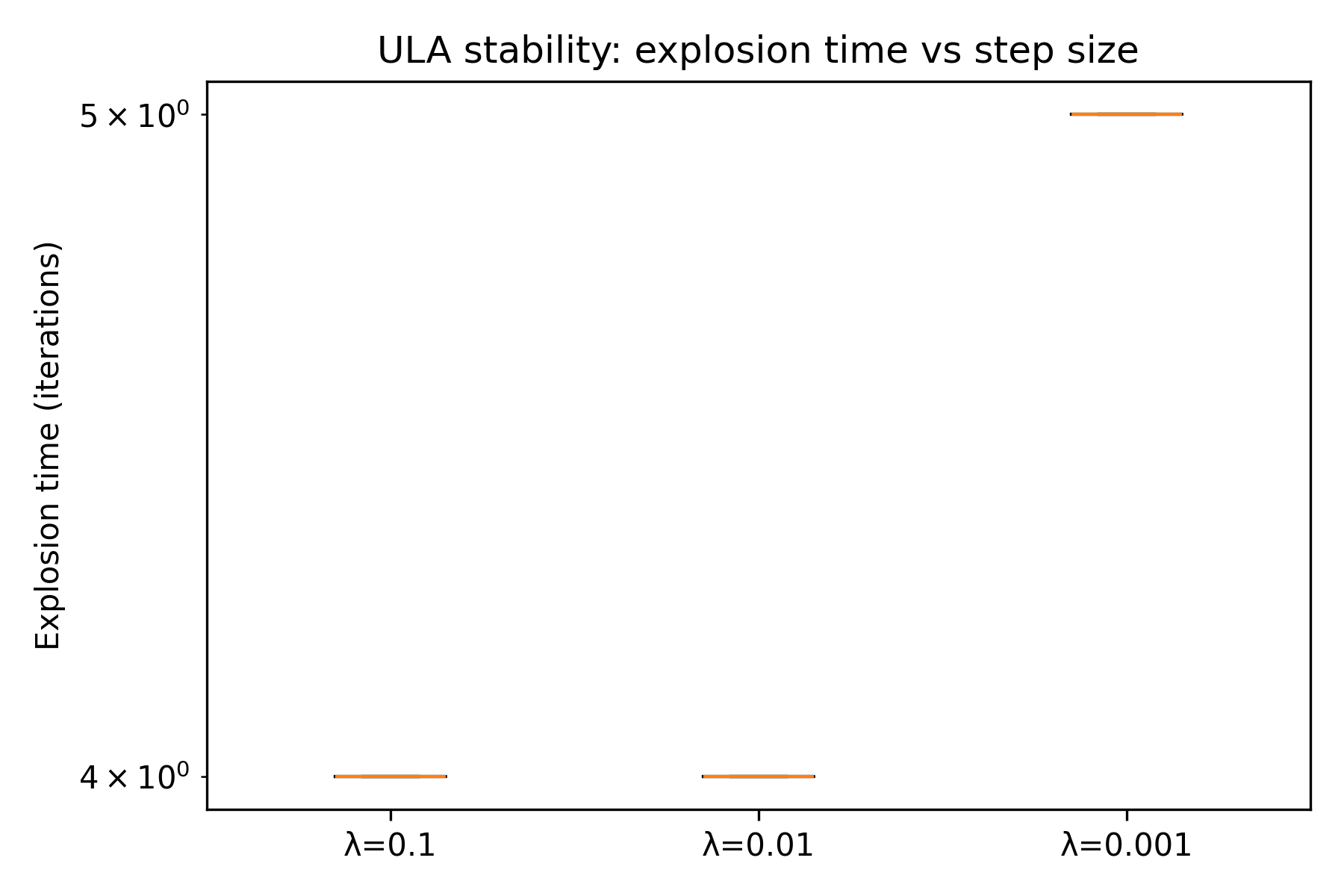}
        \caption{ULA explosion time vs.\ stepsize.}
    \end{subfigure}
    \hfill
    \begin{subfigure}[b]{0.32\textwidth}
        \centering
        \includegraphics[width=\linewidth]{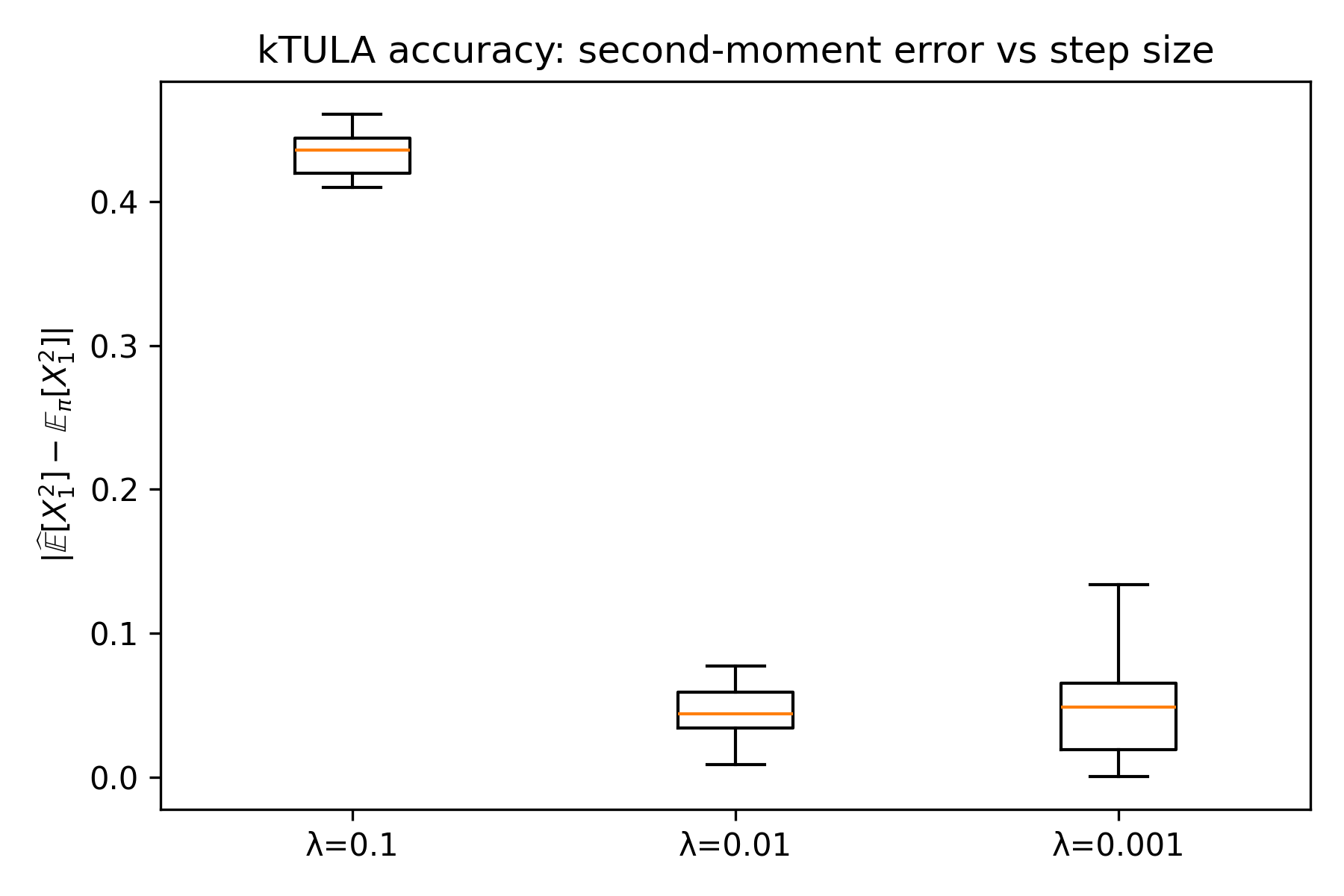}
        \caption{kTULA second-moment error vs.\ stepsize.}
    \end{subfigure}
    \hfill
    \begin{subfigure}[b]{0.32\textwidth}
        \centering
        \includegraphics[width=\linewidth]{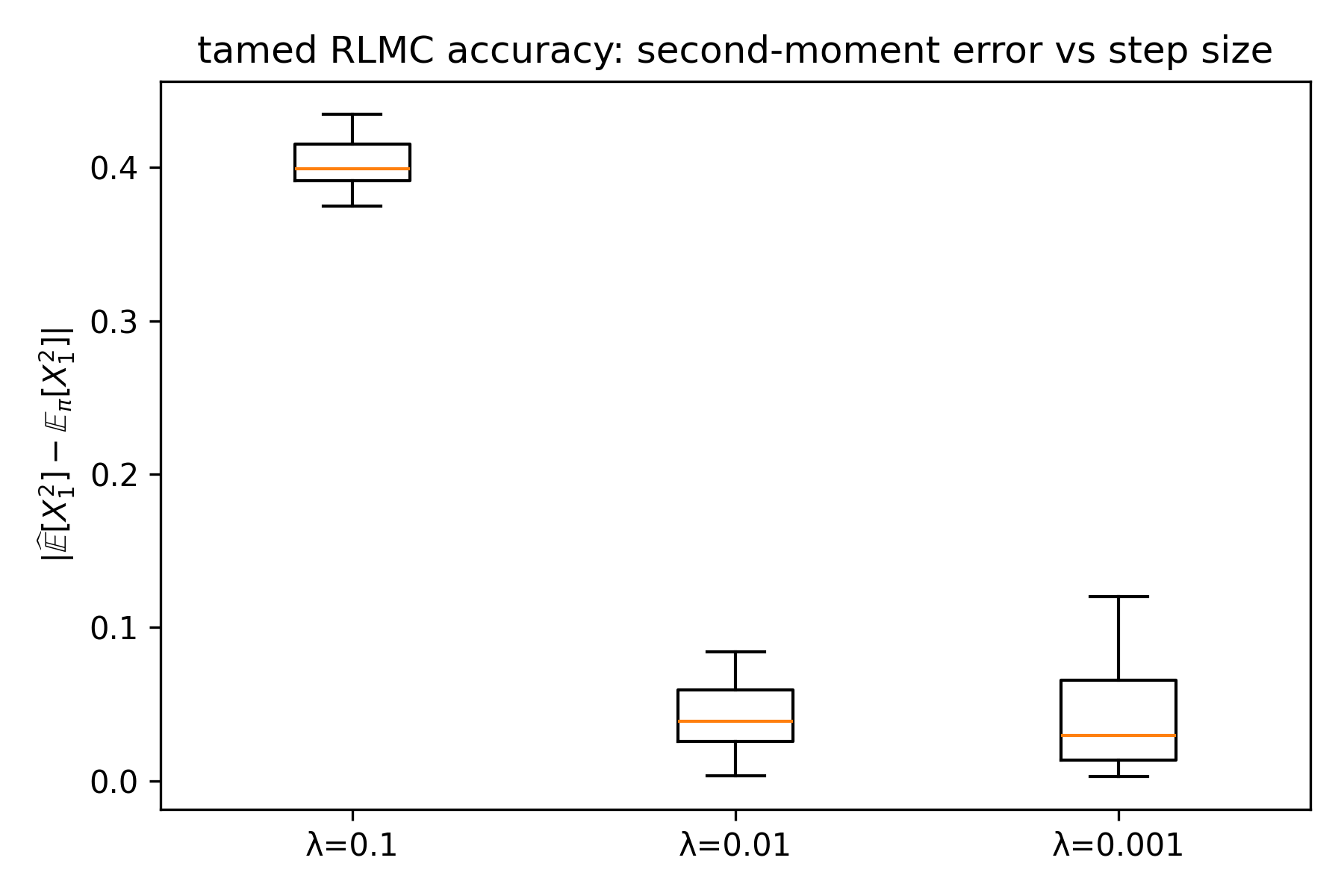}
        \caption{tRLMC second-moment error vs.\ stepsize.}
    \end{subfigure}
    \caption{Stability and accuracy in the aggressive regime (\(d=100\)). Left: ULA explosion time across the tested stepsizes. Center--right: second-moment error for kTULA and tRLMC.}
    \label{fig:ula-ktula-rlmc-boxplots}
\end{figure}

Figure~\ref{fig:ula-ktula-rlmc-boxplots} shows a clear separation between the untamed and tamed schemes. For every tested stepsize, ULA becomes unstable after only a few iterations, which is consistent with the known failure of Euler-type Langevin discretizations under superlinear drifts. By contrast, both kTULA and tRLMC remain numerically stable throughout the simulation horizon. Their empirical second-moment error also decreases substantially when the stepsize is reduced from \(\lambda=0.1\) to \(\lambda=0.01\), while stability is preserved across all tested values of \(\lambda\).

\begin{table}[htbp]
\centering
\caption{Summary statistics over \(30\) independent repetitions. For ULA we report explosion time; for kTULA and tRLMC we report the absolute second-moment error. Entries are empirical mean \(\pm\) standard deviation.}
\label{tab:sampling-summary}
\small
\setlength{\tabcolsep}{7pt}
\begin{tabular}{lccc}
\toprule
Method & \(\lambda=0.1\) & \(\lambda=0.01\) & \(\lambda=0.001\) \\
\midrule
ULA explosion time
& \(4.00 \pm 0.00\)
& \(4.00 \pm 0.00\)
& \(5.00 \pm 0.00\) \\

kTULA:
\(\bigl|\widehat{\E}[X_1^2]-\E_{\pi_\beta}[X_1^2]\bigr|\)
& \(0.4336 \pm 0.0154\)
& \(0.0453 \pm 0.0170\)
& \(0.0455 \pm 0.0325\) \\

tRLMC:
\(\bigl|\widehat{\E}[X_1^2]-\E_{\pi_\beta}[X_1^2]\bigr|\)
& \(0.4025 \pm 0.0159\)
& \(0.0428 \pm 0.0211\)
& \(0.0437 \pm 0.0360\) \\
\bottomrule
\end{tabular}
\end{table}

Table~\ref{tab:sampling-summary} makes the separation quantitative. ULA explodes after roughly four to five iterations for all tested stepsizes. The zero empirical standard deviation for the ULA explosion times indicates that, for each tested stepsize, overflow occurred at the same iteration in all \(30\) repetitions from the chosen deterministic initialization. In contrast, both tamed schemes remain stable in all runs.

Figure~\ref{fig:ktula-rlmc-behavior} complements this quantitative comparison by displaying the trajectory behavior of the two tamed schemes at the representative stepsize \(\lambda=0.01\). We plot the evolution of \(\|X_n\|^2\) and the empirical marginal density of the first coordinate against the corresponding one-dimensional marginal of the target law.

\begin{figure}[htbp]
    \centering
    \begin{subfigure}{0.32\textwidth}
        \centering
        \includegraphics[width=\linewidth]{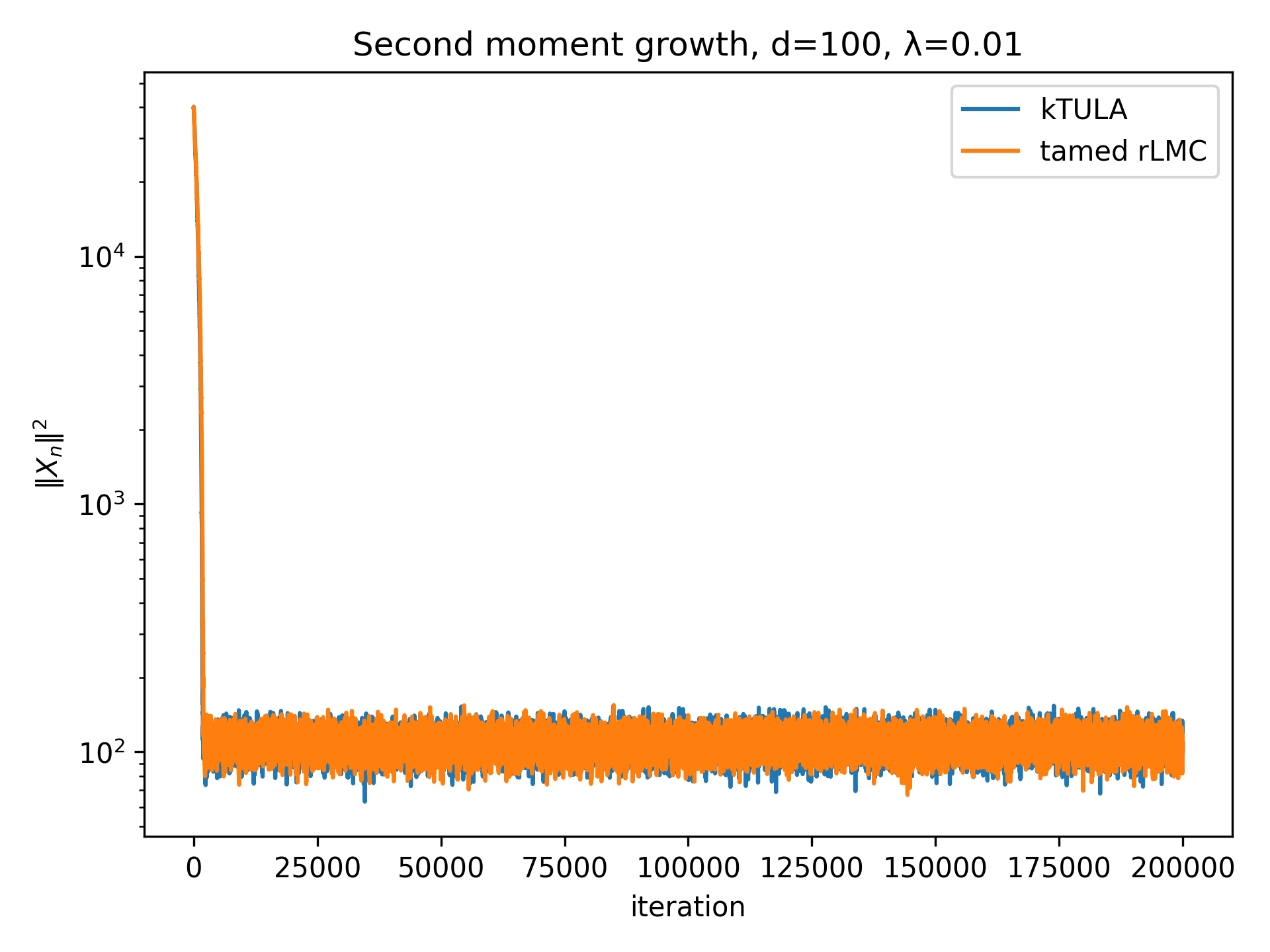}
        \caption{Second-moment growth.}
    \end{subfigure}
    \hfill
    \begin{subfigure}{0.32\textwidth}
        \centering
        \includegraphics[width=\linewidth]{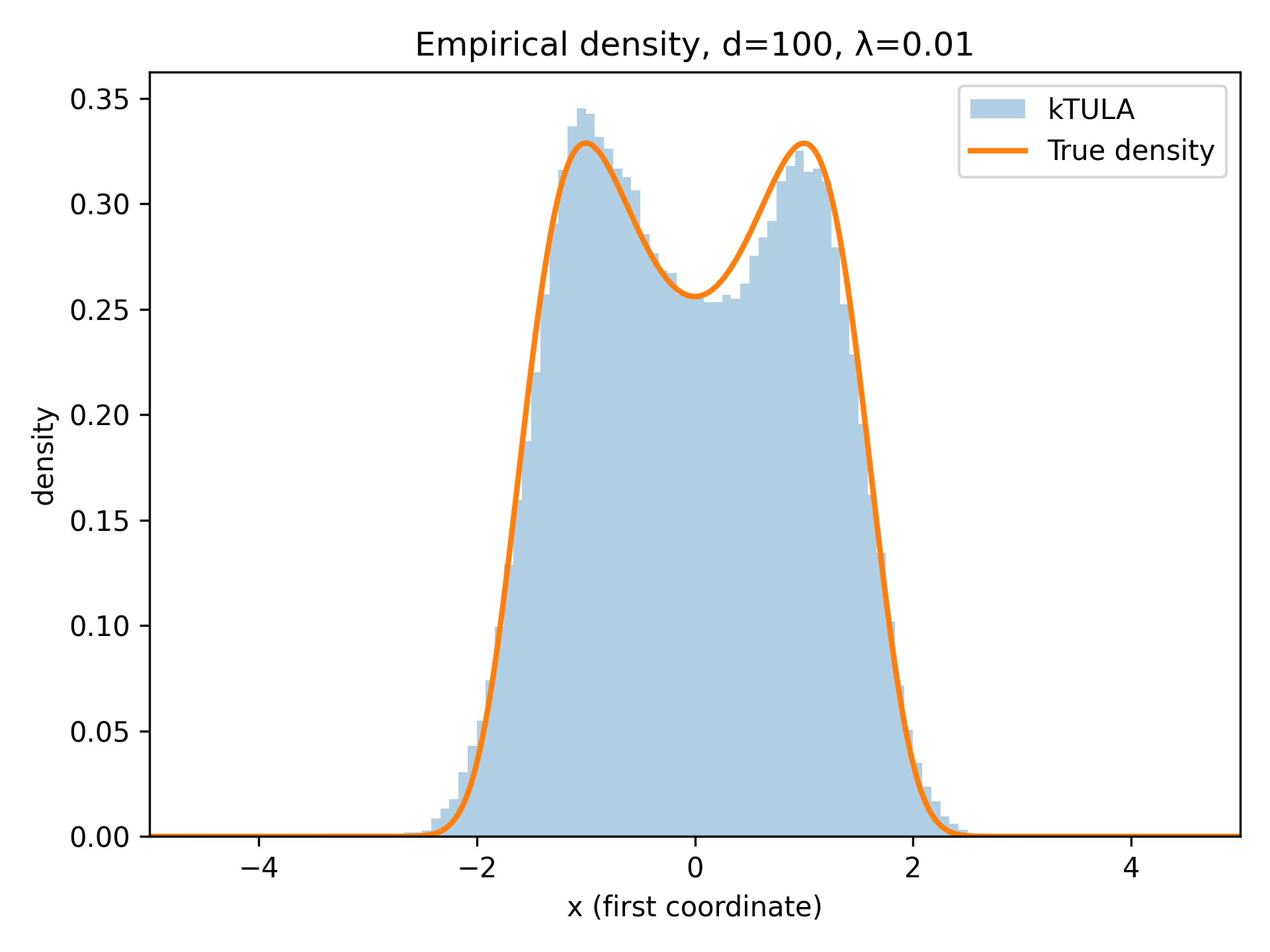}
        \caption{kTULA empirical density.}
    \end{subfigure}
    \hfill
    \begin{subfigure}{0.32\textwidth}
        \centering
        \includegraphics[width=\linewidth]{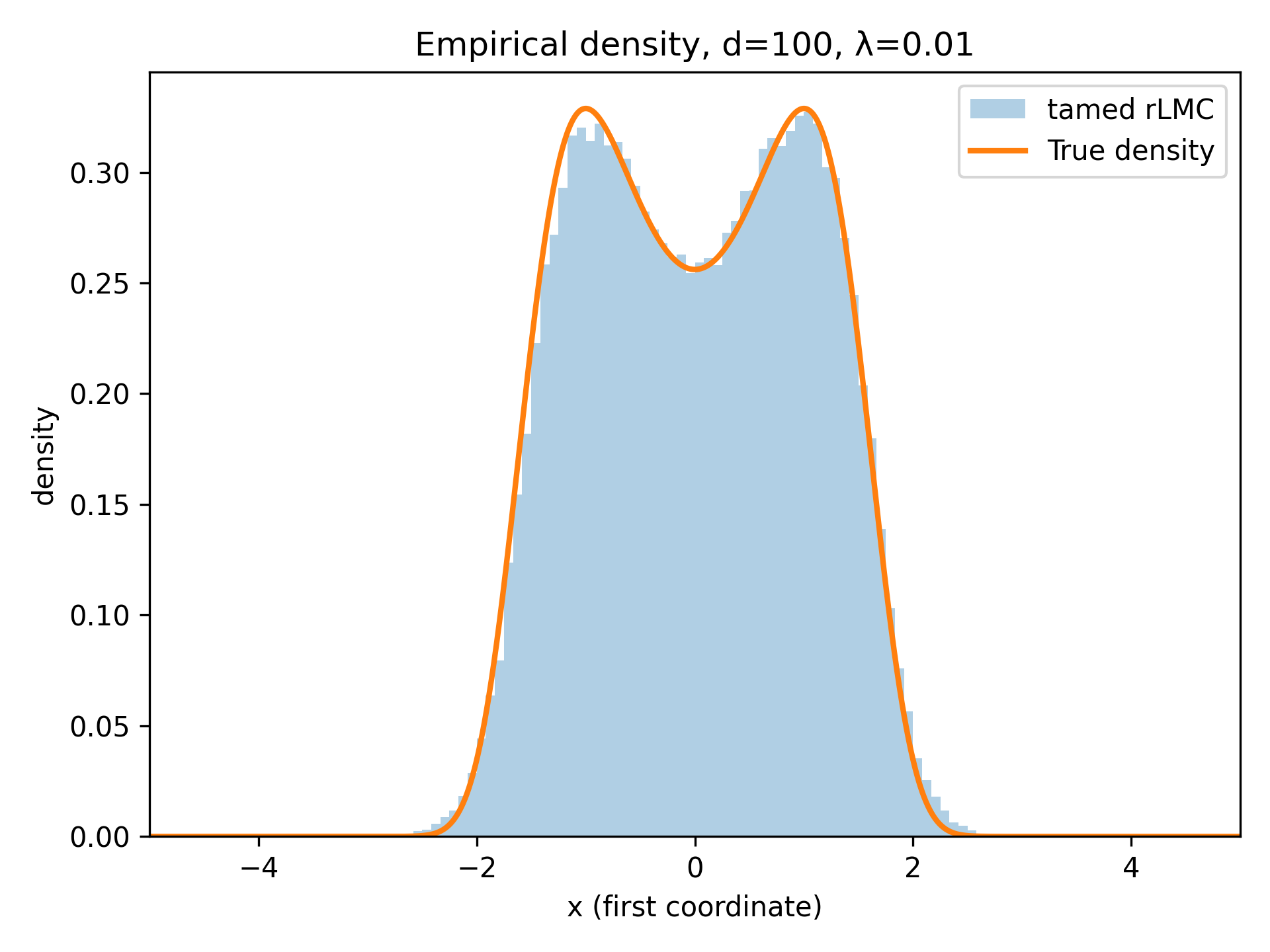}
        \caption{tRLMC empirical density.}
    \end{subfigure}
    \caption{Representative behavior at \(\lambda=0.01\) in the aggressive regime (\(d=100\)). Left: evolution of \(\|X_n\|^2\) (log scale) for kTULA and tRLMC. Center--right: empirical marginal densities of the first coordinate compared with the corresponding target marginal.}
    \label{fig:ktula-rlmc-behavior}
\end{figure}

The left panel of Figure~\ref{fig:ktula-rlmc-behavior} shows that both tamed schemes enter a stable regime after a short transient and maintain controlled second moments. The center and right panels indicate that the empirical marginals recover the expected bimodal structure of the target distribution, providing further evidence that the tamed schemes remain stable while capturing the qualitative shape of the invariant law.

Taken together, these experiments illustrate the stabilizing effect of taming in the sampling regime. In this aggressive setting, ULA fails almost immediately, whereas kTULA and tRLMC remain stable over long horizons and attain substantially smaller second-moment error once the stepsize is reduced from the coarsest regime.

\subsection{Optimization involving a neural network in the fixed-feature setting}
\label{subsec:nn-experiment}

We finally consider the optimization problem
\begin{equation}
\min_{\theta\in\R^d} u_n(\theta)
:=
\frac{1}{n}\sum_{j=1}^n \bigl(Y_j-\mathcal{N}(\theta,Z_j)\bigr)^2
+\frac{\eta}{6}\sum_{k=1}^d |\theta_k|^6,
\label{eq:nn-objective}
\end{equation}
where \(X_j=(Z_j,Y_j)\in\R^m\), with \(Y_j\in\R\) the target variable and \(Z_j\in\R^{m-1}\) the input variable, \(\eta>0\) is a regularization constant, and \(\mathcal{N}:\R^d\times\R^{m-1}\to\R\) is a single-hidden-layer feed-forward neural network with fixed random features, given by
\(\mathcal{N}(\theta,z):=\sum_{i=1}^{d_1} W_i\,\sigma(\langle c_i,z\rangle+b_i)\),
where \(c_i\in\R^{m-1}\) are fixed random feature vectors, \(W_i\in\R\) are trainable output weights, and \(b_i\in\R\) are trainable bias parameters. The activation is the sigmoid linear unit \(\sigma(x)=x/(1+e^{-x})\), and the trainable parameter is \(\theta=((W_1,\dots,W_{d_1}),(b_1,\dots,b_{d_1}))\in\R^{2d_1}\), so that \(d=2d_1\).

We ask whether the
provable stability of the tamed schemes manifests in a simple nonlinear
problem of the form~\eqref{eq:nn-objective} when the learning rate is taken
into an aggressive regime, and whether standard first-order optimizers — which
carry no comparable guarantees on \eqref{eq:nn-objective} — display the
deterioration one would expect.

We generate a synthetic regression dataset from a teacher model of the same fixed-feature form. The inputs are sampled from a standard Gaussian distribution in dimension \(m-1=20\); the teacher network has hidden width \(80\); independent Gaussian noise with standard deviation \(0.05\) is added to the teacher output; and the student network has hidden width \(d_1=100\). The training set size is \(n=4000\), the test set size is \(1000\), and training is performed with mini-batches of size \(128\).

We compare SGD, Adam, AMSGrad, kTULA, and tRLMC on the empirical objective \eqref{eq:nn-objective}. For SGD we use momentum \(0.9\). For the two tamed schemes we set \(\beta=10^6\), \(a=10^{-2}\), \(\ell=4\), and \(\eta=0.05\). Here \(a\) is the linear stabilizing component in the taming map; we fix \(a=10^{-2}\) throughout all runs as a small baseline dissipative term. Each method is run for \(20\) epochs, repeated over the five seeds \(1,2,3,4,5\), and tested at learning rates \(\lambda\in\{0.1,0.2,0.3\}\).

Figure~\ref{fig:nn-final-mse} reports the final test MSE as a function of the learning rate. The main qualitative trend is that the two tamed methods are empirically less sensitive to increasing learning rate than SGD, Adam, and AMSGrad in this synthetic setting. In particular, the three standard optimizers deteriorate substantially as \(\lambda\) increases, whereas kTULA varies only moderately across the tested range. tRLMC performs comparably to kTULA at \(\lambda=0.1\) and \(\lambda=0.2\), but degrades more noticeably at \(\lambda=0.3\). Since this experiment is small and synthetic, we interpret these results as evidence of a favorable stability trend rather than as a broad benchmark conclusion.

\begin{figure}[htbp]
    \centering
    \includegraphics[width=0.62\textwidth]{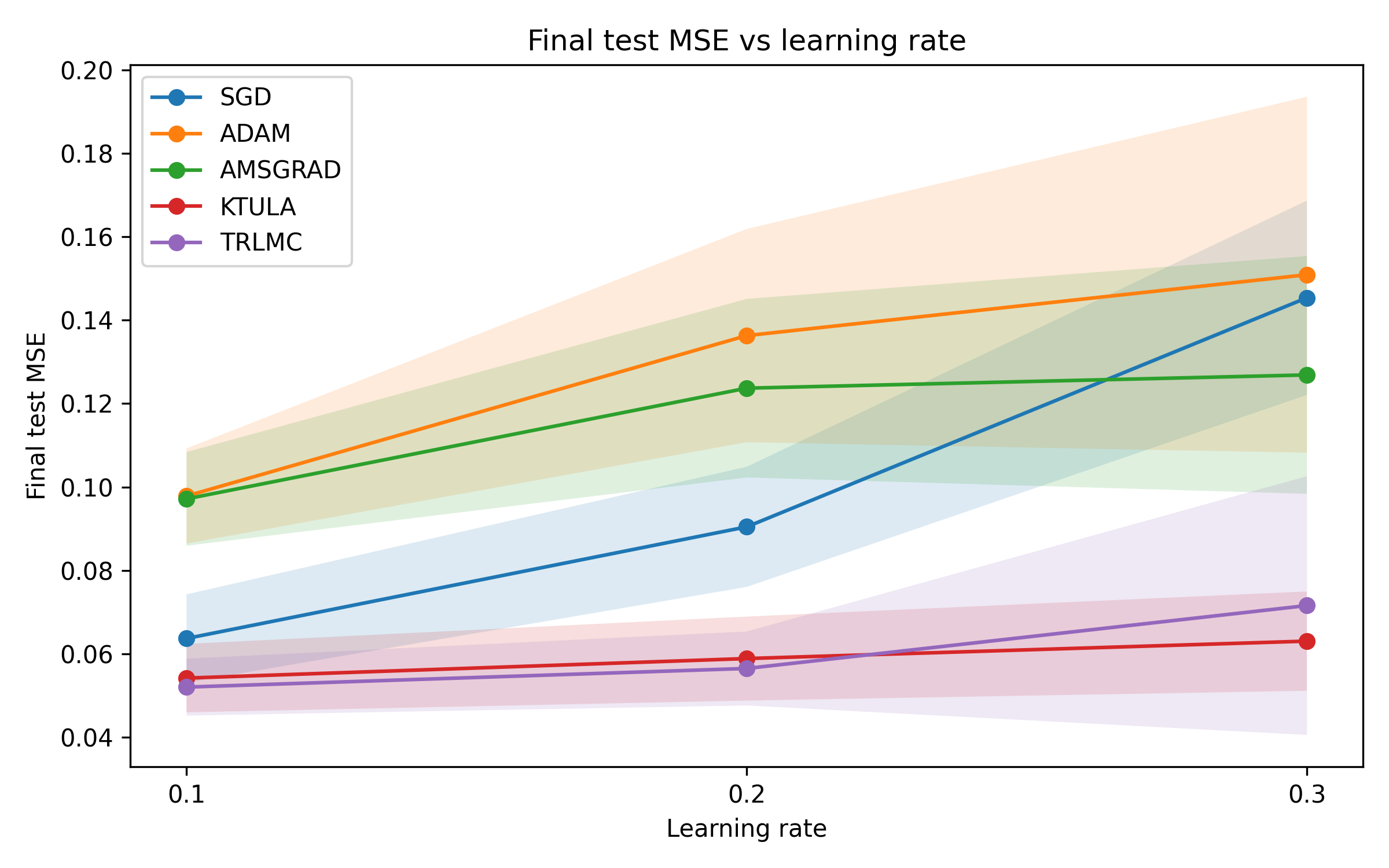}
    \caption{Final test MSE versus learning rate in the fixed-feature neural-network experiment. The tamed methods are empirically less sensitive to increasing learning rate than SGD, Adam, and AMSGrad on this synthetic task.}
    \label{fig:nn-final-mse}
\end{figure}

\begin{table}[htbp]
\centering
\caption{Final test MSE over $5$ seeds in the fixed-feature neural-network experiment. Entries are empirical mean $\pm$ standard deviation across seeds.}
\label{tab:nn-final-mse}
\begin{tabular}{lccc}
\toprule
Method & $\lambda=0.1$ & $\lambda=0.2$ & $\lambda=0.3$ \\
\midrule
SGD & $0.0637 \pm 0.0106$ & $0.0904 \pm 0.0144$ & $0.1454 \pm 0.0233$ \\
Adam & $0.0978 \pm 0.0114$ & $0.1363 \pm 0.0256$ & $0.1509 \pm 0.0427$ \\
AMSGrad & $0.0972 \pm 0.0112$ & $0.1237 \pm 0.0214$ & $0.1269 \pm 0.0285$ \\
kTULA & $0.0542 \pm 0.0082$ & $0.0589 \pm 0.0101$ & $0.0630 \pm 0.0119$ \\
tRLMC & $0.0520 \pm 0.0068$ & $0.0565 \pm 0.0089$ & $0.0716 \pm 0.0310$ \\
\bottomrule
\end{tabular}
\end{table}

Table~\ref{tab:nn-final-mse} makes the final-error comparison quantitative. In particular, kTULA has the smallest mean final test MSE at $\lambda=0.2$ and $\lambda=0.3$, while tRLMC is slightly better at $\lambda=0.1$ but becomes more variable at $\lambda=0.3$.

Figure~\ref{fig:nn-epochwise} complements this comparison by displaying the epoch-wise evolution of the parameter norm, test MSE, and training objective at the representative learning rate \(\lambda=0.1\). The parameter-norm panel provides the clearest visual separation between the standard and tamed methods. SGD exhibits steadily increasing norm, while Adam and AMSGrad operate at substantially larger parameter norms throughout training. By contrast, kTULA and tRLMC remain in a much smaller norm regime over the full horizon.

The test-MSE panel suggests that, after the initial transient, the two tamed methods maintain lower average test error than the standard baselines at \(\lambda=0.1\). Likewise, the training-objective panel indicates that kTULA and tRLMC settle at lower objective values, while Adam and AMSGrad level off at visibly higher levels. These plots do not show instability in the strict sense of numerical explosion, but they do indicate that taming leads to more controlled iterates and more stable empirical performance in this regime.

\begin{figure}[htbp]
    \centering
    \begin{subfigure}{0.32\textwidth}
        \centering
        \includegraphics[width=\linewidth]{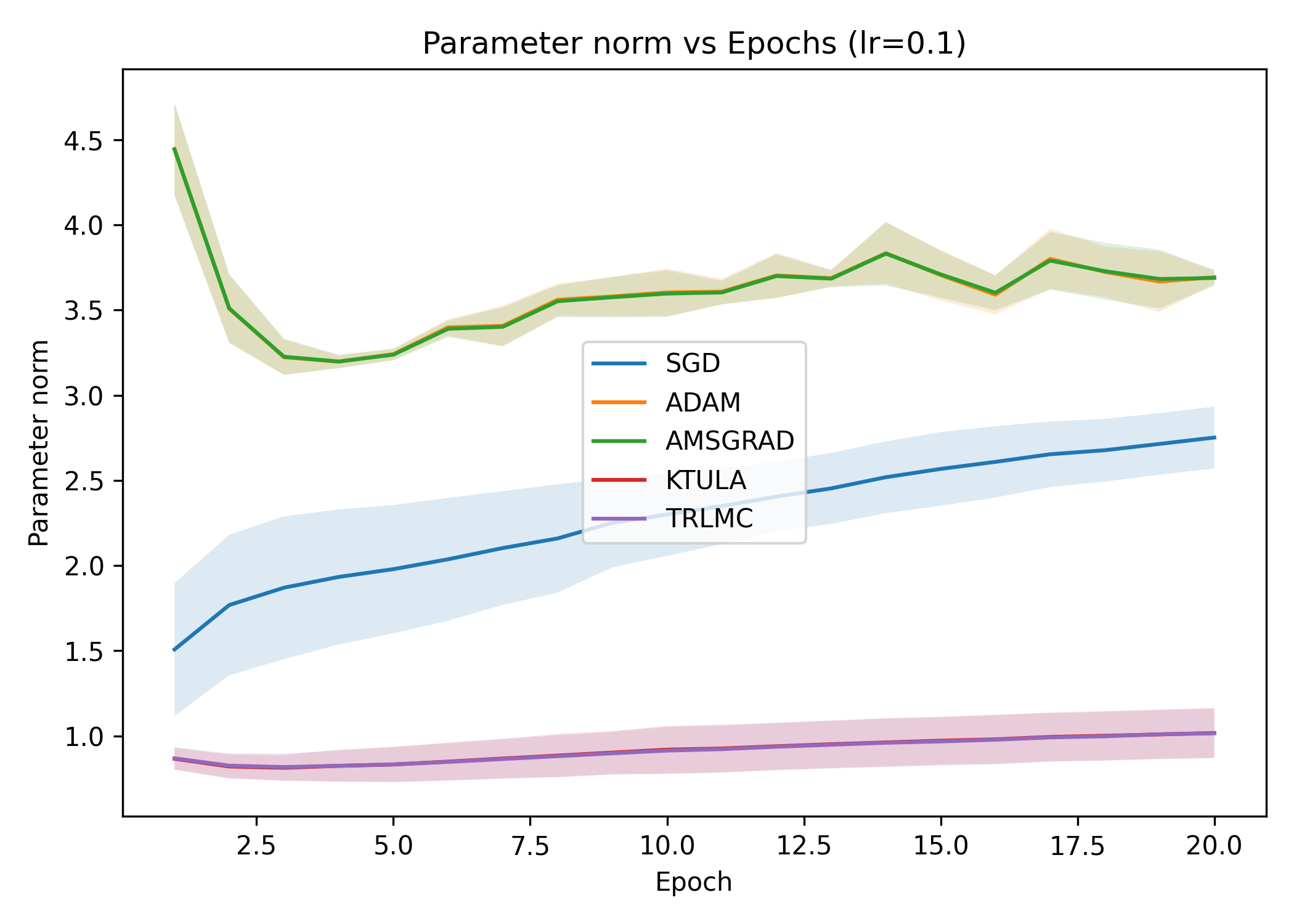}
        \caption{Parameter norm.}
    \end{subfigure}
    \hfill
    \begin{subfigure}{0.32\textwidth}
        \centering
        \includegraphics[width=\linewidth]{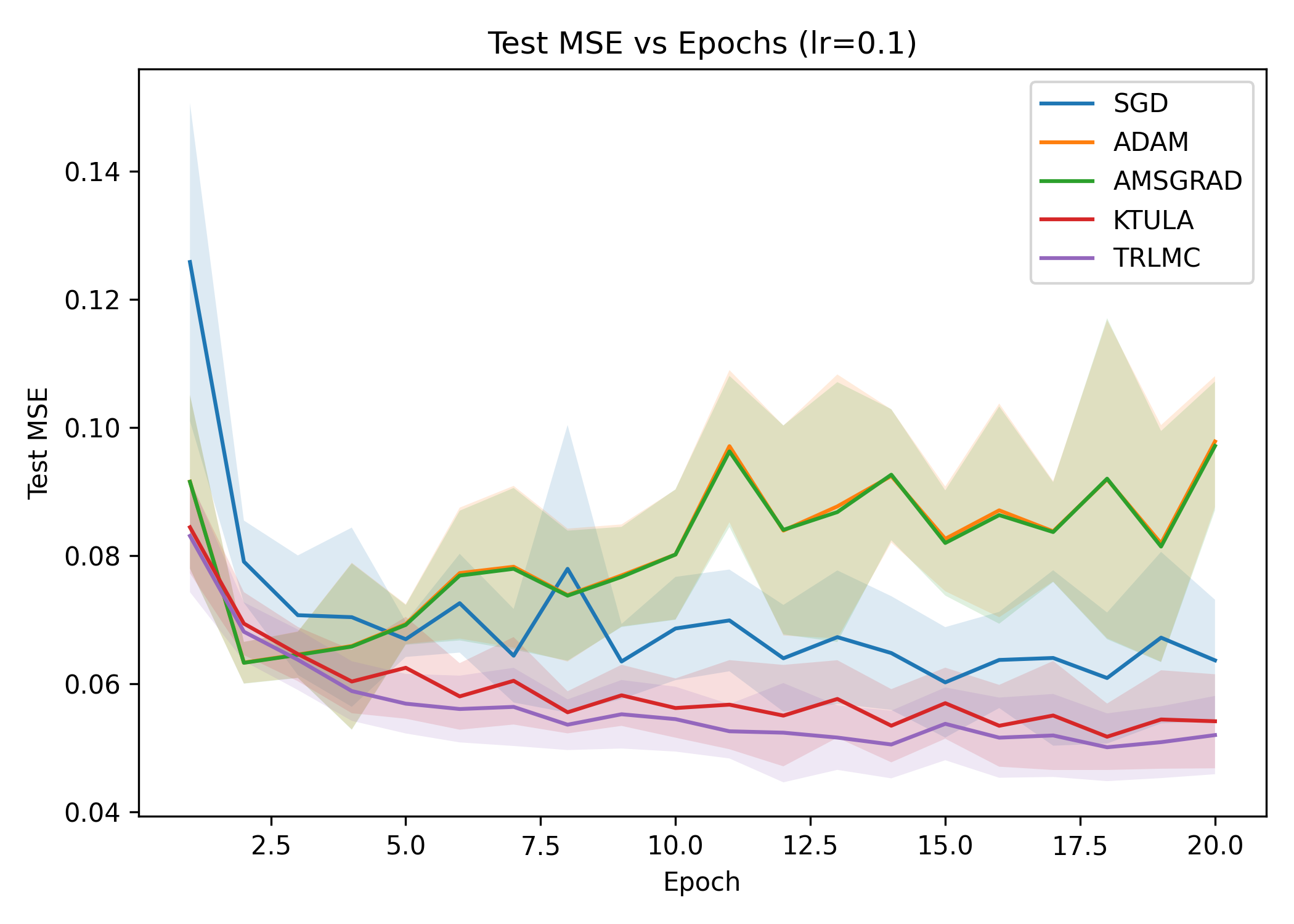}
        \caption{Test MSE.}
    \end{subfigure}
    \hfill
    \begin{subfigure}{0.32\textwidth}
        \centering
        \includegraphics[width=\linewidth]{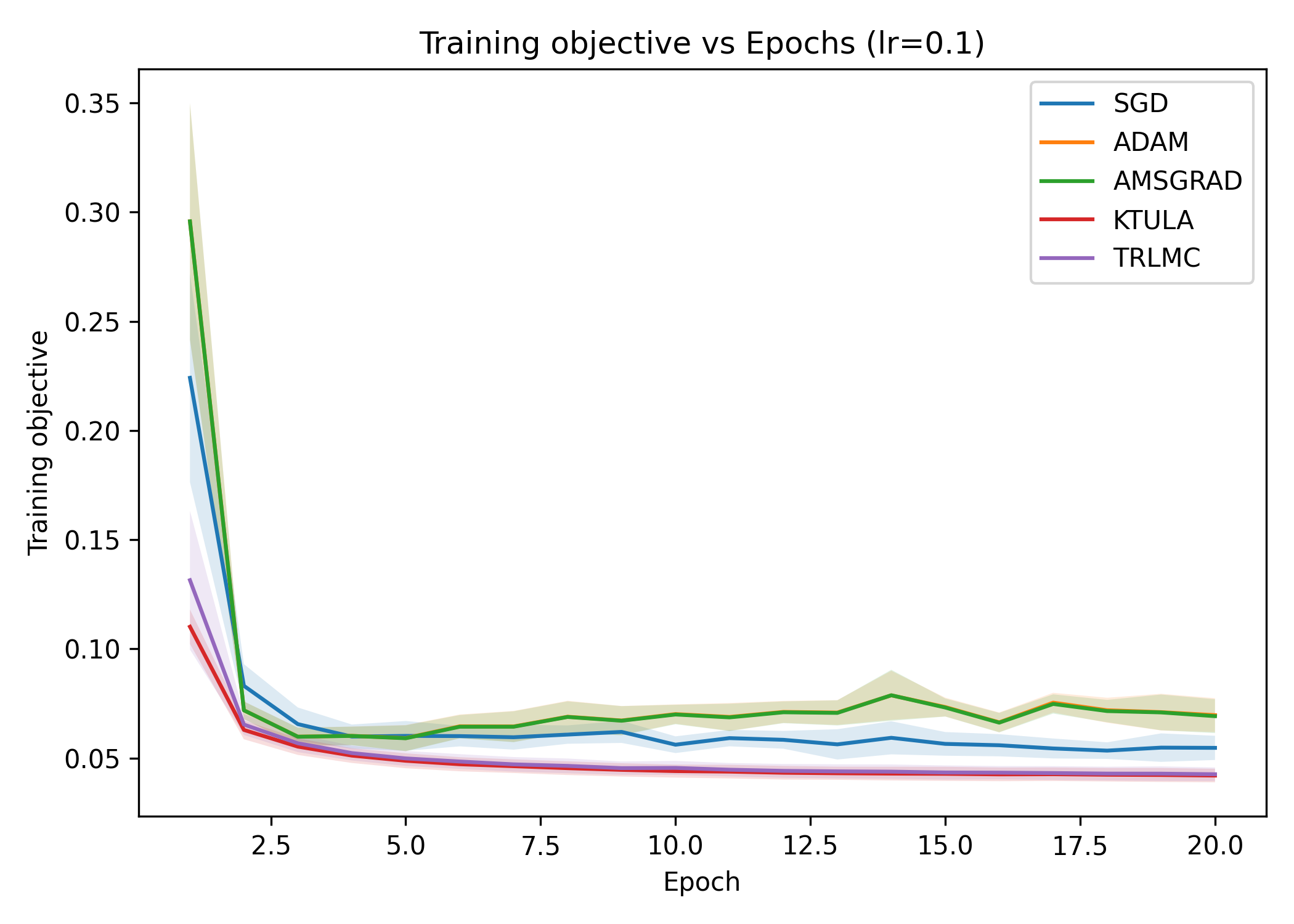}
        \caption{Training objective.}
    \end{subfigure}
    \caption{Epoch-wise behavior at learning rate \(\lambda=0.1\). The tamed methods remain in a smaller parameter-norm regime and achieve lower average test error and training objective than the standard baselines on this synthetic task. Shaded regions indicate one empirical standard deviation across the five seeds.}
    \label{fig:nn-epochwise}
\end{figure}

Taken together, these experiments support the same qualitative message as the preceding study. In this nonlinear regression problem with high-order regularization and relatively aggressive learning rates, the tamed Langevin-type schemes display more controlled parameter growth and more robust empirical behavior than SGD, Adam, and AMSGrad. We emphasize, however, that this experiment is intended as a small-scale diagnostic study.

%% file: body/Framework.tex
\section{Proof outline and technical innovations}
\label{sec:framework}

The convergence analyses underlying our main results share a common
technical core, which we develop in detail in this section. The goal is
to bound the divergence between two stochastic processes driven by
\emph{different} Markov kernels: the discrete scheme \ref{eq:kTULA} or
\ref{eq:tamed-rLMC} on the one hand, and the ~\eqref{eq:SDE} on the other. Classical coupling arguments are
tailored to the Wasserstein metric and do not transfer verbatim to
information-theoretic divergences such as Kullback--Leibler or
total variation. The shifted composition methodology of~\cite{chewi2024local} provides a principled remedy by
introducing an auxiliary, third process that interpolates between the
two processes of interest. We adapt this methodology to the present
setting, where the super-linear growth of the drift induces a
state-dependent coupling increment that does not appear
in~\cite{chewi2024local}. This is the only point at which the abstract
machinery genuinely departs from~\cite{chewi2024local}.

\medskip
Following the two-layer structure of the paper, the local-error frameworks developed in this section---both the KL framework of~\S\ref{subsec:framework-setup}--\S\ref{subsec:framework-KL-bound} and the TV framework of~\S\ref{subsec:framework-tv}---are formulated at the level of the ~\eqref{eq:SDE} with general locally Lipschitz drift $h$, requiring no gradient structure. The Wasserstein convergence argument of~\S\ref{subsec:framework-W2} additionally invokes the sampling specialization $h=\nabla u$ and Assumption~\ref{ass:LSI}, as it relies on the exponential contraction of $(P_t)_{t\ge 0}$ toward $\pi_\beta$.

% =====================================================================
\subsection{Setup, kernels, and one-step local errors}
\label{subsec:framework-setup}

Fix a step size $\lambda\in(0,1]$. Let $P=P_\lambda$ denote the Markov
kernel of~\eqref{eq:SDE} run for time $\lambda$,
and let $\widehat P$ denote the one-step transition kernel of
~\eqref{eq:kTULA}. For any $\mu,\nu\in\mathcal{P}_2(\R^d)$
and $n\in\mathbb{N}_0$, set $\widehat\mu_n := \mu\,\widehat P^{\,n}$ and
$\nu_n := \nu\,P^{n}$.
The objective of the framework is to control
$\KL(\widehat\mu_N\,\|\,\nu_N)$ in terms of \emph{one-step}, i.e.\ local,
quantities computed at a single iteration.

We work with three measurable functions encoding the one-step
discrepancy between $\widehat P$ and $P$:
\[
  E_{\mathrm{strong}}, E_{\mathrm{weak}}\colon\R^d\to[0,\infty),
  \qquad
  \Gamma\colon\R^d\times\R^d\times(0,\infty)\to[0,\infty).
\]
Here $E_{\mathrm{strong}}(x)$ controls the $L^2$-distance between
one step of $\widehat P$ and one step of $P$ from the same point $x$;
$E_{\mathrm{weak}}(x)$ is the corresponding bias, which exploits
cancellations and is therefore typically smaller than
$E_{\mathrm{strong}}(x)$; and $\Gamma(x,y,\lambda)$ is a
\emph{coupling increment} controlling the synchronous-coupling
distortion between two trajectories of~\eqref{eq:SDE} started at $x$ and $y$.
For convenience we set
\begin{equation}\label{eq:framework-a0a1}
  a_0(x):=E_{\mathrm{strong}}(x),
  \qquad
  a_1(x,y):=E_{\mathrm{weak}}(x)+\Gamma(x,y,\lambda)\,E_{\mathrm{strong}}(x).
\end{equation}

The framework is fed by the following one-step Wasserstein bound,
verified for \ref{eq:kTULA} in Lemma~\ref{lem:one-step-w2}: there
exists $L>0$ such that, for all $x,y\in\R^d$,
\begin{equation}\label{eq:one-step-w2-framework}
  \W^{2}\!\bigl(\delta_x\widehat P,\delta_y P\bigr)
  \;\le\;
  L^{2}\,\|x-y\|^{2}
  +2\,a_1(x,y)\,\|x-y\|
  +a_0(x)^{2}.
\end{equation}

\begin{remark}[State-dependence of $a_1$]
\label{rem:state-dependence}
In~\cite{chewi2024local}, the coupling increment $\gamma$ is a constant. In the present
super-linear setting, the polynomial Lipschitz
condition~\ref{ass:PLC} forces $\Gamma$ to depend on both endpoints,
namely $\Gamma(x,y,\lambda)=C_{\mathrm{cpl}}\lambda(1+\|x\|^{\ell'}+\|y\|^{\ell'})$.
This is the only structural difference between our framework
and~\cite{chewi2024local}; as we shall see, it is benign provided
uniform moments along the auxiliary process introduced
in~\S\ref{subsec:framework-aux} are established.
\end{remark}

% =====================================================================
\subsection{The auxiliary interpolating process}
\label{subsec:framework-aux}

The principal obstacle in bounding $\KL(\widehat\mu_N\,\|\,\nu_N)$
directly is that $\widehat\mu_N$ and $\nu_N$ are produced by
\emph{different} Markov kernels: any one-step KL bound therefore
encodes both the discretization bias and the contraction of the
SDE semigroup, and the two effects are difficult to separate.
The shifted composition strategy circumvents this obstacle by
introducing a third process $(Y_n')_{n=0}^{N}$ that
(i) starts at $\nu_0$, (ii) terminates at $\widehat\mu_N$, and
(iii) evolves at intermediate steps via the SDE kernel~$P$, so
that intermediate one-step KL terms only involve~$P$ and are amenable
to standard regularity estimates of the semigroup $(P_t)_{t\ge 0}$.

\paragraph{Construction.}
Let $(\widehat X_n)_{n\ge0}$, $(Y_n)_{n\ge0}$ denote
the \ref{eq:kTULA} chain and the chain of one-step SDE updates with $\widehat X_0\sim\mu$ and
$Y_0\sim\nu$. Choose shift coefficients
$(\eta_n)_{n=0}^{N-1}\subset[0,1]$ with $\eta_{N-1}=1$, and define
\begin{equation}\label{eq:tildeY-def}
  \widetilde Y_n := (1-\eta_n)\,Y_n' + \eta_n\,\widehat X_n,
  \qquad
  n=0,1,\dots,N-1.
\end{equation}
Geometrically, $\widetilde Y_n$ lies on the segment between $Y_n'$
and $\widehat X_n$; the parameter $\eta_n$ measures how far the
auxiliary process is pushed toward the \ref{eq:kTULA} chain at step
$n$. Define
\begin{equation}\label{eq:Qn-def}
  Q_n :=
  \begin{cases}
    P, & 0\le n\le N-2,\\[2pt]
    \widehat P, & n=N-1,
  \end{cases}
\end{equation}
and update the auxiliary process by
\begin{equation}\label{eq:aux-update}
  Y_{n+1}' \sim Q_n\bigl(\widetilde Y_n,\cdot\bigr),
  \qquad
  n=0,1,\dots,N-1,
  \qquad
  Y_0'=Y_0.
\end{equation}
We write $\nu_n':=\Law(Y_n')$.

The terminal condition $\eta_{N-1}=1$ together with the choice
$Q_{N-1}=\widehat P$ guarantees the interpolation property
\begin{equation}\label{eq:interpolation}
  \nu_0' \;=\; \nu_0,
  \qquad
  \nu_N' \;=\; \widehat\mu_N,
\end{equation}
since $\widetilde Y_{N-1}=\widehat X_{N-1}$ and the last update applies
$\widehat P$ to $\widehat X_{N-1}$. Figure~\ref{fig:aux-process}
illustrates the construction. The freedom to choose $\eta_n\in[0,1]$
is what enables the framework to exploit \emph{both} the contraction
of $P$ and the local discretization errors of $\widehat P$: the
bounds, determined in Proposition~\ref{prop:deterministic-dn},
dictate the final iteration complexity.

\begin{figure}[ht]
  \centering
  \includegraphics[width=0.70\linewidth]{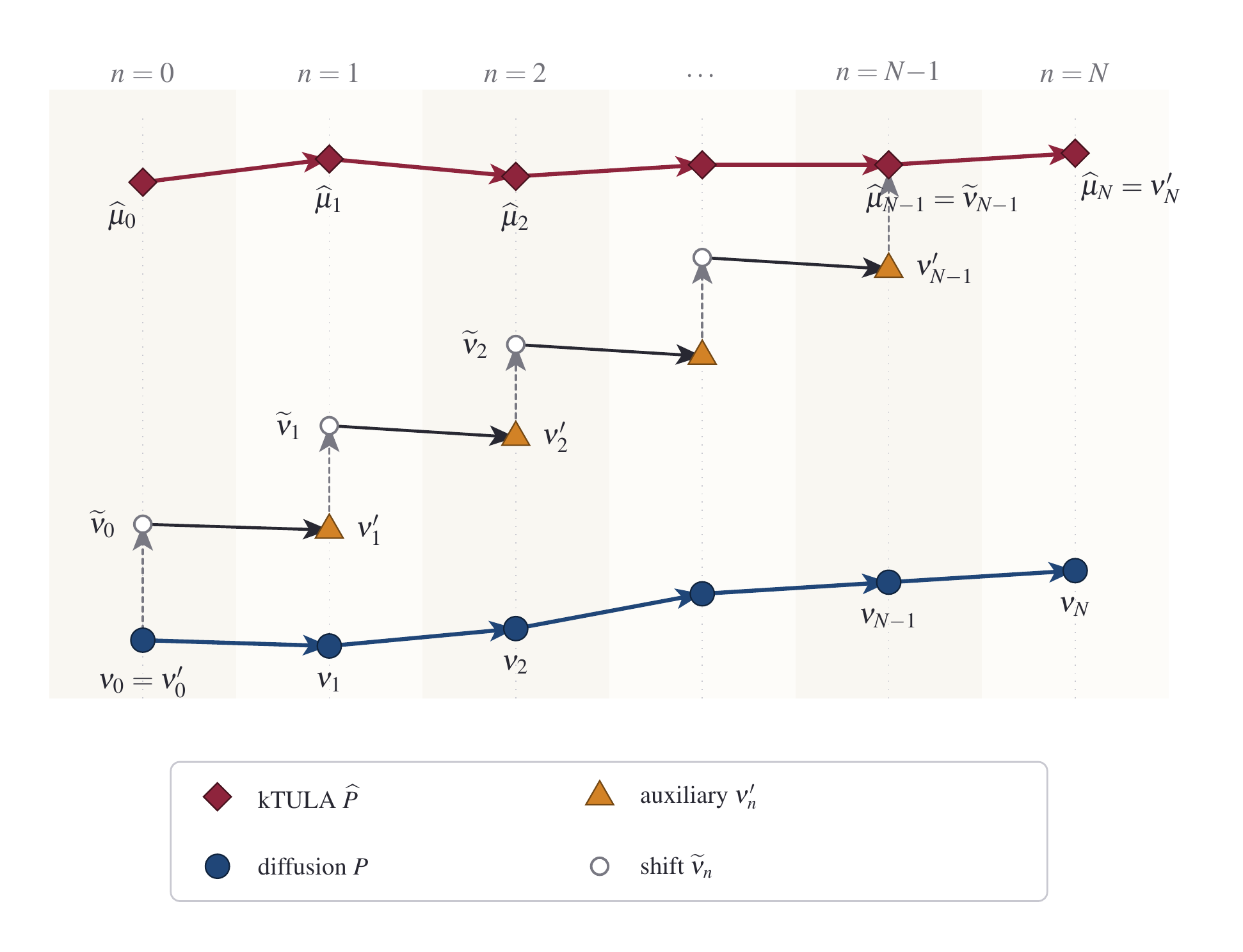}
  \caption{Auxiliary interpolating process underlying the shifted-composition argument. The top trajectory $\{\widehat\mu_n\}$ (red diamonds) is generated by the kTULA kernel $\widehat P$ from initial law $\mu$, while the bottom trajectory $\{\nu_n\}$ (blue circles) is generated by the SDE kernel $P$ from initial law $\nu$. The auxiliary process $\{\nu_n'\}$ (orange triangles) matches the bottom at $n=0$ and the top at $n=N$. At each step $n<N-1$, the law $\nu_n'$ is shifted along the Wasserstein geodesic toward $\widehat\mu_n$ (gray dashed arrow) to produce the shifted state $\widetilde\nu_n$ (open circle), after which the SDE kernel $P$ is applied. Only at $n=N-1$, where the constraint $\eta_{N-1}=1$ forces $\widetilde\nu_{N-1}=\widehat\mu_{N-1}$, the kTULA kernel $\widehat P$ is applied instead to enforce~\eqref{eq:interpolation}.}
  \label{fig:aux-process}
\end{figure}

% =====================================================================
\subsection{KL telescoping via the shifted chain rule}
\label{subsec:framework-shifted-chain-rule}

The bridge between the auxiliary process and the KL divergence is
the shifted chain rule~\eqref{eq:shifted-chain-rule} introduced in
Section~\ref{sec:preliminaries}. Recall that the auxiliary variable
$X'$ allows comparison of Markov updates from \emph{different}
starting points, which is precisely what is needed when the
underlying kernels disagree.

We assume the following two one-step KL bounds:
\begin{align}
  \KL\!\bigl(\delta_x P\,\|\,\delta_y P\bigr)
  &\;\le\; c\,\|x-y\|^2,
  \label{eq:reg-framework}\\[2pt]
  \KL\!\bigl(\delta_x \widehat P\,\|\,\delta_y P\bigr)
  &\;\le\; c'\,\|x-y\|^2 + b(x)^2,
  \label{eq:crossreg-framework}
\end{align}
for constants $c,c'\ge0$ and a measurable function
$b:\R^d\to[0,\infty)$. We refer to~\eqref{eq:reg-framework} as the
\emph{regularity} of the SDE kernel and
to~\eqref{eq:crossreg-framework} as the \emph{cross-regularity} of
the scheme. For \ref{eq:kTULA}, both are established in
Propositions~\ref{prop:regularity-app}--\ref{prop:cross-regularity-app}.

\begin{lemma}[KL telescoping bound]
\label{lem:KL-telescoping}
Let $(Y_n')_{n=0}^N$ and $(\widetilde Y_n)_{n=0}^{N-1}$ be defined
by~\eqref{eq:tildeY-def}--\eqref{eq:aux-update}. Then
\begin{equation}\label{eq:KL-telescope-1}
  \KL\!\bigl(\widehat\mu_N\,\|\,\nu_N\bigr)
  \;\le\;
  \sum_{n=0}^{N-1}
  \E\!\left[
    \KL\!\bigl(Q_n(\widetilde Y_n,\cdot)\,\big\|\,P(Y_n',\cdot)\bigr)
  \right].
\end{equation}
If~\eqref{eq:reg-framework}--\eqref{eq:crossreg-framework} hold, then
\begin{equation}\label{eq:KL-telescope-2}
  \KL\!\bigl(\widehat\mu_N\,\|\,\nu_N\bigr)
  \;\le\;
  c\,\sum_{n=0}^{N-2}\eta_n^{2}\,\E\|\widehat X_n-Y_n'\|^{2}
  +c'\,\E\|\widehat X_{N-1}-Y_{N-1}'\|^{2}
  +\bar b^{\,2},
\end{equation}
where $\bar b^{\,2}:=\|b\|_{L^2(\widehat\mu_{N-1})}^{2}$.
\end{lemma}

% =====================================================================
\subsection{Shifted Wasserstein recursion}
\label{subsec:framework-recursion}

The right-hand side of~\eqref{eq:KL-telescope-2} is controlled entirely
by the squared coupling distances
\begin{equation}\label{eq:dn-def}
  d_n^{\,2} \;:=\; \E\|\widehat X_n - Y_n'\|^{2},
  \qquad n=0,1,\dots,N-1,
\end{equation}
which propagate forward in time according to the next lemma. This is
the only place in the analysis where the state-dependent coefficient
$a_1$ enters.

\begin{lemma}[Shifted Wasserstein recursion]
\label{lem:shifted-recursion}
Assume~\eqref{eq:one-step-w2-framework}. For all $0\le n\le N-2$,
\begin{equation}\label{eq:dn-recursion}
  d_{n+1}^{\,2}
  \;\le\;
  L^{2}(1-\eta_n)^{2}\,d_n^{\,2}
  +2\,(1-\eta_n)\,\overline a_1\,d_n
  +\overline a_0^{\,2},
\end{equation}
where
$\overline a_0:=\sup_{0\le k<N}\|a_0(\widehat X_k)\|_{L^{2}}$
and
$\overline a_1:=\sup_{0\le k<N}\|a_1(\widehat X_k,\widetilde Y_k)\|_{L^{2}}$.
\end{lemma}

\begin{proof}
By the Wasserstein lift of the coupling
and~\eqref{eq:one-step-w2-framework},
\[
  \W^{2}(\widehat\mu_{n+1},\nu_{n+1}')
  \;\le\;
  L^{2}\,\E\|\widehat X_n-\widetilde Y_n\|^{2}
  +2\,\E\!\left[a_1(\widehat X_n,\widetilde Y_n)\,\|\widehat X_n-\widetilde Y_n\|\right]
  +\E\!\left[a_0(\widehat X_n)^{2}\right].
\]
Since $\widehat X_n-\widetilde Y_n=(1-\eta_n)(\widehat X_n-Y_n')$, the
first term equals $L^2(1-\eta_n)^2 d_n^2$, Cauchy--Schwarz bounds the
cross term by $2(1-\eta_n)\overline a_1 d_n$, and the third term
by $\overline a_0^{\,2}$. Choosing the synchronous Wasserstein coupling
on the next step yields~\eqref{eq:dn-recursion}.
\end{proof}

\begin{remark}[Effect of state-dependence on $\overline a_1$]
\label{rem:state-dependence-effect}
The supremum defining $\overline a_1$ involves the joint law of
$(\widehat X_k,\widetilde Y_k)$, hence requires uniform moment bounds on
both $(\widehat X_k)$ and the shifted process $(\widetilde Y_k)$. Since
$\widetilde Y_k$ is a convex combination of $Y_k'$ and $\widehat X_k$,
controlling its moments reduces to controlling those of $Y_k'$, which
we establish in
Lemma~\ref{lem:tildeY-moment} via the dissipativity condition~\ref{ass:D} on the drift of~\eqref{eq:SDE}. This is the only place where the
polynomial-Lipschitz nature of the drift plays a non-trivial role at
the level of the framework.
\end{remark}

The recursion~\eqref{eq:dn-recursion} is purely deterministic once
$\overline a_0,\overline a_1,L$ have been fixed. Its analysis is
identical to that of the corresponding recursion
in~\cite[Lemma~B.5]{chewi2024local}, and yields the following optimal
control of the shifted distances.

\begin{proposition}[Deterministic control of the shifted distances]
\label{prop:deterministic-dn}
Let $L\in[\tfrac12,2]$, $\overline a_0,\overline a_1\ge0$, and $\bar N := N\wedge 1/(1-L)_{+}$.
There exists a sequence
$(\eta_n)_{n=0}^{N-2}\subset[0,1]$ with $\eta_{N-1}=1$ such that any
nonnegative sequence $(d_n)_{n=0}^{N-1}$
satisfying~\eqref{eq:dn-recursion} and $d_0=\W(\mu,\nu)$ obeys
\begin{align}
  d_{N-1}^{\,2}
  &\;\lesssim\;
  \tfrac{L^{-1}-1}{L^{-N}-1}\,d_0^{\,2}
  +\bigl(((L-1)N)\vee\log\bar N\bigr)\,\overline a_0^{\,2}
  +\bar N\,\overline a_1^{\,2},
  \label{eq:dn-terminal-bound}\\[2pt]
  \sum_{n=0}^{N-2}\eta_n^{2}\,d_n^{\,2}
  &\;\lesssim\;
  \tfrac{L^{-1}-1}{L^{-N}-1}\,d_0^{\,2}
  +\bigl(((L-1)N)\vee\log\bar N\bigr)\,\overline a_0^{\,2}
  +\bar N\,\overline a_1^{\,2}.
  \label{eq:dn-sum-bound}
\end{align}
\end{proposition}

\begin{proof}
The recursion~\eqref{eq:dn-recursion} has the same form as the
recursion analyzed in~\cite[Appendix~B.2]{chewi2024local}, with
deterministic coefficients $\overline a_0,\overline a_1$ in place of
the local errors used there. The optimal shift sequence and the dynamic
programming argument of~\cite[Lemmas~B.4--B.6]{chewi2024local} apply
verbatim and yield~\eqref{eq:dn-terminal-bound}--\eqref{eq:dn-sum-bound}.
\end{proof}

% =====================================================================
\subsection{The KL local-error bound for kTULA}
\label{subsec:framework-KL-bound}

Combining Lemma~\ref{lem:KL-telescoping} with
Proposition~\ref{prop:deterministic-dn} yields the main KL bound of the
framework.

\begin{theorem}[KL local-error framework]
\label{thm:KL-local-error-framework}
Assume~\eqref{eq:one-step-w2-framework},~\eqref{eq:reg-framework}, and
\eqref{eq:crossreg-framework}, and let $L\in[\tfrac12,2]$. With
$\overline a_0,\overline a_1$ as in Lemma~\ref{lem:shifted-recursion},
$\bar b := \max_{0\le n<N}\|b\|_{L^2(\widehat\mu_n)}$, and
$\bar N=N\wedge 1/(1-L)_{+}$,
\begin{equation}\label{eq:kl-blackbox}
  \KL\!\bigl(\mu\widehat P^{N}\,\|\,\nu P^{N}\bigr)
  \;\lesssim\;
  (c+c')\!\left[
    \tfrac{L^{-1}-1}{L^{-N}-1}\,\W^{2}(\mu,\nu)
    +\bigl(((L-1)N)\vee\log\bar N\bigr)\,\overline a_0^{\,2}
    +\bar N\,\overline a_1^{\,2}
  \right]
  +\bar b^{\,2}.
\end{equation}
\end{theorem}

We emphasize that the bound~\eqref{eq:kl-blackbox} compares the discretization $\mu\widehat P^N$ against the law $\nu P^N$ of the ~\eqref{eq:SDE}; no invariant-measure or gradient structure is invoked at this stage. The specialization to sampling toward $\pi_\beta$ proceeds via the additional ergodicity input provided by Assumption~\ref{ass:LSI}, as carried out in Section~\ref{sec:main-results}.

% =====================================================================
\subsection{A TV bound for the randomized scheme tRLMC}
\label{subsec:framework-tv}

For the tamed randomized midpoint scheme \ref{eq:tamed-rLMC}, the KL
framework above is not directly applicable: the randomization of the
integration time prevents the natural construction of an adapted
continuous-time interpolation, and consequently a sharp cross-regularity
estimate of the form~\eqref{eq:crossreg-framework} is not currently
available. We therefore work in the weaker total variation metric,
where a direct one-step estimate is at hand.

Let $\widetilde P$ denote the one-step kernel
of~\ref{eq:tamed-rLMC}, and write
$\widetilde\mu_n:=\mu\widetilde P^{\,n}$. We assume, in addition to
the regularity~\eqref{eq:reg-framework}, the existence of a one-step
TV local error: there exists a measurable function
$C_{\mathrm{TV}}\colon\R^d\to[0,\infty)$ such that, for all $x\in\R^d$,
\begin{equation}\label{eq:tv-local-error}
  TV\!\bigl(\delta_x\widetilde P,\delta_x P\bigr)
  \;\le\;
  C_{\mathrm{TV}}(x)\,\lambda.
\end{equation}
For \ref{eq:tamed-rLMC}, this estimate is established in
Proposition~\ref{prop:cross-tv-trlmc}.

The auxiliary process construction of~\S\ref{subsec:framework-aux}
remains valid with $\widehat P$ replaced by $\widetilde P$ throughout.

\begin{lemma}[TV bound via the shifted construction]
\label{lem:tv-shifted}
Assume~\eqref{eq:reg-framework} and~\eqref{eq:tv-local-error}, and let
$d_n^{\,2}:=\E\|\widetilde X_n-Y_n'\|^{2}$. Writing
$\overline C_{\mathrm{TV}}:=\sup_{0\le n<N}\|C_{\mathrm{TV}}(x)\|_{L^2(\widetilde\mu_n)}$,
\begin{equation}\label{eq:tv-shifted-bound1}
  TV\!\bigl(\widetilde\mu_N,\nu_N\bigr)
  \;\le\;
  \overline C_{\mathrm{TV}}\,\lambda
  +\sqrt{\tfrac{c}{2}}\,d_{N-1}
  +\sqrt{\tfrac{c}{2}\,\sum_{n=0}^{N-2}\eta_n^{2}\,d_n^{\,2}}.
\end{equation}
\end{lemma}

\begin{theorem}[TV local-error framework]
\label{thm:tv-local-framework}
Assume~\eqref{eq:one-step-w2-framework},~\eqref{eq:reg-framework}, and
\eqref{eq:tv-local-error}, and let $L\in[\tfrac12,2]$. With
$\overline a_0,\overline a_1$ as in
Lemma~\ref{lem:shifted-recursion} (with $\widehat P$ replaced by
$\widetilde P$) and $\bar N=N\wedge 1/(1-L)_{+}$,
\begin{equation}\label{eq:tv-local-framework-bound}
  TV\!\bigl(\mu\widetilde P^{N},\nu P^{N}\bigr)
  \;\lesssim\;
  \overline C_{\mathrm{TV}}\,\lambda
  +\sqrt{
    \tfrac{L^{-1}-1}{L^{-N}-1}\,\W^{2}(\mu,\nu)
    +\bigl(((L-1)N)\vee\log\bar N\bigr)\,\overline a_0^{\,2}
    +\bar N\,\overline a_1^{\,2}
  }.
\end{equation}
\end{theorem}

The proof combines Lemma~\ref{lem:tv-shifted} with
Proposition~\ref{prop:deterministic-dn}. The KL and TV frameworks rest
on the same shifted Wasserstein recursion and differ only at the
terminal step: the KL framework requires the cross-regularity
estimate~\eqref{eq:crossreg-framework}, whereas the TV framework only
requires~\eqref{eq:tv-local-error}, making
Theorem~\ref{thm:tv-local-framework} applicable to schemes such as
\ref{eq:tamed-rLMC} for which a sharp cross-regularity estimate is currently
out of reach. As with Theorem~\ref{thm:KL-local-error-framework}, the bound~\eqref{eq:tv-local-framework-bound} is stated against the law $\nu P^N$ of~\eqref{eq:SDE} and requires no gradient structure on the drift.

% =====================================================================
\subsection{Wasserstein convergence of tRLMC}
\label{subsec:framework-W2}

Wasserstein convergence for \ref{eq:tamed-rLMC} avoids the auxiliary process
construction and proceeds via a more direct route based on the
contraction of the underlying semigroup toward $\pi_\beta$. Unlike the local-error frameworks of~\S\ref{subsec:framework-setup}--\S\ref{subsec:framework-tv}, this argument is genuinely sampling-specific: throughout the present subsection we therefore assume the specialization $h=\nabla u$, $\sigma=\sqrt{2/\beta}$ of~\eqref{eq:SDE} to the ~\eqref{eq:LSDE}, and we additionally impose Assumption~\ref{ass:LSI}.

Under this specialization, the local Lipschitz property of
the drift, combined with the one-sided Lipschitz
condition~\ref{ass:OSL}, yields a one-step Wasserstein bound of the
form~\eqref{eq:one-step-w2-framework} with
$\Gamma(x,y,\lambda)=C_{\mathrm{cpl}}\lambda(1+\|x\|^{2\ell}+\|y\|^{2\ell})$.
Iterating this estimate over $n_0$ steps and exploiting the uniform
moment bounds of
Lemma~\ref{lem:moment-tamed-rLMC} furnishes the long-time
discretization-error estimate
\begin{equation}\label{eq:W2-discretization}
  \W\!\left(\mu_0\widetilde P^{\,n},\;\bigl(\mu_0\widetilde P^{\,n-n_0}\bigr)P^{n_0}\right)
  \;\le\;
  C\,e^{c\lambda n_0}\,\lambda,
  \qquad
  n\ge n_0\ge0,
\end{equation}
where $C,c$ depend polynomially on the dimension. In parallel, the
combination of the logarithmic Sobolev inequality~\ref{ass:LSI} and the one-sided
Lipschitz condition~\ref{ass:OSL} implies the contraction
\begin{equation}\label{eq:W2-contraction}
  \W\!\bigl(\nu P^{t},\pi_\beta\bigr)
  \;\le\;
  C_{\mathrm{contr}}\,e^{-\dot c\,t}\,\W\!\bigl(\nu,\pi_\beta\bigr),
  \qquad
  t\ge0,
\end{equation}
for some $\dot c>0$ depending on $C_{\mathrm{LSI}}$. Combining
\eqref{eq:W2-discretization} and~\eqref{eq:W2-contraction} via the
triangle inequality
\[
  \W\!\bigl(\mu_0\widetilde P^{\,n},\pi_\beta\bigr)
  \;\le\;
  \W\!\Bigl(\mu_0\widetilde P^{\,n},\,
            \bigl(\mu_0\widetilde P^{\,n-n_0}\bigr)P^{n_0}\Bigr)
  +\W\!\Bigl(\bigl(\mu_0\widetilde P^{\,n-n_0}\bigr)P^{n_0},\pi_\beta\Bigr),
\]
and balancing the two contributions by an appropriate choice of $n_0$
yields the announced Wasserstein convergence estimate of
Theorem~\ref{thm:wass conv}.

%% file: Lemmas.tex
% ============================================================
% Appendices
% ============================================================

\appendix
\numberwithin{theorem}{section}

\newcommand{\hl}{h_\lambda}

% ============================================================
% APPENDIX A: Drift Properties and Moment Estimates
% ============================================================
\input{appendix/Drift-Moments}

% ============================================================
% APPENDIX B: Local Error Tools and Regularity
% ============================================================
\input{appendix/Local-Errors}

% ============================================================
% APPENDIX C: Auxiliary Results for the Main Theorems
% ============================================================
\input{appendix/Auxiliary-Results}

%% file: appendix/Drift-Moments.tex
\section{Drift Properties and Moment Estimates}
\label{app:drift-moments}

This appendix collects technical estimates underlying the analysis. We record
structural properties of the tamed drift, including dissipativity, growth, and
global Lipschitz bounds, and establish uniform moment bounds for both the
underlying ~\ref{eq:LSDE} and the kTULA iterates. We also include a local
Taylor remainder estimate under polynomially growing Jacobian regularity, which
is used in the control of one-step discretization errors. Proofs are deferred
to Section~\ref{sec:proof-section} unless otherwise stated.

\subsection{Properties of the tamed drift}
\begin{lemma}[Key properties of the tamed drift]\label{lem:properties-hlambda}
Assume that \hyperref[ass:PJG]{\textnormal{(A1)}}, \hyperref[ass:PLC]{\textnormal{(A2)}},
and \hyperref[ass:D]{\textnormal{(A4)}} hold.
Then for any $\theta,\bar\theta\in\mathbb{R}^d$ and any $0<\lambda<1$, the following statements hold.
\begin{enumerate}
\item\label{it:hlam-dissipativity}
\emph{Dissipativity.}
The tamed drift satisfies
\[
  \langle h_\lambda(\theta),\theta\rangle
  \ge
  a\|\theta\|^2 - b .
\]

\item\label{it:hlam-growth}
\emph{Growth control.}
The magnitude of $h_\lambda$ satisfies
\[
\|h_\lambda(\theta)\|
\le
2a\|\theta\| + 2L\,\lambda^{-1/2},
\qquad
\|h_\lambda(\theta)\|
\le
(2a+L)\bigl(1+\|\theta\|^{\ell+1}\bigr).
\]

\item\label{it:hlam-lipschitz}
\emph{Lipschitz continuity.}
There exists a constant $L_0>0$, depending only on
$a,L,\ell$, such that
\[
\|h_\lambda(\theta)-h_\lambda(\bar\theta)\|
\le
L_0\,\lambda^{-1/2}\,\|\theta-\bar\theta\|,
\]
where
\[
L_0 := 2a + 4L + (\ell+1)(2L+a).
\]

\item\label{it:hlam-taming-error}
\emph{Taming error.}
The deviation between the original and tamed drifts satisfies
\[
  \|h(\theta)-h_\lambda(\theta)\|^2
  \le
  4\lambda^2 (L+a)^2
  \bigl(1+\|\theta\|^{\,6(\ell+1)}\bigr).
\]
\end{enumerate}
\end{lemma}

\begin{proof}
The result follows from \cite[Lemma 4.1]{lytras2025ktula} specialized to $\varepsilon_h = \frac{1}{2}$.
\end{proof}

\subsection{Moment bounds for the underlying SDE}
\begin{lemma}[Uniform $2p$-moment bounds for \ref{eq:LSDE}]
\label{lem:uniform-moments-sde}
Let Assumptions~\hyperref[ass:PJG]{\textup{(A1)}}--\hyperref[ass:D]{\textup{(A4)}} hold,
and let $(X_t)_{t\ge0}$ solve~\ref{eq:LSDE} with $X_0=x$.
Then for any $p\in\mathbb{N}$ there exist constants $c_p>0$ and $C_p>0$ such that
for all $t\ge0$,
\[
  \E\|X_t\|^{2p}
  \;\le\;
  e^{-c_p t}\,\|x\|^{2p}
  \;+\;
  C_p,
\]
where one can take
\[
  c_p = p a,
  \qquad
  C_p = C\Bigl(b + \tfrac{1}{\beta}(d+2(p-1))\Bigr)^p
\]
for a constant $C>0$ depending only on $p$.
\end{lemma}

\begin{proof}
The proof follows from \cite[Lemma A.1]{lim2023nonasymptoticestimatestuslaalgorithm}.
\end{proof}

\begin{lemma}[One-step \(p\)-moment drift for \ref{eq:LSDE}]
\label{lem:P-moment-drift}
Fix \(p\ge2\). Under Assumption~\hyperref[ass:D]{\textup{(A4)}}, there exist constants
\(q>0\) and \(C_p>0\), depending only on \(p\), \(d\), \(a\), \(b\), and \(\beta\), such that for every \(\lambda\in(0,1]\) and every
\(x\in\mathbb R^d\),
\[
\E\|X(t,x;t+\lambda)\|^p
\le
e^{-q\lambda}\|x\|^p + C_p\lambda,
\]
where $X(t,x;\cdot)$ denotes the solution of \eqref{eq:LSDE} started from $x$ at time $t$.
\end{lemma}

\begin{proof}
The proof is given in~\ref{proof:P-moment-drift-proof}.
\end{proof}

\subsection{Moment bounds for the kTULA scheme}

We now establish uniform-in-time moment bounds for the kTULA chain.

\begin{lemma}[Uniform moments for the \texorpdfstring{\eqref{eq:kTULA}}{kTULA} scheme]
\label{lem:moment-algorithm}
Let Assumptions~\hyperref[ass:PJG]{(A1)}--\hyperref[ass:D]{(A4)} hold.
There exists $\lambda_{\max}>0$ such that for all $0<\lambda\le\lambda_{\max}$ the following holds.
Let $(\widehat X_n)_{n\ge0}$ denote the kTULA iterates and let $\widehat X_t$ be any continuous
interpolation on $[n\lambda,(n+1)\lambda]$ satisfying $\widehat X_{n\lambda}=\widehat X_n$.
Then for every $p\ge2$ there exists a constant $c_p<\infty$ such that for all $t\in[n\lambda,(n+1)\lambda]$,
\[
  \E\|\widehat X_t\|^{2p}
  \;\le\;
  (1-a(t-n\lambda))(1-a\lambda)^n\,\E\|\widehat X_0\|^{2p}
  + c_p\Bigl(1+\frac{1}{a}\Bigr).
\]
In particular,
\[
  \sup_{t\ge0}\E\|\widehat X_t\|^{2p}
  \;\le\;
  \E\|\widehat X_0\|^{2p} + c_p\Bigl(1+\frac{1}{a}\Bigr)
  <\infty.
\]
\end{lemma}

\begin{proof}
The result follows from \cite[Lemma 4.2]{lytras2025ktula}.
\end{proof}

\subsection{A local smoothness remainder bound}

We close this appendix with an elementary Taylor remainder estimate that will be invoked
repeatedly when controlling one-step discretization errors under polynomially growing local
smoothness of the Jacobian.

\begin{lemma}[Taylor remainder under locally Lipschitz Jacobian]
\label{lem:local-jacobian}
Let $h:\R^d\to\R^d$ be continuously differentiable.
Under Assumption~\hyperref[ass:JLC]{\textup{(A5)}}, for all $x,y\in\R^d$,
\[
 \big\|h(x)-h(y)-J(h)(x)(x-y)\big\|
 \le \frac{L''}{2}\,(1+\|x\|+\|y\|)^{\ell''}\,\|x-y\|^{2}.
\]
\end{lemma}

\begin{proof}
A Taylor expansion of $h\bigl(y+t(x-y)\bigr)$ for $t\in[0,1]$, combined with \ref{ass:JLC} to bound the increment of the Jacobian $J(h)$, yields the claim.
\end{proof}

%% file: appendix/Local-Errors.tex
\section{Local Error Tools and Regularity}
\label{app:local-error}
This appendix collects auxiliary estimates for the underlying ~\eqref{eq:SDE} that are
used to verify the local-error framework. In particular, we establish
$W_2$-Lipschitz continuity, derive a coupling increment bound, and record the
resulting one-step Wasserstein estimate together with a KL regularity bound.

\begin{proposition}[$W_2$--Lipschitz continuity]
\label{prop:W2-Lip-app}
Assume \hyperref[ass:OSL]{\textup{(A3)}}. Let \((X_t)_{t\ge0}\) and \((Y_t)_{t\ge0}\) be
solutions of \eqref{eq:SDE} driven by the same Brownian motion, with initial conditions
\(X_0=x\) and \(Y_0=y\).
Then for all \(t\ge0\),
\[
  \|X_t-Y_t\|_{L^2}\le e^{K' t}\,\|x-y\|.
\]
In particular, for \(t\in[0,\lambda]\), setting \(L := e^{K'\lambda}\) yields
\[
  \|X_t-Y_t\|_{L^2}\le L\,\|x-y\|.
\]
For \(\lambda\le \frac{\ln 2}{|K'|}\), then \(L\in[\tfrac12,2]\).
\end{proposition}
\begin{proof}
    The proof is given in Subsection \ref{proof:W2-Lip-app}.
\end{proof}

\begin{proposition}[Coupling increment]
\label{prop:coupling-increment-app}
Assume \hyperref[ass:PJG]{\textup{(A1)}}--\hyperref[ass:OSL]{\textup{(A3)}}. There exists a constant \(C_{\mathrm{cpl}}>0\), depending on $L'$, $\ell'$, and $\beta$, such that for all
\(t\in[0,\lambda]\), with \(\lambda\le \frac{\ln 2}{|K'|}\),
\[
  \|X_t - x - (Y_t - y)\|_{L^2}
  \;\le\;
  \Gamma(x,y,t)\,\|x-y\|,
\]
where $(X_t)_{t\ge0}$, $(Y_t)_{t\ge0}$ are solutions of \eqref{eq:LSDE} driven by the same Brownian motion with $X_0=x$, $Y_0=y$, and
\[
  \Gamma(x,y,t)
  := C_{\mathrm{cpl}}\,t\,\bigl(1+\|x\|^{\ell'}+\|y\|^{\ell'}\bigr).
\]
In particular, taking \(t=\lambda\) yields
\[
  \Gamma(x,y,\lambda)
  := C_{\mathrm{cpl}}\,\lambda\,\bigl(1+\|x\|^{\ell'}+\|y\|^{\ell'}\bigr).
\]
\end{proposition}
\begin{proof}
The proof is given in Subsection \ref{proof:coupling-increment}.
\end{proof}

\begin{lemma}[One-step Wasserstein inequality]\label{lem:one-step-w2}
Let $L>0$ be the $W_2$-Lipschitz constant from Proposition~\ref{prop:W2-Lip-app}, let $\Gamma(x,y,\lambda)$ be the coupling increment from Proposition~\ref{prop:coupling-increment-app}, and let $E_{\mathrm{weak}},E_{\mathrm{strong}}:\R^d\to[0,\infty)$ denote the weak and strong one-step local errors of \eqref{eq:kTULA} (Propositions~\ref{prop:weak-error-app}--\ref{prop:strong-error-app}). Then, for all $x,y\in\R^d$,
\[
  \W^2(\delta_x \widehat P,\delta_y P)
  \le
  L^2\|x-y\|^2
  +2\Bigl(E_{\mathrm{weak}}(x)+\Gamma(x,y,\lambda)\,E_{\mathrm{strong}}(x)\Bigr)\|x-y\|
  +E_{\mathrm{strong}}(x)^2.
\]
\end{lemma}
\begin{proof}
The result follows by the same argument as in \cite{chewi2024local}, with the only difference that here the coupling increment $\Gamma(x,y,\lambda)$ depends on both endpoints rather than only on $x$, leading to the state-dependent coefficient.
\end{proof}

\begin{remark}
With the notation
\[
a_0(x):=E_{\mathrm{strong}}(x),
\qquad
a_1(x,y):=E_{\mathrm{weak}}(x)+\Gamma(x,y,\lambda)E_{\mathrm{strong}}(x),
\]
the one-step Wasserstein estimate takes the form
\[
  \W^2(\delta_x \widehat P,\delta_y P)
  \le
  L^2\|x-y\|^2
  +2a_1(x,y)\|x-y\|
  +a_0(x)^2.
\]
\end{remark}

\begin{proposition}[Regularity {\citep[Theorem~1.1]{wang2011harnack}}]
\label{prop:regularity-app}
Let $P=P_\lambda$ denote the one-step kernel of \eqref{eq:SDE}. There exists a constant $C_{\mathrm{reg}}>0$, depending on $\beta$, such that for all $x,y\in\R^d$ and all $t \in (0,\lambda]$,
\[
  \KL\!\bigl(\delta_x P_t \,\big\|\, \delta_y P_t\bigr)
  \;\le\;  R_2(\delta_x P_t,\delta_y P_t) \;\le\;
  \frac{C_{\mathrm{reg}}}{1-e^{-K't}}\,\|x-y\|^2.
\]
In particular, if \(K'\lambda\le 2\), then at \(t=\lambda\) one may take
\[
  c=\frac{3C_{\mathrm{reg}}}{\lambda}.
\]
\end{proposition}

%% file: appendix/Auxiliary-Results.tex
\newcommand{\TV}{\mathrm{TV}}

\section{Auxiliary Results for the Main Theorems}
\label{app:auxiliary-results}

This appendix collects auxiliary results used in the proofs of the main convergence theorems. We record the strong and weak local error bounds, cross-regularity estimates, and shifted moment controls needed for kTULA, together with the corresponding moment, local-error, total variation, and Wasserstein tools for the tamed randomized scheme.

\subsection{Strong and weak local errors of kTULA}

\begin{proposition}[Weak local error]
\label{prop:weak-error-app}
Let Assumptions~\hyperref[ass:PJG]{\textup{(A1)}}--\hyperref[ass:JLC]{\textup{(A5)}} hold.
Let \(X_t\) denote the solution of \eqref{eq:LSDE} started from \(X_0=x\), and
\(\widehat X_t\) the one-step interpolation associated with \eqref{eq:kTULA} started from \(\widehat X_0=x\).
Set \(M := \max\{3(\ell+1),\,4\ell+\ell''\}\).
Then there exists a constant \(C>0\), independent of \(x\in\R^d\), \(d\), and \(\lambda\in(0,1]\), such that for every \(x\in\R^d\),
\[
  \big\|\E[\widehat X_\lambda] - \E[X_\lambda]\big\|
  \;\le\;
  C\bigl(1+d^{M/2}+\|x\|^{M}\bigr)\,\lambda^{2}.
\]
In particular,
\[
  E_{\mathrm{weak}}(x)
  :=
  C\bigl(1+d^{M/2}+\|x\|^{M}\bigr)\,\lambda^{2}.
\]
\end{proposition}

\begin{proof}
The proof is given in Subsection~\ref{proof:weak-error}.
\end{proof}

\begin{proposition}[Strong local error]
\label{prop:strong-error-app}
Let Assumptions~\hyperref[ass:PJG]{\textup{(A1)}}--\hyperref[ass:D]{\textup{(A4)}} hold.
Let \(X_t\) denote the solution of \eqref{eq:LSDE} started from \(X_0=x\), and
\(\widehat X_t\) the one-step interpolation associated with \eqref{eq:kTULA} started from \(\widehat X_0=x\).
Set \(M := \max\{3(\ell+1),\,2\ell'+4\ell\}\).
Then there exists a constant \(C>0\), independent of \(x\in\R^d\), \(d\), and \(\lambda\in(0,1]\), such that for every \(x\in\R^d\),
\[
  \|\widehat X_\lambda - X_\lambda\|_{L^2}
  \;\le\;
  C\bigl(1+d^{M/2}+\|x\|^{M}\bigr)\,\lambda^{3/2}.
\]
In particular,
\[
  E_{\mathrm{strong}}(x)
  :=
  C\bigl(1+d^{M/2}+\|x\|^{M}\bigr)\,\lambda^{3/2}.
\]
\end{proposition}

\begin{proof}
The proof is given in Subsection~\ref{proof:strong-error}.
\end{proof}

\subsection{Cross-regularity for kTULA}

\begin{proposition}[Cross-regularity]
\label{prop:cross-regularity-app}
Assume \hyperref[ass:PJG]{\textup{(A1)}}--\hyperref[ass:D]{\textup{(A4)}}. There exist a constant \(C_{\mathrm{cross}}>0\) and a measurable function
\(b:\R^d\to[0,\infty)\) such that, for all \(x,y\in\R^d\) and all \(t\in(0,\lambda]\),
\[
  \KL\!\bigl(\delta_x \widehat P \,\big\|\, \delta_y P_t\bigr)
  \;\le\;
  \frac{C_{\mathrm{cross}}}{1-e^{-K't}}\,\|x-y\|^2
  \;+\;
  b(x)^2,
\]
where
\[
  b(x)^2 \le C\,\lambda^2\bigl(1+\|x\|^{2\ell+2}\bigr).
\]
In particular, taking \(t=\lambda\), one may take
\[
  c'=\frac{C_{\mathrm{cross}}}{\lambda}.
\]
\end{proposition}

\begin{proof}
The proof is given in Subsection~\ref{proof:cross-regularity}.
\end{proof}

\subsection{Shifted recursion and auxiliary moment control for kTULA}

\begin{lemma}[Uniform moment bounds for the shifted process]\label{lem:tildeY-moment}
Let $p\ge 1$ and define the shifted process $\widetilde Y_n := (1-\eta_n)Y_n' + \eta_n \widehat X_n$ for $n\ge0$. Under the hypotheses of Lemma~\ref{lem:P-moment-drift} (with constants $q, C > 0$), and assuming $\widehat M_p := \sup_{n\ge 0} \E\|\widehat X_n\|^p < \infty$, we have
\[
  \E\|\widetilde Y_n\|^p \le e^{-q\lambda n}\E\|\widetilde Y_0\|^p + \frac{C\lambda+\widehat M_p}{1-e^{-q\lambda}},
\]
yielding the uniform bound $\sup_{n\ge 0} \E\|\widetilde Y_n\|^p \le \E\|\widetilde Y_0\|^p + \frac{C\lambda+\widehat M_p}{1-e^{-q\lambda}} < \infty$.
\end{lemma}

\begin{proof}
The proof is given in Subsection~\ref{subsec:proof-tildeY-moment}.
\end{proof}

\begin{lemma}[Shifted distance recursion]
\label{lem:distance-recursion-app}
Let $E_{\mathrm{weak}}$, $E_{\mathrm{strong}}$, and $\Gamma$ denote the weak-error,
strong-error, and coupling-increment functions of
Propositions~\ref{prop:weak-error-app}, \ref{prop:strong-error-app},
and~\ref{prop:coupling-increment-app}, and set
\[
  a_0(x):=E_{\mathrm{strong}}(x),
  \qquad
  a_1(x,y):=E_{\mathrm{weak}}(x)+\Gamma(x,y,\lambda)\,E_{\mathrm{strong}}(x).
\]
For $0\le n\le N-1$, let $\eta_n\in[0,1]$,
$\widetilde Y_n:=(1-\eta_n)Y_n'+\eta_n\widehat X_n$, and
$d_n^2:=\E\|\widehat X_n-Y_n'\|^2$, and put
\[
  \overline a_0
  :=\sup_{0\le k<N}\|a_0(\widehat X_k)\|_{L^2},
  \qquad
  \overline a_1
  :=\sup_{0\le k<N}\|a_1(\widehat X_k,\widetilde Y_k)\|_{L^2}.
\]
Then, with $L:=e^{K'\lambda}$, for every $0\le n\le N-2$,
\[
  d_{n+1}^2
  \le
  L^2(1-\eta_n)^2\,d_n^2
  +2(1-\eta_n)\,\overline a_1\,d_n
  +\overline a_0^{\,2}.
\]
\end{lemma}

\begin{proof}
The proof is given in Subsection~\ref{subsec:proof-distance-recursion}.
\end{proof}

\subsection{Auxiliary moment bounds for tRLMC}

\begin{lemma}[Uniform moments for tRLMC]
\label{lem:moment-tamed-rLMC}
Assume \hyperref[ass:PJG]{\textup{(A1)}}, \hyperref[ass:PLC]{\textup{(A2)}}, and
\hyperref[ass:D]{\textup{(A4)}}. Let \((\bar Y_n)_{n\ge0}\) denote the tamed randomized
Langevin Monte Carlo scheme \eqref{eq:tamed-rLMC}, initialized at
\(\bar Y_0=x_0\in\R^d\). Then the following statements hold.

\begin{enumerate}
\item
There exist constants \(\lambda_{2,\max}^{tRLMC}>0\), \(\mu>0\), and \(M_2>0\),
depending only on \(a,b,L,L_0,\beta\), such that for every
\(\lambda\in(0,\lambda_{2,\max}^{tRLMC}]\),
\[
\E\|\bar Y_n\|^2
\le
e^{-\mu n\lambda}\,\E\|x_0\|^2 + M_2\,d,
\qquad n\in\mathbb N_0.
\]

\item
For every \(p\in[2,4\ell]\cap\mathbb N\), there exist constants
\(\lambda_{p,\max}^{tRLMC}>0\), \(c_p>0\), and \(\widetilde C_p>0\), depending only on
\(p,a,b,L,L_0,\beta\), such that for every
\(\lambda\in(0,\lambda_{p,\max}^{tRLMC}]\),
\[
\E\|\bar Y_n\|^{2p}
\le
(1-c_p\lambda)^n\,\E\|x_0\|^{2p}
+
\widetilde C_p(1+d^p),
\qquad n\in\mathbb N_0.
\]
In particular,
\[
\sup_{n\ge0}\E\|\bar Y_n\|^{2p}<\infty.
\]
\end{enumerate}
\end{lemma}

\begin{proof}
The proof of \textup{(i)} is given in Subsection~\ref{subsec:proof-uniform-moments-tamed-rLMC},
and the proof of \textup{(ii)} is given in Subsection~\ref{subsec:proof-uniform-p-moments-tamed-rLMC}.
\end{proof}

\begin{lemma}[One-step \(p\)-moment drift for RLMC]
\label{lem:P-moment-drift-RLMC}
Fix \(p\ge 2\). Under Assumption~\hyperref[ass:D]{\textup{(A4)}}, there exist constants
\(q>0\) and \(C_p>0\), depending only on \(p\), \(d\), \(a\), \(b\), and \(\beta\), such that for all
\(\lambda\in(0,1]\),
\[
  \E\|Y'_{n+1}\|^p
  \le
  e^{-q\lambda}\,\E\|\widetilde Y_n\|^p
  +
  C_p\lambda,
\]
where
\[
  Y'_{n+1}\sim P_\lambda(\widetilde Y_n,\cdot),
  \qquad
  \widetilde Y_n=(1-\eta_n)Y_n' + \eta_n\bar Y_n,
  \qquad
  \eta_n\in[0,1].
\]
\end{lemma}

\begin{proof}
The proof is given in Subsection~\ref{subsec:proof-P-moment-drift-RLMC}.
\end{proof}

\begin{lemma}[Uniform moment bounds for the shifted process under tRLMC]\label{lem:tildeY-moment-tamed-rLMC}
Let $p\ge 2$ and define the shifted process $\widetilde Y_n := (1-\eta_n)Y_n' + \eta_n \bar Y_n$ for $n\ge0$. Under the hypotheses of Lemma~\ref{lem:P-moment-drift-RLMC} (with constants $q, C_p > 0$), and assuming $\bar M_p := \sup_{n\ge 0} \E\|\bar Y_n\|^p < \infty$, we have
\[
  \E\|\widetilde Y_n\|^p \le e^{-q\lambda n}\E\|\widetilde Y_0\|^p + \frac{C_p\lambda+\bar M_p}{1-e^{-q\lambda}},
\]
yielding the uniform bound $\sup_{n\ge 0} \E\|\widetilde Y_n\|^p \le \E\|\widetilde Y_0\|^p + \frac{C_p\lambda+\bar M_p}{1-e^{-q\lambda}} < \infty$.
\end{lemma}

\begin{proof}
The proof is given in Subsection~\ref{subsec:proof-tildeY-moment-tamed-rLMC}.
\end{proof}

\subsection{Weak and strong local errors for tRLMC}

\begin{proposition}[Weak error for tRLMC]
\label{prop:weak-error-tamed-midpoint}
Assume \hyperref[ass:PJG]{\textup{(A1)}}, \hyperref[ass:PLC]{\textup{(A2)}}, and
\hyperref[ass:D]{\textup{(A4)}}. Let \(X(t,x;t+\lambda)\) denote the solution of
\eqref{eq:LSDE} at time \(t+\lambda\), starting from \(X_t=x\), and let \(\lambda\in(0,1)\). Let \(\bar Y(t,x;t+\lambda)\) be the one-step \emph{tamed randomized midpoint} approximation
defined in \eqref{eq:tamed-rLMC}. Set $R := 3(\ell+1)$.
Then, for all \(t\ge0\) and \(x\in\R^d\), there exists a constant \(C>0\),
independent of \(d,t,x,\lambda\), such that
\[
  \bigl\|\E\!\bigl[X(t,x;t+\lambda)-\bar Y(t,x;t+\lambda)\bigr]\bigr\|
  \;\le\;
  C\,\bigl(1+d^{R/2}+\|x\|^{R}\bigr)\lambda^{2}.
\]
\end{proposition}

\begin{proof}
The proof is given in Subsection~\ref{subsec:proof-weak-error-tamed-midpoint}.
\end{proof}

\begin{proposition}[Strong error for tRLMC]
\label{prop:strong-error-tamed-midpoint}
Assume \hyperref[ass:PJG]{\textup{(A1)}}, \hyperref[ass:PLC]{\textup{(A2)}}, and \hyperref[ass:D]{\textup{(A4)}}. Let \(X(t,x;t+\lambda)\) denote the solution of
\eqref{eq:SDE} at time \(t+\lambda\), starting from \(X_t=x\), and let \(\lambda\in(0,1)\). Let \(\bar Y(t,x;t+\lambda)\) be the one-step \emph{tamed randomized midpoint} approximation
defined in \eqref{eq:tamed-rLMC}. Set $\widetilde R := \max\{2\ell'+4\ell,\, 6(\ell+1)\}$.
Then, for all \(t\ge0\) and \(x\in\R^d\), there exists a constant \(C>0\),
independent of \(d,t,x,\lambda\), such that
\[
  \E\bigl[\|X(t,x;t+\lambda)-\bar Y(t,x;t+\lambda)\|^2\bigr]
  \;\le\;
  C\,\bigl(1+d^{\widetilde R}+\|x\|^{2\widetilde R}\bigr)\lambda^{3}.
\]
\end{proposition}

\begin{proof}
The proof is given in Subsection~\ref{subsec:proof-strong-error-tamed-midpoint}.
\end{proof}

\subsection{Total variation estimates for tRLMC}

\begin{proposition}[Cross-TV regularity for tRLMC]
\label{prop:cross-tv-trlmc}
There exist a constant \(C>0\) and a measurable function
\(C_{\mathrm{TV}}:\R^d\to[0,\infty)\) such that, for all \(x,y\in\R^d\),
\[
  \TV(\delta_x \widetilde P, \delta_y P)
  \;\le\;
  C_{\mathrm{TV}}(x)\,\lambda
  \;+\;
  C\,\frac{\|x-y\|}{\sqrt{\lambda}}.
\]
\end{proposition}

\begin{proof}
The proof is given in Section~\ref{se:cross-TRMLC}.
\end{proof}

\begin{lemma}[TV bound via the shifted construction]
\label{lem:tv-shifted-bound}
Assume \eqref{eq:reg-framework} and \eqref{eq:tv-local-error}. Let
\[
  d_n^2:=\E\|\bar Y_n-Y_n'\|^2.
\]
Then
\begin{equation}\label{eq:tv-shifted-bound}
  \TV(\widetilde\mu_N,\nu_N)
  \le
  \overline C_{\mathrm{TV}}\,\lambda
  +\sqrt{\tfrac{c}{2}}\,d_{N-1}
  +\sqrt{\tfrac{c}{2}\sum_{n=0}^{N-2}\eta_n^2\,d_n^2},
\end{equation}
where \(c=\mathcal{O}(\lambda^{-1})\).
\end{lemma}

\begin{proof}
The proof is given in Section~\ref{se:TV-shifted}.
\end{proof}

\subsection{Wasserstein convergence for tRLMC}

Throughout this subsection, we work in the sampling specialization $h=\nabla u$, $\sigma=\sqrt{\frac{2}{\beta}}$, under Assumption~\ref{ass:LSI}. We denote by \(P\) the one-step Langevin kernel of \eqref{eq:LSDE}, and by \(\widetilde P\) the one-step kernel of \eqref{eq:tamed-rLMC}.

\begin{lemma}\label{lem:uniform-wass}
For any \(n,n_0\in\mathbb{N}\),
\[
  W_2\!\bigl(\mu_0 \widetilde{P}^{n},\,\bigl(\mu_0 \widetilde P^{\,n-n_0}\bigr) P^{n_0}\bigr)
  \;\le\;
  C_d\, e^{\frac{3}{2}K'\lambda n_0}\,\lambda,
\]
where \(C_d = \mathcal{O}(d^{2\ell+2})\).
\end{lemma}

\begin{proof}
Recall the one-step Wasserstein estimate \eqref{eq:one-step-w2-framework}:
\[
  W_2^2(\delta_x \widetilde P,\delta_y P)
  \le
  L^2\|x-y\|^2
  +2a_1(x,y)\|x-y\|
  +a_0(x)^2.
\]
Applying Young’s inequality yields
\[
  W_2^2(\delta_x \widetilde P,\delta_y P)
  \le
  L^3\|x-y\|^2
  + \frac{a_1(x,y)^2}{(L-1)L^2}
  + a_0(x)^2.
\]
By convexity of the Wasserstein distance,
\[
  W_2^2(\mu \widetilde P,\nu P)
  \le
  L^3 W_2^2(\mu,\nu)
  + \frac{1}{(L-1)L^2}\int a_1(x,y)^2\,d\mu(x)\,d\nu(y)
  + \int a_0(x)^2\,d\mu(x).
\]

Define
\[
  \bar a_0^2 := \sup_{0\le k < n_0} \E\bigl[a_0(\bar Y_k)^2\bigr],
  \qquad
  \bar a_1^2 := \sup_{0\le k < n_0} \E\bigl[a_1(\bar Y_k,Y_k)^2\bigr],
\]
where $(\bar Y_k)$ are the \eqref{eq:tamed-rLMC} iterates and $(Y_k)$ the corresponding SDE updates.
Iterating the above inequality yields
\[
  W_2^2\!\bigl(\mu_0 \widetilde{P}^{n},\,\bigl(\mu_0 \widetilde P^{\,n-n_0}\bigr) P^{n_0}\bigr)
  \le
  L^{3n_0}
  \Bigl(
    \frac{\bar a_1^2}{(L-1)^2}
    +
    \frac{\bar a_0^2}{L-1}
  \Bigr).
\]

Since \(L=e^{K'\lambda}\), one has \((L-1)^{-1}=\mathcal{O}(\lambda^{-1})\).
Using the moment bounds from Lemma~\ref{lem:moment-tamed-rLMC}, together with the
local error estimates from Propositions~\ref{prop:weak-error-tamed-midpoint}--\ref{prop:strong-error-tamed-midpoint}, we obtain
\[
  \bar a_0^2 = \mathcal{O}(d^{2\ell+2}\lambda^3),
  \qquad
  \bar a_1^2 = \mathcal{O}(d^{4\ell}\lambda^2).
\]
Substituting these bounds yields the result.
\end{proof}

\begin{lemma}[Contraction]\label{lem:contr}
For all \(n,n_0\in\mathbb{N}\),
\[
  W_2\!\bigl(\bigl(\mu_0 \widetilde P^{\,n-n_0}\bigr) P^{n_0},\pi_\beta\bigr)
  \le
  C_{\mathrm{contr}}\, e^{-\dot c\,\lambda n_0}\,
  W_2\!\bigl(\mu_0 \widetilde P^{\,n-n_0},\pi_\beta\bigr).
\]
\end{lemma}

\begin{proof}
Since \(\pi_\beta\) satisfies a logarithmic Sobolev inequality (Assumption~\ref{ass:LSI}) and
\(\nabla u\) is one-sided Lipschitz (Assumption~\ref{ass:OSL}), the result follows from
\cite[Proposition 2.5]{yang2025non} applied at time \(t=\lambda n_0\).
\end{proof}

\begin{proposition}[Wasserstein convergence]
\label{prop:Wasserstein-convergence}
For all \(n\in\mathbb{N}\),
\[
  W_2(\mu_0\widetilde{P}^n,\pi_\beta)
  \;\le\;
  \frac{C^* C_d}{1-e^{-1}}\,\lambda
  \;+\;
  C' e^{-c_0\lambda n},
\]
where \(C_d, C^*, C', c_0 > 0\) are given explicitly in the proof. This estimate yields Theorem~\ref{thm:wass conv}.
\end{proposition}

\begin{proof}
Let \(n_0=\left\lceil \frac{\log C_{\mathrm{contr}}+1}{\dot c\,\lambda}\right\rceil\), so that \(C_{\mathrm{contr}}e^{-\dot c\lambda n_0}\le e^{-1}\), and set \(C^*:=e^{\frac32 K'\lambda n_0}\).

By Lemma~\ref{lem:uniform-wass} and Lemma~\ref{lem:contr}, for \(n\ge n_0\),
\[
\begin{aligned}
  W_2(\mu_0\widetilde{P}^n,\pi_\beta)
  &= W_2\bigl(\bigl(\mu_0\widetilde{P}^{n-n_0}\bigr)\widetilde{P}^{n_0},\pi_\beta\bigr) \\
  &\le
  W_2\bigl(\bigl(\mu_0\widetilde{P}^{n-n_0}\bigr)\widetilde{P}^{n_0},\,\bigl(\mu_0\widetilde{P}^{n-n_0}\bigr)P^{n_0}\bigr)
  + W_2\bigl(\bigl(\mu_0\widetilde{P}^{n-n_0}\bigr)P^{n_0},\pi_\beta\bigr) \\
  &\le
  C_d e^{\frac{3}{2}K'\lambda n_0}\lambda
  + C_{\mathrm{contr}} e^{-\dot{c}\lambda n_0}\,
  W_2(\mu_0\widetilde{P}^{n-n_0},\pi_\beta) \\
  &\le C^* C_d \lambda
  + \frac{1}{e}\, W_2(\mu_0\widetilde{P}^{n-n_0},\pi_\beta).
\end{aligned}
\]

By \cite[Lemma 3.18]{ye2024error},
\[
  W_2(\mu_0\widetilde{P}^n,\pi_\beta)
  \le
  \frac{C^* C_d}{1-e^{-1}}\,\lambda
  + \sup_{N\ge 0} W_2(\mu_0\widetilde{P}^{N},\pi_\beta)\, e^{-\frac{n}{n_0}+1}.
\]

From the uniform moment bounds in Lemma~\ref{lem:moment-tamed-rLMC},
\[
  W_2(\mu_0\widetilde{P}^{N},\pi_\beta)
  \le
  \sqrt{\E_{\mu_0\widetilde{P}^{N}}\|x\|^2 + \E_{\pi_\beta}\|x\|^2}
  \le
  C_{\mathrm{mom}},
\]
where \(C_{\mathrm{mom}}=\mathcal{O}(d)\). Hence, defining $C':=C_{\mathrm{mom}}\,e$ and $c_0 = \frac{\dot{c}}{\log C_{\mathrm{contr}} + 1}$,
\[
  W_2(\mu_0\widetilde P^n,\pi_\beta)
  \le
  \frac{C^* C_d}{1-e^{-1}}\,\lambda
  + C'\,e^{-c_0\lambda n},
\]
as claimed.
\end{proof}

%% file: Proofs.tex
\section{Detailed Proofs}
\label{sec:proof-section}

This appendix gathers the detailed proofs of the auxiliary lemmas and main results stated throughout the paper.

\subsection{Proofs for Appendix B: Drift Properties and Moment Estimates}

\subsubsection{Proof of Lemma~\ref{lem:P-moment-drift}}
\label{proof:P-moment-drift-proof}
\begin{proof}
Set $X_s:=X(t,x;t+s)$ for $s\in[0,\lambda]$ and consider the Lyapunov function $V(z):=\|z\|^p$. Evaluating the infinitesimal generator $\mathcal L$ of the \ref{eq:SDE} on $V$, and invoking the dissipativity of $h$ alongside Young's inequality, yields
\[
  \mathcal L V(z) = -p\|z\|^{p-2}\langle h(z),z\rangle + \frac{p(d+p-2)}{\beta}\|z\|^{p-2} \le -qV(z)+C_p,
\]
for some constants $q,C_p>0$. Applying It\^o's formula and taking expectations yields the differential inequality $\frac{d}{ds}\E\|X_s\|^p \le -q\E\|X_s\|^p + C_p$. Integrating this over $[0,\lambda]$ with the initial condition $X_0 = x$, and applying the elementary bound $1-e^{-q\lambda}\le q\lambda$, establishes
\[
  \E\|X(t,x;t+\lambda)\|^p \le e^{-q\lambda}\|x\|^p+\frac{C_p}{q}(1-e^{-q\lambda}) \le e^{-q\lambda}\|x\|^p+C_p\lambda,
\]
which completes the proof.
\end{proof}

\subsection{Proofs for Appendix C: Local-Errors and Regularity}

\subsubsection{Proof of Proposition~\ref{prop:W2-Lip-app}}
\label{proof:W2-Lip-app}
\begin{proof}
Let $(X_t)$ and $(Y_t)$ solve~\eqref{eq:LSDE} from $X_0=x$ and $Y_0=y$
respectively, driven by the same Brownian motion, and set $\Delta_t:=X_t-Y_t$.
Under this synchronous coupling the diffusion terms cancel, so
\[
  d\Delta_t=-\bigl(h(X_t)-h(Y_t)\bigr)\,dt.
\]
Consequently by It\^o's lemma,
\[
  \frac{d}{dt}\|\Delta_t\|^2
  =-2\bigl\langle \Delta_t,\,h(X_t)-h(Y_t)\bigr\rangle.
\]
By the one-sided Lipschitz condition~\ref{ass:OSL},
$\langle h(X_t)-h(Y_t),\,\Delta_t\rangle\ge -K'\|\Delta_t\|^2$, hence
\[
  \frac{d}{dt}\|\Delta_t\|^2
  \le 2K'\|\Delta_t\|^2.
\]
Grönwall's lemma yields $\|\Delta_t\|^2\le e^{2K't}\|x-y\|^2$ pointwise, and
taking expectations and square roots gives
\[
  \|X_t-Y_t\|_{L^2}\le e^{K't}\|x-y\|,
  \qquad t\ge 0.
\]
In particular, for $t\in[0,\lambda]$,
\[
  \|X_t-Y_t\|_{L^2}\le L\,\|x-y\|,
  \qquad L:=e^{K'\lambda}.
\]
Finally, if $\lambda\le\frac{\ln 2}{|K'|}$ then $|K'\lambda|\le\ln 2$, so
$e^{K'\lambda}\in[\tfrac12,2]$.
\end{proof}

\subsubsection{Proof of Proposition~\ref{prop:coupling-increment-app}}
\label{proof:coupling-increment}
\begin{proof}
Let $(X_t)$ and $(Y_t)$ solve~\eqref{eq:LSDE} from $X_0=x$ and $Y_0=y$
respectively, driven by the same Brownian motion. Under this synchronous
coupling the diffusion terms cancel, so
\[
  X_t-x-(Y_t-y)
  =
  -\int_0^t \bigl(h(X_s)-h(Y_s)\bigr)\,ds.
\]
As in the proof of Proposition~\ref{prop:W2-Lip-app}, the same coupling gives
the pathwise contraction estimate
\[
  \|X_s-Y_s\|\le e^{K's}\|x-y\|,
  \qquad s\ge 0.
\]

By the Cauchy--Schwarz inequality applied to the time integral,
\[
  \Bigl\|\int_0^t \bigl(h(X_s)-h(Y_s)\bigr)\,ds\Bigr\|^2
  \le
  t\int_0^t \|h(X_s)-h(Y_s)\|^2\,ds,
\]
and the polynomial Lipschitz condition~\ref{ass:PLC} gives,
\[
  \|h(X_s)-h(Y_s)\|^2
  \le
  L'^2\bigl(1+\|X_s\|+\|Y_s\|\bigr)^{2\ell'}\|X_s-Y_s\|^2.
\]
Taking expectations and applying Cauchy--Schwarz,
\[
  \E\|h(X_s)-h(Y_s)\|^2
  \le
  L'^2
  \bigl(\E(1+\|X_s\|+\|Y_s\|)^{4\ell'}\bigr)^{1/2}
  \bigl(\E\|X_s-Y_s\|^{4}\bigr)^{1/2}.
\]
Substituting the pathwise contraction estimate, which yields
$\bigl(\E\|X_s-Y_s\|^{4}\bigr)^{1/2}\le e^{2K's}\|x-y\|^2$, and writing
\[
  M_s:=\bigl(\E(1+\|X_s\|+\|Y_s\|)^{4\ell'}\bigr)^{1/2},
\]
we obtain
\[
  \|X_t-x-(Y_t-y)\|_{L^2}^2
  \le
  t\,L'^2\,\|x-y\|^2
  \int_0^t e^{2K's}\,M_s\,ds
  \le
  t\,L'^2\,\|x-y\|^2
  \Bigl(\sup_{0\le s\le t}M_s\Bigr)
  \int_0^t e^{2K's}\,ds.
\]

By Lemma~\ref{lem:uniform-moments-sde} with $p=2\ell'$,
\[
  \sup_{0\le s\le t}M_s
  \le
  \widetilde C_{\ell'}\bigl(1+\|x\|^{2\ell'}+\|y\|^{2\ell'}\bigr).
\]
For the remaining integral, if $t\le\frac{\ln 2}{|K'|}$, then $e^{2K's}\le 4$ on $[0,t]$, yielding the bound \\
$t\int_0^t e^{2K's}\,ds\le 4t^2 \le C_{K'}^2\,t^2$.

Combining these two bounds and taking square roots, gives
$\bigl(1+\|x\|^{2\ell'}+\|y\|^{2\ell'}\bigr)^{1/2}\le 1+\|x\|^{\ell'}+\|y\|^{\ell'}$,
hence
\[
  \|X_t-x-(Y_t-y)\|_{L^2}
  \le
  C_{\mathrm{cpl}}\,t\,\bigl(1+\|x\|^{\ell'}+\|y\|^{\ell'}\bigr)\,\|x-y\|,
\]
for a constant $C_{\mathrm{cpl}}>0$ depending only on $L'$, $\ell'$, $K'$, and
$\beta$.
\end{proof}

\subsection{Proofs for Appendix D: Auxiliary Results for kTULA}

\subsubsection{Proof of Proposition~\ref{prop:weak-error-app}}
\label{proof:weak-error}
\begin{proof}
Let $\widehat X_t:=x-t\,h_\lambda(x)+\sqrt{\frac{2}{\beta}}\,B_t$ for $t\in[0,\lambda]$
denote the one-step interpolation started at $x$, and let
$(X_t)_{t\in[0,\lambda]}$ solve~\eqref{eq:LSDE} from $X_0=x$ driven by the same
Brownian motion. Subtracting the two mild forms, the Brownian terms cancel and
\[
  \widehat X_\lambda-X_\lambda
  =
  \int_0^\lambda \bigl(h(X_s)-h_\lambda(x)\bigr)\,ds.
\]
Adding and subtracting $h(x)$ inside the integral and taking expectations,
\begin{equation}\label{eq:weak-split}
  \bigl\|\E[\widehat X_\lambda-X_\lambda]\bigr\|
  \le
  \lambda\,\|h_\lambda(x)-h(x)\|
  +
  \int_{0}^{\lambda}\|\E[h(X_s)]-h(x)\|\,ds.
\end{equation}
By Lemma~\ref{lem:properties-hlambda}\ref{it:hlam-taming-error}, the taming
error satisfies $\|h(x)-h_\lambda(x)\|\le C(1+\|x\|^{3(\ell+1)})\lambda$, so the
first term in~\eqref{eq:weak-split} is bounded by
$C(1+\|x\|^{3(\ell+1)})\lambda^2$.

It remains to bound the integrand $\|\E[h(X_s)]-h(x)\|$. By
Lemma~\ref{lem:local-jacobian} under Assumption~\ref{ass:JLC}, $h$ admits the
first-order expansion
\[
  h(X_s)-h(x)=J(h)(x)\,(X_s-x)+R_s,
  \qquad
  \|R_s\|\le\frac{L''}{2}\bigl(1+\|x\|+\|X_s\|\bigr)^{\ell''}\|X_s-x\|^2,
\]
so that, taking expectations,
\begin{equation}\label{eq:weak-decomp}
  \|\E[h(X_s)]-h(x)\|
  \le
  \|J(h)(x)\|\,\|\E[X_s-x]\|
  +
  \E\|R_s\|.
\end{equation}
We estimate the two contributions separately.

For the first, the mild form gives $\E[X_s-x]=-\int_0^s\E[h(X_r)]\,dr$, hence
$\|\E[X_s-x]\|\le\int_0^s\E\|h(X_r)\|\,dr$. Assumption~\ref{ass:PJG} bounds
$\E\|h(X_r)\|\le L(1+\E\|X_r\|^{2\ell})$, and Lemma~\ref{lem:uniform-moments-sde}
controls the moment, yielding
\[
  \|\E[X_s-x]\|
  \le
  C\bigl(1+d^{\ell}+\|x\|^{2\ell}\bigr)\,s.
\]
Assumption~\ref{ass:PJG} also gives the Jacobian growth bound
$\|J(h)(x)\|\le L(1+\|x\|^{2\ell})$, and combining the two,
\begin{equation}\label{eq:weak-Jacobian-part}
  \|J(h)(x)\|\,\|\E[X_s-x]\|
  \le
  C\bigl(1+d^{2\ell}+\|x\|^{4\ell}\bigr)\,s.
\end{equation}

For the remainder $\E\|R_s\|$, Cauchy--Schwarz gives
\[
  \E\|R_s\|
  \le
  \frac{L''}{2}
  \bigl(\E(1+\|x\|+\|X_s\|)^{2\ell''}\bigr)^{1/2}
  \bigl(\E\|X_s-x\|^{4}\bigr)^{1/2}.
\]
The first factor is bounded by $C(1+d^{\ell''/2}+\|x\|^{\ell''})$ via
Lemma~\ref{lem:uniform-moments-sde} with $2p=2\ell''$. For the second, the mild
form $X_s-x=-\int_0^s h(X_r)\,dr+\sqrt{\frac{2}{\beta}}\,B_s$, Jensen's inequality,
Assumption~\ref{ass:PJG} together with
Lemma~\ref{lem:uniform-moments-sde} (with $2p=8\ell$), and the Gaussian moment
$\E\|B_s\|^4\le Cd^2s^2$, give, using $s\le\lambda\le1$,
\[
  \E\|X_s-x\|^4
  \le
  C\bigl(1+d^{4\ell}+\|x\|^{8\ell}\bigr)\,s^2.
\]
Combining the two factors,
\begin{equation}\label{eq:weak-R-part}
  \E\|R_s\|
  \le
  C\bigl(1+d^{\ell''/2}+\|x\|^{\ell''}\bigr)
   \bigl(1+d^{2\ell}+\|x\|^{4\ell}\bigr)\,s.
\end{equation}

Substituting~\eqref{eq:weak-Jacobian-part} and~\eqref{eq:weak-R-part}
into~\eqref{eq:weak-decomp} and retaining the dominant powers,
\[
  \|\E[h(X_s)]-h(x)\|
  \le
  C\bigl(1+d^{M/2}+\|x\|^{M}\bigr)\,s,
  \qquad s\in[0,\lambda],
  \qquad
  M:=\max\{3(\ell+1),\,4\ell+\ell''\}.
\]
Integrating over $[0,\lambda]$ contributes a factor $\lambda^2$, and combining
with the taming term in~\eqref{eq:weak-split} yields
\[
  \bigl\|\E[\widehat X_\lambda-X_\lambda]\bigr\|
  \le
  C\bigl(1+d^{M/2}+\|x\|^{M}\bigr)\lambda^2,
\]
for a constant $C>0$ independent of $x$, $d$, and $\lambda$.
\end{proof}

\subsubsection{Proof of Proposition~\ref{prop:strong-error-app}}
\label{proof:strong-error}
\begin{proof}
Let $\widehat X_s:=x-s\,h_\lambda(x)+\sqrt{\frac{2}{\beta}}\,B_s$ and
$E_s:=\widehat X_s-X_s$ for $s\in[0,\lambda]$, where $(X_s)$ solves~\eqref{eq:LSDE}
from $X_0=x$ under the same Brownian motion. Subtracting the two mild forms, the
Brownian terms cancel and $dE_s=-(h_\lambda(x)-h(X_s))\,ds$, so
by It\^o's lemma
\[
  \frac{d}{ds}\|E_s\|^2
  =-2\bigl\langle E_s,\,h_\lambda(x)-h(X_s)\bigr\rangle.
\]
By Young's inequality with parameter $\varepsilon=\lambda^{-1}$,
\[
  \frac{d}{ds}\|E_s\|^2
  \le
  \frac{1}{\lambda}\|E_s\|^2
  +\lambda\,\|h(X_s)-h_\lambda(x)\|^2,
\]
and integrating over $[0,t]$ and taking expectations,
\begin{equation}\label{eq:strong-young-eps}
  \E\|E_t\|^2
  \le
  \frac{1}{\lambda}\int_0^t \E\|E_s\|^2\,ds
  +
  \lambda\int_0^t \E\|h(X_s)-h_\lambda(x)\|^2\,ds.
\end{equation}

We bound the second integrand. Splitting
\[
  \|h(X_s)-h_\lambda(x)\|^2
  \le
  2\|h(X_s)-h(x)\|^2+2\|h(x)-h_\lambda(x)\|^2,
\]
we treat the two terms separately.

For the drift increment, Assumption~\ref{ass:PLC} and the Cauchy--Schwarz
inequality give
\[
  \E\|h(X_s)-h(x)\|^2
  \le
  C\bigl(\E(1+\|x\|+\|X_s\|)^{4\ell'}\bigr)^{1/2}
   \bigl(\E\|X_s-x\|^{4}\bigr)^{1/2}.
\]
The first factor is bounded by $C(1+d^{\ell'}+\|x\|^{2\ell'})$ via
Lemma~\ref{lem:uniform-moments-sde} with $2p=4\ell'$. For the second, the mild
form $X_s-x=-\int_0^s h(X_r)\,dr+\sqrt{\frac{2}{\beta}}\,B_s$, Jensen's inequality,
Assumption~\ref{ass:PJG} together with Lemma~\ref{lem:uniform-moments-sde}
(with $2p=8\ell$), and the Gaussian moment $\E\|B_s\|^4\le Cd^2s^2$ yield, using
$s\le\lambda\le1$,
\[
  \E\|X_s-x\|^4
  \le
  C\bigl(1+d^{4\ell}+\|x\|^{8\ell}\bigr)\,s^2.
\]
Combining the two factors and retaining the dominant powers,
\begin{equation}\label{eq:strong-drift-incr}
  \E\|h(X_s)-h(x)\|^2
  \le
  C\bigl(1+d^{(2\ell'+4\ell)/2}+\|x\|^{2\ell'+4\ell}\bigr)\,s.
\end{equation}
For the taming term, Lemma~\ref{lem:properties-hlambda}\ref{it:hlam-taming-error}
gives
\begin{equation}\label{eq:strong-taming}
  \|h(x)-h_\lambda(x)\|^2
  \le
  C\bigl(1+\|x\|^{6(\ell+1)}\bigr)\lambda^2.
\end{equation}
Writing $M:=\max\{3(\ell+1),\,2\ell'+4\ell\}$, the exponent in
\eqref{eq:strong-taming} is $6(\ell+1)=2\cdot3(\ell+1)\le 2M$ and that in
\eqref{eq:strong-drift-incr} is $2\ell'+4\ell\le M$ with dimension exponent
$(2\ell'+4\ell)/2\le M/2$. Integrating both over $[0,t]$, which contributes a
factor $\lambda^2$, using $s\le\lambda\le1$,
\begin{equation}\label{eq:strong-drift-int-final}
  \int_0^t \E\|h(X_s)-h_\lambda(x)\|^2\,ds
  \le
  C\bigl(1+d^{M/2}+\|x\|^{M}\bigr)\lambda^2.
\end{equation}

Substituting~\eqref{eq:strong-drift-int-final} into~\eqref{eq:strong-young-eps}
gives
\[
  \E\|E_t\|^2
  \le
  \frac{1}{\lambda}\int_0^t \E\|E_s\|^2\,ds
  +
  C\bigl(1+d^{M/2}+\|x\|^{M}\bigr)\lambda^3,
  \qquad t\in[0,\lambda],
\]
and Grönwall's inequality on $[0,\lambda]$, where the prefactor of the integral
term is $\lambda^{-1}$ and the interval has length at most $\lambda$, yields
\[
  \E\|E_t\|^2
  \le
  C\bigl(1+d^{M/2}+\|x\|^{M}\bigr)\lambda^3\,e^{t/\lambda},
  \qquad t\in[0,\lambda].
\]
Evaluating at $t=\lambda$, so that $e^{t/\lambda}\le e$, and taking square roots,
\[
  \|\widehat X_\lambda-X_\lambda\|_{L^2}
  \le
  C\bigl(1+d^{M/2}+\|x\|^{M}\bigr)\lambda^{3/2},
\]
for a constant $C>0$ independent of $x$, $d$, and $\lambda$.
\end{proof}

\subsubsection{Proof of Proposition~\ref{prop:cross-regularity-app}}
\label{proof:cross-regularity}
\begin{proof}
We first estimate the divergence $\KL(\delta_x\widehat{P}\|\delta_x P)$.
By Girsanov's theorem, (see similar calculations in \cite{tula}) one obtains 
\[\KL(\delta_x\widehat{P}\|\delta_x P)\leq \E \int_0^\lambda |h_\lambda(x)-h(x_t)|^2\,dt\] where $x_t$ is the law of continuous time interpolation of the algorithm at time $t.$
Bounding \[\E \int_0^\lambda |h_\lambda(x)-h(x_t)|^2\,dt\leq 2 \E \int_0^\lambda |h(x)-h(x_t)|^2\,dt+ 2 \lambda |h(x)-h_\lambda(x)|^2\]
By the local Lipschitz continuity, by standard calculations done before, the first term is $\mathcal{O}(\lambda^2)$ and the second term is the taming error which is again $\mathcal{O}(\lambda^2).$
Using the weak triangle inequality for KL divergence
\[KL(\delta_x\hat{P}||\delta_y\hat{P})\leq KL(\delta_x \hat{P}|| \delta_x P)+ R_2(\delta_x P||\delta_y P)\] using the Renyi regularity of the Langevin kernel in Proposition \ref{prop:regularity-app} yields the result. 
\end{proof}

\subsubsection{Proof of Lemma~\ref{lem:tildeY-moment}}
\label{subsec:proof-tildeY-moment}
\begin{proof}
By convexity of $\|\cdot\|^p$ and $\widetilde Y_n=(1-\eta_n)Y_n'+\eta_n\widehat X_n$,
\[
  \E\|\widetilde Y_n\|^p
  \le
  (1-\eta_n)\,\E\|Y_n'\|^p+\eta_n\widehat M_p,
  \qquad
  \widehat M_p:=\sup_{n\ge0}\E\|\widehat X_n\|^p.
\]
Since $Y_n'\sim P_\lambda(\widetilde Y_{n-1},\cdot)$, Lemma~\ref{lem:P-moment-drift} gives $\E\|Y_n'\|^p\le e^{-q\lambda}\E\|\widetilde Y_{n-1}\|^p+C\lambda$, and since $1-\eta_n\le1$,
\[
  \E\|\widetilde Y_n\|^p
  \le
  e^{-q\lambda}\,\E\|\widetilde Y_{n-1}\|^p+C\lambda+\widehat M_p.
\]
Iterating this recursion and summing the geometric series $\sum_{k=0}^{n-1}e^{-q\lambda k}\le \frac{1}{1-e^{-q\lambda}}$ yields
\[
  \sup_{n\ge0}\E\|\widetilde Y_n\|^p
  \le
  \E\|\widetilde Y_0\|^p+\frac{C\lambda+\widehat M_p}{1-e^{-q\lambda}}
  <\infty,
\]
which is finite by Lemma~\ref{lem:moment-algorithm}.
\end{proof}

\subsubsection{Proof of Lemma~\ref{lem:distance-recursion-app}}
\label{subsec:proof-distance-recursion}
\begin{proof}
By the $W_2$ coupling lift, $d_{n+1}^2\le\E[W_2^2(\delta_{\widehat X_n}\widehat P,\delta_{\widetilde Y_n}P)]$, and Lemma~\ref{lem:one-step-w2} bounds the right-hand side by
\[
  \E\bigl[L^2\|\widehat X_n-\widetilde Y_n\|^2
  +2\,a_1(\widehat X_n,\widetilde Y_n)\,\|\widehat X_n-\widetilde Y_n\|
  +a_0(\widehat X_n)^2\bigr].
\]
Since $\widetilde Y_n=(1-\eta_n)Y_n'+\eta_n\widehat X_n$, we have $\widehat X_n-\widetilde Y_n=(1-\eta_n)(\widehat X_n-Y_n')$, hence
$\E\|\widehat X_n-\widetilde Y_n\|^2=(1-\eta_n)^2 d_n^2$. Applying Cauchy--Schwarz to the cross term and the definitions
\[
  \overline a_1:=\sup_{0\le k<N}\|a_1(\widehat X_k,\widetilde Y_k)\|_{L^2},
  \qquad
  \overline a_0:=\sup_{0\le k<N}\|a_0(\widehat X_k)\|_{L^2},
\]
gives $\E[a_1(\widehat X_n,\widetilde Y_n)\|\widehat X_n-\widetilde Y_n\|]\le\overline a_1(1-\eta_n)d_n$ and $\E[a_0(\widehat X_n)^2]\le\overline a_0^{\,2}$. Substituting these three bounds yields
\[
  d_{n+1}^2
  \le
  L^2(1-\eta_n)^2 d_n^2
  +2(1-\eta_n)\overline a_1 d_n
  +\overline a_0^{\,2},
\]
as claimed.
\end{proof}

\subsection{Proofs for Appendix D: Auxiliary Results for tRLMC}

\subsubsection{Proof of Lemma~\ref{lem:moment-tamed-rLMC}(i)}
\label{subsec:proof-uniform-moments-tamed-rLMC}
\begin{proof}
Rewrite~\eqref{eq:tamed-rLMC} as
\begin{equation}\label{eq:tamed-rLMC-recast}
  \bar Y_{n+1}
  =
  \bar Y_n-\lambda h_\lambda(\bar Y_n)
  +\sqrt{\frac{2}{\beta}}\,\Delta W_{n+1}
  -\lambda\,\Xi_{n+1},
  \qquad
  \Xi_{n+1}:=h_\lambda(\bar Y_{n+1}^{\tau})-h_\lambda(\bar Y_n).
\end{equation}
Expanding $\|\bar Y_{n+1}\|^2$, taking expectations, and using that
$\Delta W_{n+1}$ is centred and independent of $\bar Y_n$, so the three
Brownian cross terms vanish, the surviving inner products are estimated by
Cauchy--Schwarz, giving
\begin{align}
  \E\|\bar Y_{n+1}\|^2
  &\le
  \E\|\bar Y_n\|^2
  -2\lambda\,\E\langle \bar Y_n,h_\lambda(\bar Y_n)\rangle
  +\lambda^2\E\|h_\lambda(\bar Y_n)\|^2
  +\frac{2}{\beta}\E\|\Delta W_{n+1}\|^2
  \nonumber\\
  &\quad
  +2\lambda\,\E\bigl[\|\bar Y_n\|\,\|\Xi_{n+1}\|\bigr]
  +2\lambda^2\,\E\bigl[\|h_\lambda(\bar Y_n)\|\,\|\Xi_{n+1}\|\bigr]
  \nonumber\\
  &\quad
  +2\lambda\sqrt{\frac{2}{\beta}}\,\E\bigl[\|\Delta W_{n+1}\|\,\|\Xi_{n+1}\|\bigr]
  +\lambda^2\E\|\Xi_{n+1}\|^2 .
  \label{eq:tamed-moment-expand}
\end{align}
By the dissipativity bound of Lemma~\ref{lem:properties-hlambda}\ref{it:hlam-dissipativity},
$-2\lambda\,\E\langle\bar Y_n,h_\lambda(\bar Y_n)\rangle\le-2a\lambda\,\E\|\bar Y_n\|^2+2b\lambda$.
Applying Young's inequality to the cross term
$2\lambda\,\E[\|\bar Y_n\|\,\|\Xi_{n+1}\|]
\le\tfrac{a}{2}\lambda\,\E\|\bar Y_n\|^2+\tfrac{2\lambda}{a}\E\|\Xi_{n+1}\|^2$,
and to the two remaining cross terms with parameter $\varepsilon=1$, we collect
\begin{equation}\label{eq:tamed-moment-prelim}
  \E\|\bar Y_{n+1}\|^2
  \le
  \Bigl(1-\tfrac{3a}{2}\lambda\Bigr)\E\|\bar Y_n\|^2
  +2\lambda^2\E\|h_\lambda(\bar Y_n)\|^2
  +\frac{4}{\beta}\E\|\Delta W_{n+1}\|^2
  +2b\lambda
  +\Bigl(\tfrac{2\lambda}{a}+3\lambda^2\Bigr)\E\|\Xi_{n+1}\|^2
\end{equation}

It remains to bound the two $h_\lambda$-dependent terms. By
Lemma~\ref{lem:properties-hlambda}\ref{it:hlam-growth},
\begin{equation}\label{eq:hlam-growth}
  \lambda^2\E\|h_\lambda(\bar Y_n)\|^2
  \le
  8a^2\lambda^2\,\E\|\bar Y_n\|^2+8L^2\lambda .
\end{equation}
For $\Xi_{n+1}$, the Lipschitz bound of
Lemma~\ref{lem:properties-hlambda}\ref{it:hlam-lipschitz} gives
$\|\Xi_{n+1}\|^2\le L_0^2\lambda^{-1}\|\bar Y_{n+1}^{\tau}-\bar Y_n\|^2$, and from
$\bar Y_{n+1}^{\tau}-\bar Y_n
=-\lambda\tau_{n+1}h_\lambda(\bar Y_n)+\sqrt{\frac{2}{\beta}}\,\Delta W_{n+1}^{\tau}$ with
$\tau_{n+1}\in(0,1)$,
\[
  \E\|\Xi_{n+1}\|^2
  \le
  2L_0^2\lambda\,\E\|h_\lambda(\bar Y_n)\|^2
  +\frac{4L_0^2}{\beta}\lambda^{-1}\E\|\Delta W_{n+1}^{\tau}\|^2 .
\]
Conditionally on $\tau_{n+1}$, $\Delta W_{n+1}^{\tau}$ is a Gaussian increment
over an interval of length $\tau_{n+1}\lambda$, so \\
$\E[\|\Delta W_{n+1}^{\tau}\|^2\mid\tau_{n+1}]=\tau_{n+1}\lambda d$, and taking
expectation over $\tau_{n+1}\sim\mathcal U(0,1)$ gives
$\E\|\Delta W_{n+1}^{\tau}\|^2=\tfrac12\lambda d\le\lambda d$. With
$\E\|\Delta W_{n+1}\|^2=\lambda d$ and $0<\lambda\le1$, combining the last two
displays with~\eqref{eq:hlam-growth} yields
\begin{equation}\label{eq:hlam-diff-collected}
  \Bigl(\tfrac{2\lambda}{a}+3\lambda^2\Bigr)\E\|\Xi_{n+1}\|^2
  \le
  C\lambda^2\,\E\|\bar Y_n\|^2+C\lambda d ,
\end{equation}
for a constant $C>0$ depending only on $a,L,L_0,\beta$.

Substituting~\eqref{eq:hlam-growth} and~\eqref{eq:hlam-diff-collected}
into~\eqref{eq:tamed-moment-prelim} and using $\E\|\Delta W_{n+1}\|^2=\lambda d$,
\[
  \E\|\bar Y_{n+1}\|^2
  \le
  \Bigl(1-\tfrac{3a}{2}\lambda+C\lambda^2\Bigr)\E\|\bar Y_n\|^2
  +C\lambda d .
\]
Set
\[
  \lambda_{2,\max}^{tRLMC}
  :=\min\Bigl\{1,\;\tfrac{1}{8a},\;\tfrac{a}{2CL_0^2}\Bigr\}.
\]
For $\lambda\in(0,\lambda_{2,\max}^{tRLMC}]$ one has $C\lambda^2\le C\lambda\le
\tfrac{a}{2}$, so $1-\tfrac{3a}{2}\lambda+C\lambda^2\le 1-a\lambda$, hence,
with $\mu:=a$,
\[
  \E\|\bar Y_{n+1}\|^2
  \le
  (1-\mu\lambda)\,\E\|\bar Y_n\|^2+C\lambda d .
\]
Iterating, and using $\sum_{k=0}^{n-1}(1-\mu\lambda)^k\le(\mu\lambda)^{-1}$ and
$(1-\mu\lambda)^n\le e^{-\mu n\lambda}$,
\[
  \E\|\bar Y_n\|^2
  \le
  e^{-\mu n\lambda}\,\E\|x_0\|^2+\frac{C}{\mu}\,d ,
\]
so the claim holds with $M_2:=C/\mu$.
\end{proof}

\subsubsection{Proof of Lemma~\ref{lem:moment-tamed-rLMC}(ii)}
\label{subsec:proof-uniform-p-moments-tamed-rLMC}
\begin{proof}
Fix $p\in[2,4\ell]\cap\mathbb N$ and $0<\lambda\le\lambda_{p,\max}^{tRLMC}$, with the threshold $\lambda_{p,\max}^{tRLMC}\le 1$ to be specified. Let $C_p>0$ denote a generic constant depending only on $p,a,b,L,L_0,\beta$. Defining the residual $\Xi_{n+1}:=h_\lambda(\bar Y_{n+1}^{\tau})-h_\lambda(\bar Y_n)$ yields the increment decomposition
\[
  \bar Y_{n+1}-\bar Y_n = -\lambda h_\lambda(\bar Y_n) +\sqrt{\frac{2}{\beta}}\,\Delta W_{n+1} -\lambda\,\Xi_{n+1}.
\]
By Lemma~\ref{lem:properties-hlambda}\ref{it:hlam-growth}\ref{it:hlam-lipschitz} and the inner increment $\bar Y_{n+1}^{\tau}-\bar Y_n = -\lambda\tau_{n+1}h_\lambda(\bar Y_n)+\sqrt{\frac{2}{\beta}}\,\Delta W_{n+1}^{\tau}$, the residual satisfies $\lambda\|\Xi_{n+1}\|\le C_p\lambda^{3/2}\|\bar Y_n\|+C_p\lambda +C_p\lambda^{1/2}\|\Delta W_{n+1}^{\tau}\|$. Since $\Delta W_{n+1}\sim \mathcal{N}(0,\lambda I_d)$ and $\Delta W_{n+1}^{\tau}\sim \mathcal{N}(0,\tau_{n+1}\lambda I_d)$, taking conditional Gaussian moments yields
\[
  \E\bigl[\|\bar Y_{n+1}-\bar Y_n\|^2\mid\bar Y_n\bigr] \le C_p\lambda^2\|\bar Y_n\|^2+C_p\lambda(1+d),
  \qquad
  \E\bigl[\|\bar Y_{n+1}-\bar Y_n\|^{2p}\mid\bar Y_n\bigr] \le C_p\lambda^{2p}\|\bar Y_n\|^{2p}+C_p\lambda^p(1+d^p).
\]
The polynomial expansion $\|x+z\|^{2p}\le\|x\|^{2p}+2p\|x\|^{2p-2}\langle x,z\rangle +C_p(\|x\|^{2p-2}\|z\|^2+\|z\|^{2p})$ furnishes the conditional bound
\begin{align*}
  \E\bigl[\|\bar Y_{n+1}\|^{2p}\mid\bar Y_n\bigr]
  &\le \|\bar Y_n\|^{2p} +2p\|\bar Y_n\|^{2p-2}\E\bigl[\langle\bar Y_n,\bar Y_{n+1}-\bar Y_n\rangle\mid\bar Y_n\bigr] \\
  &\quad +C_p\|\bar Y_n\|^{2p-2}\E\bigl[\|\bar Y_{n+1}-\bar Y_n\|^2\mid\bar Y_n\bigr] +C_p\E\bigl[\|\bar Y_{n+1}-\bar Y_n\|^{2p}\mid\bar Y_n\bigr].
\end{align*}
Since the Brownian increment is zero-mean, the conditional drift is bounded via the dissipativity condition of Lemma~\ref{lem:properties-hlambda}\ref{it:hlam-dissipativity} and Young's inequality, yielding
\[
  2p\|\bar Y_n\|^{2p-2}\E\bigl[\langle\bar Y_n,\bar Y_{n+1}-\bar Y_n\rangle\mid\bar Y_n\bigr] \le -\tfrac{3pa}{2}\lambda\|\bar Y_n\|^{2p} +C_p\lambda^{3/2}\|\bar Y_n\|^{2p} +C_p\lambda(1+d^p).
\]
Inserting the drift and increment moment bounds directly into the Taylor expansion yields
\[
  \E\bigl[\|\bar Y_{n+1}\|^{2p}\mid\bar Y_n\bigr] \le \Bigl(1-\tfrac{3pa}{2}\lambda+C_p\lambda^{3/2}+C_p\lambda^2\Bigr)\|\bar Y_n\|^{2p} +C_p\lambda(1+d)\|\bar Y_n\|^{2p-2} +C_p\lambda(1+d^p).
\]
Defining the step-size threshold 
\[
\lambda_{p,\max}^{tRLMC} := \min\left\{\lambda_{2,\max}^{tRLMC},\,1,\,\left(\frac{pa}{4C_p}\right)^2,\,\frac{pa}{4C_p}\right\}
\]
ensures the leading bracket is strictly bounded by $1-pa\lambda$, giving
\[
  \E\bigl[\|\bar Y_{n+1}\|^{2p}\mid\bar Y_n\bigr] \le (1-pa\lambda)\|\bar Y_n\|^{2p} +C_p\lambda(1+d)\|\bar Y_n\|^{2p-2} +C_p\lambda(1+d^p).
\]
Setting $M_p^2:=\frac{2C_p(1+d)}{pa}$, the strict polynomial decay $-pa\,r^{2p}+C_p(1+d)\,r^{2p-2}\le-\tfrac{pa}{2}r^{2p}$ holds for $r\ge M_p$. Applying this for $\|\bar Y_n\| > M_p$ and uniformly bounding the terms for $\|\bar Y_n\| \le M_p$, it follows that for $c_p:=\frac{pa}{2}$,
\[
  \E\bigl[\|\bar Y_{n+1}\|^{2p}\mid\bar Y_n\bigr] \le (1-c_p\lambda)\|\bar Y_n\|^{2p}+C_p\lambda(1+d^p).
\]
Taking unconditional expectations and iterating the recurrence bounds the geometric sum, yielding
\[
  \E\|\bar Y_n\|^{2p} \le e^{-c_p n\lambda}\,\E\|x_0\|^{2p}+\frac{C_p}{c_p}(1+d^p),
\]
ensuring $\sup_{n\ge0}\E\|\bar Y_n\|^{2p}<\infty$ as claimed.
\end{proof}

\subsubsection{Proof of Lemma~\ref{lem:P-moment-drift-RLMC}}
\label{subsec:proof-P-moment-drift-RLMC}
\begin{proof}
Let
\[
Y'_{n+1}=X(t_n,\widetilde Y_n;t_{n+1}),
\qquad t_{n+1}-t_n=\lambda,
\]
so that, by definition,
\[
Y'_{n+1}\sim P_\lambda(\widetilde Y_n,\cdot).
\]

By Lemma~\ref{lem:P-moment-drift} for the Langevin diffusion
over a time interval of length \(\lambda\), there exist
\(q>0\) and \(C_p>0\), depending on \(p\), \(d\), \(a\), and \(b\), such that for every
\(x\in\mathbb R^d\) and every \(\lambda\in(0,1]\),
\[
\E\bigl[\|X(t_n,x;t_{n+1})\|^p\bigr]
\le
e^{-q\lambda}\|x\|^p + C_p\lambda .
\]
Applying this with \(x=\widetilde Y_n\), conditionally on
\(\widetilde Y_n\), yields
\[
\E\!\left[\|Y'_{n+1}\|^p \mid \widetilde Y_n\right]
=
\E\!\left[\|X(t_n,\widetilde Y_n;t_{n+1})\|^p \mid \widetilde Y_n\right]
\le
e^{-q\lambda}\|\widetilde Y_n\|^p + C_p\lambda .
\]
Taking expectations gives
\[
\E\|Y'_{n+1}\|^p
\le
e^{-q\lambda}\E\|\widetilde Y_n\|^p + C_p\lambda .
\]
This proves the claim.
\end{proof}

\subsubsection{Proof of Lemma~\ref{lem:tildeY-moment-tamed-rLMC}}
\label{subsec:proof-tildeY-moment-tamed-rLMC}
\begin{proof}
The argument is identical to that of Lemma~\ref{lem:tildeY-moment}, with $\widehat X_n$ replaced by the \eqref{eq:tamed-rLMC} iterate $\bar Y_n$ and Lemma~\ref{lem:P-moment-drift} replaced by Lemma~\ref{lem:P-moment-drift-RLMC}. By convexity of $\|\cdot\|^p$ and $\widetilde Y_n=(1-\eta_n)Y_n'+\eta_n\bar Y_n$,
\[
  \E\|\widetilde Y_n\|^p
  \le
  (1-\eta_n)\,\E\|Y_n'\|^p+\eta_n\bar M_p,
  \qquad
  \bar M_p:=\sup_{n\ge0}\E\|\bar Y_n\|^p,
\]
which is finite by Lemma~\ref{lem:moment-tamed-rLMC}. Since $Y_n'\sim P_\lambda(\widetilde Y_{n-1},\cdot)$, Lemma~\ref{lem:P-moment-drift-RLMC} gives $\E\|Y_n'\|^p\le e^{-q\lambda}\E\|\widetilde Y_{n-1}\|^p+C_p\lambda$, and since $1-\eta_n\le1$,
\[
  \E\|\widetilde Y_n\|^p
  \le
  e^{-q\lambda}\,\E\|\widetilde Y_{n-1}\|^p+C_p\lambda+\bar M_p.
\]
Iterating and summing $\sum_{k=0}^{n-1}e^{-q\lambda k}\le(1-e^{-q\lambda})^{-1}$ yields
\[
  \sup_{n\ge0}\E\|\widetilde Y_n\|^p
  \le
  \E\|\widetilde Y_0\|^p+\frac{C_p\lambda+\bar M_p}{1-e^{-q\lambda}}
  <\infty.
\]
\end{proof}

\subsubsection{Proof of Proposition~\ref{prop:weak-error-tamed-midpoint}}
\label{subsec:proof-weak-error-tamed-midpoint}
\begin{proof}
Let $\tau\sim\mathcal U(0,1)$ be independent of the driving Brownian motion. Writing the exact solution and the one-step numerical approximation \eqref{eq:tamed-rLMC} in their integral forms yields
\begin{align*}
  X(t,x;t+\lambda) &= x-\int_t^{t+\lambda} h(X(t,x;s))\,ds +\sqrt{\frac{2}{\beta}}\,(B_{t+\lambda}-B_t),\\
  \bar Y_m(t,x;t+\tau\lambda) &= x-\tau\lambda\,h_\lambda(x)+\sqrt{\frac{2}{\beta}}\,(B_{t+\tau\lambda}-B_t),
\end{align*}
and
\begin{equation}\label{eq:tamed-midpoint-one-step}
  \bar Y(t,x;t+\lambda) = x-\lambda\,h_\lambda\bigl(\bar Y_m(t,x;t+\tau\lambda)\bigr) +\sqrt{\frac{2}{\beta}}\,(B_{t+\lambda}-B_t).
\end{equation}
Subtracting the approximation from the exact solution, the Brownian increments cancel. Adding and subtracting the intermediate values $h(X(t,x;t+\tau\lambda))$ and $h_\lambda(X(t,x;t+\tau\lambda))$ inside the integral yields the decomposition
\begin{equation}\label{eq:weak-decomp-noProj}
\begin{aligned}
  X(t,x;t+\lambda)-\bar Y(t,x;t+\lambda) &= -\int_t^{t+\lambda}\bigl(h(X(t,x;s))-h(X(t,x;t+\tau\lambda))\bigr)\,ds\\
  &\quad +\lambda\bigl(h(X(t,x;t+\tau\lambda))-h_\lambda(X(t,x;t+\tau\lambda))\bigr)\\
  &\quad +\lambda\bigl(h_\lambda(X(t,x;t+\tau\lambda))-h_\lambda(\bar Y_m(t,x;t+\tau\lambda))\bigr).
\end{aligned}
\end{equation}
Taking the norm of the expectation and applying the triangle inequality yields the bound
\[
  \bigl\|\E[X(t,x;t+\lambda)-\bar Y(t,x;t+\lambda)]\bigr\| \le I_1+I_2+I_3,
\]
where $I_1, I_2, I_3$ are the norms of the expectations of the respective terms above.

For the first term, the change of variables $s=t+u\lambda$ and Fubini's theorem, exploiting the independence of $\tau$, imply
\[
  \E\Bigl[\int_t^{t+\lambda}h(X(t,x;s))\,ds\Bigr] = \lambda\,\E\bigl[h(X(t,x;t+\tau\lambda))\bigr].
\]
Since the subtracted term $h(X(t,x;t+\tau\lambda))$ is independent of $s$, its integral over $[t, t+\lambda]$ equals to the same quantity. Consequently, the expectations cancel entirely, yielding $I_1=0$.

For the second term, setting $Z:=X(t,x;t+\tau\lambda)$ and $R:=3(\ell+1)$, the taming error bound of Lemma~\ref{lem:properties-hlambda}\ref{it:hlam-taming-error} yields $\|h(Z)-h_\lambda(Z)\|\le 2(L+a)\lambda\,(1+\|Z\|^{2R})^{1/2}$. Taking expectations and applying Jensen's inequality yields
\[
  I_2 \le 2(L+a)\lambda^2\bigl(1+(\E\|Z\|^{2R})^{1/2}\bigr).
\]
Applying the uniform moment bound of Lemma~\ref{lem:uniform-moments-sde} ensures $\E\|Z\|^{2R}\le C(\|x\|^{2R}+d^R)$, which, upon taking the square root, simplifies the estimate to $I_2\le C\lambda^2\bigl(1+d^{R/2}+\|x\|^R\bigr)$.

For the third term, the local Lipschitz property of Lemma~\ref{lem:properties-hlambda}\ref{it:hlam-lipschitz} gives
\begin{equation}\label{eq:I3-start}
  I_3 \le L_0\,\lambda\,\E\bigl\|X(t,x;t+\tau\lambda)-\bar Y_m(t,x;t+\tau\lambda)\bigr\|.
\end{equation}
Expanding this difference via the mild form on $[t,t+\tau\lambda]$ and canceling the Brownian increments provides
\[
  X(t,x;t+\tau\lambda)-\bar Y_m(t,x;t+\tau\lambda) = -\int_t^{t+\tau\lambda}\bigl(h(X(t,x;s))-h(x)\bigr)\,ds +\tau\lambda\bigl(h_\lambda(x)-h(x)\bigr).
\]
Taking norms and expectations, and noting $\E[\tau]=\tfrac12$ for the deterministic second term, yields
\begin{equation}\label{eq:EY-split}
  \E\bigl\|X(t,x;t+\tau\lambda)-\bar Y_m(t,x;t+\tau\lambda)\bigr\| \le A_\lambda+B_\lambda,
\end{equation}
where $A_\lambda$ bounds the integral term and $B_\lambda$ bounds the taming term. By Lemma~\ref{lem:properties-hlambda}\ref{it:hlam-taming-error}, the taming term satisfies
\begin{equation}\label{eq:B-lam}
  B_\lambda\le C\lambda^2(1+\|x\|^{R}).
\end{equation}
For $A_\lambda$, applying the Cauchy--Schwarz and Jensen's inequalities, gives
\[
  A_\lambda \le \Bigl(\lambda\int_t^{t+\lambda}\E\|h(X(t,x;s))-h(x)\|^2\,ds\Bigr)^{1/2}.
\]
By Assumption~\ref{ass:PLC} and Cauchy--Schwarz, the integrand is bounded by the product of the moment estimates $\bigl(\E(1+\|X(t,x;s)\|+\|x\|)^{4\ell'}\bigr)^{1/2}$ and $\bigl(\E\|X(t,x;s)-x\|^{4}\bigr)^{1/2}$. Evaluating these via Assumption~\ref{ass:PJG} and Lemma~\ref{lem:uniform-moments-sde} yields $\E\|h(X(t,x;s))-h(x)\|^2 \le C(1+d^R+\|x\|^{2R})(s-t)$. Integrating this over $[t, t+\lambda]$ and taking the square root establishes
\begin{equation}\label{eq:A-lam-final}
  A_\lambda \le C\bigl(1+d^{R/2}+\|x\|^R\bigr)\lambda^{3/2}.
\end{equation}
Substituting \eqref{eq:EY-split}, \eqref{eq:B-lam}, and \eqref{eq:A-lam-final} into~\eqref{eq:I3-start} and using $\lambda^{5/2}\le\lambda^2$ for $\lambda\in(0,1)$ establishes 
\[
I_3\le C\bigl(1+d^{R/2}+\|x\|^R\bigr)\lambda^2
.\]
Combining the bounds for $I_1, I_2$, and $I_3$ yields the final one-step estimate
\[
  \bigl\|\E[X(t,x;t+\lambda)-\bar Y(t,x;t+\lambda)]\bigr\| \le C\,\bigl(1+d^{R/2}+\|x\|^R\bigr)\lambda^2,
\]
which proves the claim.
\end{proof}

\begin{lemma}
\label{lem:time-reg-h-of-X}
Under Assumptions~\ref{ass:PJG} and~\ref{ass:PLC}, there exists a constant $C>0$, independent of $d,t,x,\lambda$, such that for every $0<\lambda\le1$ and every $s_1,s_2\in[t,t+\lambda]$,
\[
  \E\|h(X(t,x;s_1))-h(X(t,x;s_2))\|^2
  \le
  C\bigl(1+d^{2\ell'+4\ell}+\|x\|^{2\ell'+4\ell}\bigr)|s_1-s_2|.
\]
\end{lemma}

\begin{proof}
Let $s_1<s_2$. Applying Assumption~\ref{ass:PLC} and the Cauchy--Schwarz inequality yields
\[
  \E\|h(X_{s_1})-h(X_{s_2})\|^2
  \le
  L'^2\bigl(\E(1+\|X_{s_1}\|+\|X_{s_2}\|)^{4\ell'}\bigr)^{1/2}
  \bigl(\E\|X_{s_2}-X_{s_1}\|^{4}\bigr)^{1/2}.
\]
By Lemma~\ref{lem:uniform-moments-sde}, the first factor is bounded by $C(1+d^{2\ell'}+\|x\|^{2\ell'})$. For the second factor, the integral form $X_{s_2}-X_{s_1}=-\int_{s_1}^{s_2}h(X_r)\,dr+\sqrt{\frac{2}{\beta}}(B_{s_2}-B_{s_1})$, Jensen's inequality, Assumption~\ref{ass:PJG}, and the Gaussian fourth moment yield the bound
\[
  \E\|X_{s_2}-X_{s_1}\|^4
  \le
  C\bigl(1+d^{8\ell}+\|x\|^{8\ell}\bigr)(s_2-s_1)^2.
\]
Taking the square root and multiplying the bounds yields the claim.
\end{proof}

\subsubsection{Proof of Proposition~\ref{prop:strong-error-tamed-midpoint}}
\label{subsec:proof-strong-error-tamed-midpoint}
\begin{proof}
The exact decomposition \eqref{eq:weak-decomp-noProj} of the error $X(t,x;t+\lambda)-\bar Y(t,x;t+\lambda)$ from the weak-error proof applies verbatim. Applying the elementary bound $\|a+b+c\|^2\le 3(\|a\|^2+\|b\|^2+\|c\|^2)$ yields
\[
  \E\|X(t,x;t+\lambda)-\bar Y(t,x;t+\lambda)\|^2 \le 3(J_1+J_2+J_3),
\]
where $J_1,J_2,J_3$ denote the squared $L^2$ norms of the corresponding three terms in~\eqref{eq:weak-decomp-noProj}.

By the Cauchy--Schwarz inequality applied to the time integral, the first term satisfies
\[
  J_1 \le \lambda\int_t^{t+\lambda} \E\|h(X(t,x;s))-h(X(t,x;t+\tau\lambda))\|^2\,ds.
\]
Applying the time-regularity bound of Lemma~\ref{lem:time-reg-h-of-X} to the integrand provides the upper bound $C(1+d^{2\ell'+4\ell}+\|x\|^{2\ell'+4\ell})|s-(t+\tau\lambda)|$. Since $|s-(t+\tau\lambda)|\le\lambda$, integrating over the interval of length $\lambda$ furnishes
\[
  J_1 \le C\bigl(1+d^{2\ell'+4\ell}+\|x\|^{2\ell'+4\ell}\bigr)\lambda^3.
\]

For the second term, applying Lemma~\ref{lem:properties-hlambda}\ref{it:hlam-taming-error} and the uniform SDE moment bound of Lemma~\ref{lem:uniform-moments-sde} yields
\[
  J_2 = \lambda^2\E\|h(X(t,x;t+\tau\lambda))-h_\lambda(X(t,x;t+\tau\lambda))\|^2 \le C\lambda^4\,\E\bigl[1+\|X(t,x;t+\tau\lambda)\|^{6(\ell+1)}\bigr].
\]
Evaluating the moment bound directly yields $J_2 \le C\bigl(1+d^{6(\ell+1)}+\|x\|^{6(\ell+1)}\bigr)\lambda^4$.

For the third term, the local Lipschitz property of Lemma~\ref{lem:properties-hlambda}\ref{it:hlam-lipschitz} gives
\[
  J_3 \le L_0^2\,\lambda\,\E\|X(t,x;t+\tau\lambda)-\bar Y_m(t,x;t+\tau\lambda)\|^2.
\]
Expanding the difference via the identity \eqref{eq:EY-split}, applying $\|u+v\|^2\le 2(\|u\|^2+\|v\|^2)$, and integrating via Cauchy--Schwarz yields
\[
  \E\|X(t,x;t+\tau\lambda)-\bar Y_m(t,x;t+\tau\lambda)\|^2 \le 2\lambda\int_t^{t+\lambda}\E\|h(X(t,x;s))-h(x)\|^2\,ds + 2\lambda^2\|h_\lambda(x)-h(x)\|^2.
\]
The time-regularity of Lemma~\ref{lem:time-reg-h-of-X} bounds the first integrand by $C(1+d^{2\ell'+4\ell}+\|x\|^{2\ell'+4\ell})(s-t)$, which integrates to an order-$\lambda^2$ term. The taming error is bounded by $C(1+\|x\|^{6(\ell+1)})\lambda^2$. Substituting these estimates into the $J_3$ inequality and using $\lambda^5\le\lambda^4$ for $\lambda\in(0,1)$ establishes
\[
  J_3 \le C\bigl(1++d^{\widetilde R}+\|x\|^{\widetilde R}\bigr)\lambda^4.
\]

Combining the bounds for $J_1, J_2$, and $J_3$, and absorbing the higher-order $\lambda^4$ terms into $\lambda^3$, yields the final one-step estimate
\[
  \E\|X(t,x;t+\lambda)-\bar Y(t,x;t+\lambda)\|^2 \le C\bigl(1+d^{\widetilde R}+\|x\|^{\widetilde R}\bigr)\lambda^3,
\]
where $\widetilde R := \max\{2\ell'+4\ell,\, 6(\ell+1)\}$. This proves the claim.
\end{proof}

\subsubsection{Cross-regularity bound for tamed RLMC}\label{se:cross-TRMLC}
\begin{lemma}\label{Kl small u}
    Let $\hat{P}^u$ be the Markov kernel of \eqref{eq:tamed-rLMC} run with fixed $u$ and $x \in \mathbb{R}^d.$
    For $u\leq 1-\lambda^2$ there holds
    \[TV( \delta_x\hat{P}^u|| \delta_x P)\leq \sqrt{C(x)} (1-u)^{-\frac{1}{2}} \lambda\]
\end{lemma}
\begin{proof}
Conditioned on the Brownian motion $B_{\lambda u}$ the continuous time interpolation the tamed RLMC can be described by the following SDE:
\[\hat{X}_t= X_0-\int_0^{t} h_\lambda(\hat{X}^{+}_{\lambda u}) ds +\sqrt{\frac{2}{\beta}} \int_0^{t} d\tilde{B}_s\]
where $X_0=x-\lambda u h_\lambda(\hat{X}^{+}_{\lambda u})+\sqrt{\frac{2}{\beta}}B_{\lambda u}$.
Let \[Y_t=X_0-\int_0^{t} h(Y_s)ds +\sqrt{\frac{2}{\beta}}\int_0^t d\tilde{B_s}.\]
By applying Girsanov's theorem one obtains 
\begin{equation}\label{eq-1kl}
    \begin{aligned}
    KL(\mathcal{L}(\hat{X}_{\lambda(1-u)})|\mathcal{L}(Y_{\lambda(1-u)}))&\leq \int_0^{\lambda (1-u)} \E |h_\lambda(\hat{X}^{+}_{\lambda u})-h(\hat{X}_s)|^2 ds\\&\leq 2 \int_0^{\lambda (1-u)} \E |h_\lambda(\hat{X}^{+}_{\lambda u})-h(\hat{X}^{+}_{\lambda u})|^2ds +2 \int_0^{\lambda (1-u)} \E |h(\hat{X}^{+}_{\lambda u})-h(\hat{X}_s)|^2ds
    \\&\leq  C \lambda^2 d^{2l+2} + \int_0^{\lambda(1-u)} \sqrt{\E |\hat{X}^{+}_{\lambda u}-\hat{X}_s|^4} \sqrt{\E (|\hat{X}^{+}_{\lambda u}|+|\hat{X}_s|+1)^{4l}}ds
\end{aligned}
\end{equation}
where the first term is bounded by the taming approximation and the second by the local Lipschitz property of $h$.
Noticing that \[\E |\hat{X}_s-\hat{X}^{+}_{\lambda u}|^4 \leq C (\E |\hat{X_s}-X_0|^4 +\E |\hat{X}^{+}_{\lambda u}-X_0|^4)\leq C d^{2l+2}\lambda^2  + \lambda^4\E |h_\lambda(\hat{X}^{+}_{\lambda u})-h_\lambda(x)|^4\leq C(x) \lambda^2\] and substituting in \eqref{eq-1kl} one deduces that 
\[TV(\mathcal{L}(\hat{X}_{\lambda(1-u)})|\mathcal{L}(Y_{\lambda(1-u)}))\leq \sqrt{2 KL(\mathcal{L}(\hat{X}_{\lambda(1-u)})|\mathcal{L}(Y_{\lambda(1-u)}))}\leq \sqrt{C(x)} \lambda.\]
% Taking the expectation with respect to the Brownian motion one obtains
% \begin{equation}\label{KL-main1}
%     \E \left(KL(\delta_x \hat{P}^u|| \delta_{X_0}\hat{P}^{\lambda(1-u)})\right)\leq \lambda^2 C(x)
% \end{equation}
 In addition, by the Renyi regularity of Langevin Kernel, there holds that if $\mu_0=\mathcal{L}(X_0)$
 \begin{equation}\label{eq-reg-tv}
     TV ((\delta_xP^{\lambda u}) P^{\lambda (1-u)},\mu_0 P^{\lambda (1-u)})\leq \frac{C}{\sqrt{\lambda (1-u)}} W_2 (\mu_0,\delta_xP^{\lambda u})\leq \frac{C}{\lambda} W_2 (\mu_0,\delta_xP^{\lambda u})
 \end{equation}
 where the last step is due to the assumption $u\leq 1-\lambda.$\\\
 It now remains to bound the Wasserstein distance.
 Taking \[Y_{\lambda u}= x- \int_0^{\lambda u }h(Y_s) ds + \int_0^{\lambda u }\sqrt{2\beta^{-1}} dB_s \] with same Brownian motion as $X_0$ one obtains that 
 \[\begin{aligned}
     \E |X_0-Y_{\lambda u}|^2&\leq  \E |\int_0 ^{\lambda u} h_\lambda (\hat{X}^{+}_{\lambda u})-h(Y_s) ds|^2 \\&\leq  2 \lambda u \E \int_0 ^{\lambda u}| h_\lambda (\hat{X}^{+}_{\lambda u})-h(\hat{X}^{+}_{\lambda u})|^2 ds + 2 \lambda u \E \int_0 ^{\lambda u}| h (\hat{X}^{+}_{\lambda u})-h(Y_s)|^2 ds
     \\&\leq 
     \lambda^4 \E |\hat{X}^{+}_{\lambda u}|^{2l+2} + 2\lambda u \int_0^{\lambda u} \sqrt{\E |\hat{X}^{+}_{\lambda u}-Y_s|^4} \sqrt{\E (1+|\hat{X}^{+}_{\lambda u}|+ |Y_s|)^{2l}}ds \end{aligned}\]
     where the first term was bounded by taming error while the second is controlled by the local Lipschitz continuity and Cauchy Swartz inequality.
     Since \[\E |Y_s-\hat{X}^{+}_{\lambda u}|^4\leq 8 \E | \int_0^s h(Y_r)dr|^4 + 8 \E |\int_0^{\lambda u} h_\lambda (x)dr|^4\leq (\lambda u)^3 \int_0^{\lambda u} \E |h(Y_r)|^4 dr + 8 (\lambda u)^4 |x|^{4l+4}\]
     Since the moment bounds of $Y_r$ can be controlled by the moment bounds of the initial condition $x$ while the same is true for the moment bounds of $\hat{X}^{+}_{\lambda u}$, one deduces that 
    \[ \E |X_0-Y_{\lambda u}|^2\leq \lambda^4 C(x)\] where $C(x)$ depends polynomially on $x.$
   Plugging this into  \eqref{eq-reg-tv} yields
   \[TV ((\delta_xP^{\lambda u}) P^{\lambda (1-u)}|| \mu_0 P^{\lambda (1-u)})\leq C_{TV}(x)\lambda.\]
 % \[\begin{aligned}
 %     R_2 ((\delta_xP^{\lambda u}) P^{\lambda (1-u)}|| \mu_0 P^{\lambda (1-u)})&\leq \frac{2C}{\lambda (1-u)}\left(\E |X_0-Y_{\lambda u}|^2\right)
 %     \\&\leq \frac{2C}{\lambda (1-u)} \E |\int_0^{\lambda u}h_\lambda(\hat{X}^{+}_{\lambda u})-h(Y_s) ds |^2
 %     \\&\leq \frac{C}{1-u}\int_0^{\lambda u}\E | h_\lambda(\hat{X}^{+}_{\lambda u})-h(Y_s) ds |^2
 %     \\&\leq \frac{2C}{1-u}\int_0^{\lambda u}\E | h_\lambda(\hat{X}^{+}_{\lambda u})-h(\hat{X}^{+}_{\lambda u}) |^2ds  \\&+ \frac{2C}{1-u}\int_0^{\lambda u}\E | h(\hat{X}^{+}_{\lambda u})-h(Y_s)  |^2ds
 %     \\&\leq  \frac{2C}{1-u} C'(x) \lambda^2
 % \end{aligned} \]
 % which leads to \begin{equation}\label{eq-main2KL}
 %      \E R_2 ((\delta_xP^{\lambda u}) P^{\lambda (1-u)}|| \delta_ {X_0} P^{\lambda (1-u)})\leq \frac{2C}{1-u} C'(x) \lambda^2
 % \end{equation}
 % Now taking expectation with respect to the Brownian motion $B_{\lambda u}$ one deduces that 
 % \[KL (\delta_x \hat{P}^u||\delta_x P)\leq \E \left(KL(\delta_x \hat{P}^u|| \delta_{X_0}\hat{P}^{\lambda(1-u)}) + 2R_2(\delta_{X_0}\hat{P}^{\lambda(1-u)}|| (\delta_x P^{\lambda u}) P^{\lambda (1-u)})\right) \]
 \end{proof}

\begin{lemma}
    There holds \[TV(\delta_x P^{RLMC}, \delta_x P)\leq C_{TV}(x) \lambda\]  
\end{lemma}
\begin{proof}
    Conditioning on the uniform random variable $u$ of the randomized scheme,
    \[\begin{aligned}TV(\delta_x P^{RLMC}, \delta_x P)&\leq \E_uTV(\delta_x \hat{P}^u,\delta_x P)\\&=\int_0^{1-\lambda} TV(\delta_x \hat{P}^u,\delta_x P)du + \int_{1-\lambda}^1  TV(\delta_x \hat{P}^u,\delta_x P)du \\&\leq \int_0^{1-\lambda} C(x)\lambda du + \lambda\\& \leq C_{TV}(x)\lambda \end{aligned}\]
    where the first term is given by Lemma \ref{Kl small u}
\end{proof}

\subsubsection{TV-shifted bound} \label{se:TV-shifted}
\begin{proof}[Proof of Lemma~\ref{lem:tv-shifted}]
Since $\widehat\mu_N=\nu_N'$, $Y_N'\sim \widehat P(\widetilde Y_{N-1},\cdot)$, and $Y_N\sim P(Y_{N-1},\cdot)$, applying the triangle inequality yields the decomposition
\[
  \TV(\widehat\mu_N,\nu_N)
  =
  \TV\bigl(\Law(Y_N'),\Law(Y_N)\bigr)
  \le I_1+I_2+I_3,
\]
where
\begin{align*}
  I_1 &:= \E\bigl[\TV\bigl(\widehat P(\widetilde Y_{N-1},\cdot),P(\widetilde Y_{N-1},\cdot)\bigr)\bigr],\\
  I_2 &:= \E\bigl[\TV\bigl(P(\widetilde Y_{N-1},\cdot),P(Y'_{N-1},\cdot)\bigr)\bigr],\\
  I_3 &:= \TV(\nu'_{N-1}P,\nu_{N-1}P).
\end{align*}

By the local error estimate \eqref{eq:tv-local-error}, the first term satisfies $I_1\le C_{\mathrm{TV}}(x)\lambda$.

For the second term, applying Pinsker's inequality and the regularity condition \eqref{eq:reg-framework} almost surely yields $\TV\bigl(P(\widetilde Y_{N-1},\cdot),P(Y'_{N-1},\cdot)\bigr) \le \sqrt{\frac{c}{2}\|\widetilde Y_{N-1}-Y'_{N-1}\|^2}$. Taking the expectation and applying Jensen's inequality provides
\[
  I_2 \le \sqrt{\frac{c}{2}\,\E\|\widetilde Y_{N-1}-Y'_{N-1}\|^2}.
\]
Since $\eta_{N-1}=1$, the identity $\widetilde Y_{N-1}=\widehat X_{N-1}$ holds exactly, bounding the second term by $I_2 \le \sqrt{\frac{c}{2}}\,d_{N-1}$.

For the third term, the contraction of total variation under Markov kernels implies $I_3\le \TV(\nu'_{N-1},\nu_{N-1})$. Applying Pinsker's inequality and telescoping the shifted chain rule up to time $N-1$ provides the relative entropy bound
\[
  \KL(\nu'_{N-1}\|\nu_{N-1}) \le \sum_{n=0}^{N-2}\E\Bigl[\KL\bigl(P(\widetilde Y_n,\cdot)\,\big\|\,P(Y_n',\cdot)\bigr)\Bigr].
\]
Invoking \eqref{eq:reg-framework} and the interpolation identity $\widetilde Y_n-Y_n'=\eta_n(\widehat X_n-Y_n')$ evaluates the sum as
\[
  \KL(\nu'_{N-1}\|\nu_{N-1}) \le c\sum_{n=0}^{N-2}\E\|\widetilde Y_n-Y_n'\|^2 = c\sum_{n=0}^{N-2}\eta_n^2\,d_n^2.
\]
Substituting this into the Pinsker bound yields $I_3 \le \sqrt{\frac{c}{2}\sum_{n=0}^{N-2}\eta_n^2\,d_n^2}$.

Combining the bounds for $I_1, I_2$, and $I_3$ establishes \eqref{eq:tv-shifted-bound}.
\end{proof}

\subsection{Proofs of the Main Results for kTULA}

\subsubsection{Proof of Theorem~\ref{thm:kl-local-kTULA}}
\label{proof:kl-local-kTULA}
\begin{proof}
By Propositions~\ref{prop:W2-Lip-app}, \ref{prop:regularity-app},
and~\ref{prop:cross-regularity-app}, together with the one-step weak and strong
error bounds, the quantities of the KL framework~\ref{thm:KL-local-error-framework}
satisfy, uniformly over $0\le n<N$,
\[
  \overline a_0^{\,2}=O(\lambda^3),
  \qquad
  \overline a_1^{\,2}=O(\lambda^4),
  \qquad
  c=O(\lambda^{-1}),
  \qquad
  c'=O(\lambda^{-1}).
\]

Set $L:=e^{K'\lambda}$, $\bar N:=N\wedge(1-L)_+^{-1}$, and
$\bar b:=\sup_{0\le n<N}\|b\|_{L^2(\widehat\mu_n)}$. By hypothesis
$L\in[\tfrac12,2]$, and $\bar N\le N$ in both regimes: $\bar N=N$ when
$L\ge1$, and $\bar N\le(1-L)^{-1}\le N$ when $L<1$. For $L\neq1$, the
Wasserstein coefficient satisfies
\[
  \frac{L^{-1}-1}{L^{-N}-1}=\frac{L-1}{L^{N}-1}\in[0,1],
\]
the bound holding when $L\ge1$, since then $L^{N}-1\ge L-1\ge0$, and when
$L\le1$, then $L^{N}-1\le L-1<0$, and the case $L=1$ following by continuity.

Applying the estimate~\eqref{eq:kl-blackbox} and bounding the
Wasserstein coefficient by $1$,
\[
  \KL\!\bigl(\mu\widehat P^{N}\,\big\|\,\nu P^{N}\bigr)
  \;\lesssim\;
  (c+c')\Bigl[
    W_2^2(\mu,\nu)
    +\bigl((L-1)N\vee\log\bar N\bigr)\overline a_0^{\,2}
    +\bar N\,\overline a_1^{\,2}
  \Bigr]
  +\bar b^{\,2}.
\]
We bound the bracketed terms. Since $L\in[\tfrac12,2]$, the mean value theorem
gives
\[
  |L-1|=|e^{K'\lambda}-1|\le e^{|K'|\lambda}|K'|\lambda\le 2|K'|\lambda=O(\lambda),
\]
so, together with $\log\bar N\le\log N$ and $\overline a_0^{\,2}=O(\lambda^3)$,
the second bracketed term is $O\bigl(((N\lambda)\vee\log N)\lambda^3\bigr)$, while
$\bar N\,\overline a_1^{\,2}\le N\,\overline a_1^{\,2}=O(N\lambda^4)$.
Multiplying the bracket by $(c+c')=O(\lambda^{-1})$ therefore contributes
\[
  O(\lambda^{-1})\,W_2^2(\mu,\nu)
  +O\bigl(((N\lambda)\vee\log N)\lambda^2\bigr)
  +O(N\lambda^3).
\]

For the remaining term, Proposition~\ref{prop:cross-regularity-app} yields
$b(x)^2\le C\lambda^2(1+\|x\|^{2\ell+2})$, so
\[
  \bar b^{\,2}
  \le
  C\lambda^2\Bigl(1+\sup_{n\ge0}\E\|\widehat X_n\|^{2\ell+2}\Bigr)
  \le
  C\lambda^2,
\]
the last bound by Lemma~\ref{lem:moment-algorithm} with $p=\ell+1$.

Collecting these contributions,
\[
  \KL\!\bigl(\mu\widehat P^{N}\,\big\|\,\nu P^{N}\bigr)
  \le
  \frac{C_1}{\lambda}\,W_2^2(\mu,\nu)
  +C\bigl((N\lambda)\vee\log N\bigr)\lambda^2
  +C\,N\lambda^3
  +C\,\lambda^2.
\]
Finally, since $0<\lambda\le1$ and $N\ge3$ (so $\log N\ge1$), both $\lambda^2$
and $N\lambda^3=(N\lambda)\lambda^2$ are dominated by
$((N\lambda)\vee\log N)\lambda^2$; absorbing them yields
\[
  \KL\!\bigl(\mu\widehat P^{N}\,\big\|\,\nu P^{N}\bigr)
  \le
  \frac{C_1}{\lambda}\,W_2^2(\mu,\nu)
  +C\bigl((N\lambda)\vee\log N\bigr)\lambda^2,
\]
for a constant $C>0$ independent of $N$ and $\lambda$.
\end{proof}

\subsubsection{Proof of Corollary~\ref{prop:kl-sampling-ktula}}
\label{proof:kl-sampling-ktula}
\begin{proof}
Applying the divergence bound $\KL(\rho\|\pi)\le \KL(\rho\|\eta)+R_2(\eta\|\pi)$ to $\rho=\mu\widehat P^N$, $\eta=\nu P^N$, and $\pi=\pi_\beta$ directly yields \eqref{eq:kl-ktula-to-pi-general}. To deduce \eqref{eq:kl-ktula-to-pi-sameinit}, we bound the components of \eqref{eq:kl-ktula-to-pi-general} using Theorem~\ref{thm:kl-local-kTULA} and the Rényi-$2$ contraction of the diffusion semigroup under \ref{ass:LSI}, obtaining
\[
\KL\!\bigl(\mu\widehat P^N\|\pi_\beta\bigr)
\le
\frac{C_1}{\lambda}\W^2(\mu,\nu)
+
C\bigl((N\lambda)\vee\log N\bigr)\lambda^2
+
e^{-2C_{\mathrm{LSI}}N\lambda} R_2\!\bigl(\nu\|\pi_\beta\bigr).
\]
Evaluating this at $\nu=\mu$ forces $\W^2(\mu,\mu)=0$, yielding \eqref{eq:kl-ktula-to-pi-sameinit} as claimed.
\end{proof}

\subsubsection{Proof of Proposition~\ref{prop:mixing-time-kl-ktula}}
\label{proof:mixing-time-kl-ktula}
\begin{proof}
By~\eqref{eq:kl-ktula-to-pi-sameinit},
\[
  \KL\!\bigl(\mu\widehat P^{N}\,\big\|\,\pi_\beta\bigr)
  \le
  C\,\bigl((N\lambda)\vee \log N\bigr)\lambda^2
  + e^{-2C_{\mathrm{LSI}}N\lambda}R_0,
\]
so it suffices to make each term at most $\varepsilon/2$. Set
\[
  T:=\frac{1}{2C_{\mathrm{LSI}}}\Bigl[\log\!\Bigl(\tfrac{2R_0}{\varepsilon}\Bigr)\Bigr]_+,
  \qquad
  \lambda:=\min\!\Bigl\{\lambda_{\max}^{kTULA},\;
  \sqrt{\tfrac{\varepsilon}{2C\,(T\vee \log(3+T/\lambda))}}\Bigr\},
  \qquad
  N:=\max\bigl\{3,\lceil T/\lambda\rceil\bigr\}.
\]

Since $N\ge\lceil T/\lambda\rceil$ we have $N\lambda\ge T$, hence
$e^{-2C_{\mathrm{LSI}}N\lambda}R_0\le e^{-2C_{\mathrm{LSI}}T}R_0$. If $T>0$ then, by the definition of $T$,
$e^{-2C_{\mathrm{LSI}}T}R_0=e^{-\log(2R_0/\varepsilon)}R_0=\varepsilon/2$. If $T=0$ then $R_0\le\varepsilon/2$ by the same definition. In either case
\[
  e^{-2C_{\mathrm{LSI}}N\lambda}R_0\le \frac{\varepsilon}{2}.
\]

From $N\le 3+T/\lambda$ we get $N\lambda\le T+3\lambda$ and $\log N\le\log(3+T/\lambda)$, so, using $0<\lambda\le\lambda_{\max}^{kTULA}$,
\[
  (N\lambda)\vee \log N
  \le
  (T+3\lambda)\vee \log\!\Bigl(3+\tfrac{T}{\lambda}\Bigr)
  \le
  C'\Bigl(T\vee \log\!\Bigl(3+\tfrac{T}{\lambda}\Bigr)\Bigr)
\]
for an absolute constant $C'>0$. By the choice of $\lambda$ and absorbing $C'$ into $C$,
\[
  C\,\bigl((N\lambda)\vee \log N\bigr)\lambda^2
  \le
  CC'\Bigl(T\vee \log\!\bigl(3+\tfrac{T}{\lambda}\bigr)\Bigr)\lambda^2
  \le
  \frac{\varepsilon}{2}.
\]
Combining the two estimates gives $\KL(\mu\widehat P^{N}\|\pi_\beta)\le\varepsilon$.

There $\lambda\asymp\sqrt{\varepsilon/T}$ up to logarithmic factors, so
\[
  N\asymp \frac{T}{\lambda}
  =
  O\!\Bigl(\frac{T^{3/2}}{\sqrt{\varepsilon}}\Bigr)
  =
  O\!\left(
  \frac{1}{\sqrt{\varepsilon}}
  \Bigl[\log\!\Bigl(\tfrac{2R_0}{\varepsilon}\Bigr)\Bigr]_+^{3/2}
  \right)
  =
  \widetilde O\!\bigl(\varepsilon^{-1/2}\bigr),
\]
since $T=\tfrac{1}{2C_{\mathrm{LSI}}}[\log(\frac{2R_0}{\varepsilon})]_+$.
\end{proof}

\subsubsection{Proof of Corollary~\ref{cor:tv-mixing-ktula}}
\label{proof:tv-ktula}
\begin{proof}
By Pinsker's inequality and~\eqref{eq:kl-ktula-to-pi-sameinit},
\[
  \TV(\mu\widehat P^{N},\pi_\beta)
  \le
  \sqrt{\tfrac12\,\KL\!\bigl(\mu\widehat P^{N}\,\big\|\,\pi_\beta\bigr)}
  \le
  \sqrt{\tfrac12\Bigl(C\,((N\lambda)\vee \log N)\lambda^2 + e^{-2C_{\mathrm{LSI}}N\lambda}R_0\Bigr)},
\]
so it suffices to make each term inside the square root at most $\varepsilon^2/2$. This is achieved exactly as in the proof of Proposition~\ref{prop:mixing-time-kl-ktula}, with $\varepsilon$ replaced by $\varepsilon^2$ in the definition of $T$: setting
\[
  T:=\frac{1}{2C_{\mathrm{LSI}}}\Bigl[\log\!\Bigl(\tfrac{2R_0}{\varepsilon^2}\Bigr)\Bigr]_+,
  \qquad
  \lambda:=\min\!\Bigl\{\lambda_{\max}^{kTULA},\;
  \sqrt{\tfrac{\varepsilon^2}{2C\,(T\vee \log(3+T/\lambda))}}\Bigr\},
  \qquad
  N:=\max\bigl\{3,\lceil T/\lambda\rceil\bigr\},
\]
the same two estimates as there give $e^{-2C_{\mathrm{LSI}}N\lambda}R_0\le\varepsilon^2/2$ and $C((N\lambda)\vee\log N)\lambda^2\le\varepsilon^2/2$. Hence
\[
  \TV(\mu\widehat P^{N},\pi_\beta)
  \le
  \sqrt{\tfrac12\bigl(\tfrac{\varepsilon^2}{2}+\tfrac{\varepsilon^2}{2}\bigr)}
  =
  \frac{\varepsilon}{\sqrt2}
  \le \varepsilon.
\]
For the complexity, $\lambda\asymp\varepsilon/\sqrt{T}$ up to logarithmic factors, so
\[
  N\asymp \frac{T}{\lambda}
  =
  O\!\Bigl(\frac{T^{3/2}}{\varepsilon}\Bigr)
  =
  \widetilde O\!\bigl(\varepsilon^{-1}\bigr),
\]
since $T=\tfrac{1}{2C_{\mathrm{LSI}}}[\log(2R_0/\varepsilon^2)]_+$, where $\widetilde O(\cdot)$ hides polylogarithmic factors in $\varepsilon^{-1}$ and $R_0$.
\end{proof}

\subsubsection{Proof of Corollary~\ref{cor:w2-mixing-ktula}}
\label{proof:w2-ktula}
\begin{proof}
By Talagrand's inequality implied by Assumption~\ref{ass:LSI} and~\eqref{eq:kl-ktula-to-pi-sameinit},
\[
  W_2\!\bigl(\mu\widehat P^{N},\pi_\beta\bigr)
  \le
  \sqrt{\tfrac{2}{C_{\mathrm{LSI}}}\,\KL\!\bigl(\mu\widehat P^{N}\,\big\|\,\pi_\beta\bigr)}
  \le
  \sqrt{\tfrac{2}{C_{\mathrm{LSI}}}
  \Bigl(C\,((N\lambda)\vee \log N)\lambda^2 + e^{-2C_{\mathrm{LSI}}N\lambda}R_0\Bigr)},
\]
so it suffices to make each term inside the square root at most $\tfrac{C_{\mathrm{LSI}}}{4}\varepsilon^2$. This is achieved exactly as in the proof of Proposition~\ref{prop:mixing-time-kl-ktula}, with the target accuracy $\tfrac{C_{\mathrm{LSI}}}{4}\varepsilon^2$ in place of $\varepsilon/2$: setting
\[
  T:=\frac{1}{2C_{\mathrm{LSI}}}
  \Bigl[\log\!\Bigl(\tfrac{4R_0}{C_{\mathrm{LSI}}\varepsilon^2}\Bigr)\Bigr]_+,
  \qquad
  \lambda:=\min\!\Bigl\{\lambda_{\max}^{kTULA},\;
  \sqrt{\tfrac{C_{\mathrm{LSI}}\varepsilon^2}{4C\,(T\vee \log(3+T/\lambda))}}\Bigr\},
  \qquad
  N:=\max\bigl\{3,\lceil T/\lambda\rceil\bigr\},
\]
the same two estimates as there give $e^{-2C_{\mathrm{LSI}}N\lambda}R_0\le\tfrac{C_{\mathrm{LSI}}}{4}\varepsilon^2$ and $C((N\lambda)\vee\log N)\lambda^2\le\tfrac{C_{\mathrm{LSI}}}{4}\varepsilon^2$. Hence
\[
  W_2\!\bigl(\mu\widehat P^{N},\pi_\beta\bigr)
  \le
  \sqrt{\tfrac{2}{C_{\mathrm{LSI}}}\cdot\tfrac{C_{\mathrm{LSI}}}{2}\,\varepsilon^2}
  =\varepsilon.
\]
For the complexity, $\lambda\asymp\varepsilon/\sqrt{T}$ up to logarithmic factors, so
\[
  N\asymp \frac{T}{\lambda}
  =
  O\!\Bigl(\frac{T^{3/2}}{\varepsilon}\Bigr)
  =
  \widetilde O\!\bigl(\varepsilon^{-1}\bigr),
\]
since $T=\tfrac{1}{2C_{\mathrm{LSI}}}[\log\left(\frac{4R_0}{C_{\mathrm{LSI}}\varepsilon^2}\right)]_+$, where $\widetilde O(\cdot)$ hides polylogarithmic factors in $\varepsilon^{-1}$ and $R_0$.
\end{proof}

\subsubsection{Proof of Corollary~\ref{cor:excess-risk-ktula}}
\label{proof:expected-risk-ktula}
\begin{proof}
Recall that \(h=\nabla u\), and let \(Z_\infty\) be an \(\R^d\)-valued random variable with
\(\Law(Z_\infty)=\pi_\beta\). Then
\begin{equation}\label{eq:excess-splitting-correct}
  \E\!\bigl[u(\widehat X_N)\bigr]-\inf_{\theta\in\R^d}u(\theta)
  =
  \Bigl(\E[u(\widehat X_N)]-\E[u(Z_\infty)]\Bigr)
  +
  \Bigl(\E[u(Z_\infty)]-\inf_{\theta\in\R^d}u(\theta)\Bigr).
\end{equation}

We first estimate the difference \(\E[u(\widehat X_N)]-\E[u(Z_\infty)]\). Since \(h=\nabla u\), the
fundamental theorem of calculus gives, for all \(x,y\in\R^d\),
\[
  u(x)-u(y)
  =
  \int_0^1 \langle h\bigl(tx+(1-t)y\bigr),\,x-y\rangle\,dt.
\]
Using Assumption~\hyperref[ass:PJG]{\textup{(A1)}}, we obtain
\begin{align*}
  |u(x)-u(y)|
  &\le
  \int_0^1 \bigl\|h\bigl(tx+(1-t)y\bigr)\bigr\|\,\|x-y\|\,dt \\
  &\le
  L\int_0^1 \Bigl(1+\bigl\|tx+(1-t)y\bigr\|^{2\ell}\Bigr)\,dt\,\|x-y\| \\
  &\le
  C_h\bigl(1+\|x\|^{2\ell}+\|y\|^{2\ell}\bigr)\|x-y\|,
\end{align*}
for some constant \(C_h>0\).

Let \(P^\star\in\mathcal C(\Law(\widehat X_N),\pi_\beta)\) be an optimal coupling, so that
\[
  W_2^2\!\bigl(\Law(\widehat X_N),\pi_\beta\bigr)
  =
  \E_{P^\star}\|\widehat X_N-Z_\infty\|^2.
\]
Then
\begin{align*}
  \E[u(\widehat X_N)]-\E[u(Z_\infty)]
  &=
  \E_{P^\star}\bigl[u(\widehat X_N)-u(Z_\infty)\bigr] \\
  &\le
  C_h\,\E_{P^\star}\Bigl[
    \bigl(1+\|\widehat X_N\|^{2\ell}+\|Z_\infty\|^{2\ell}\bigr)
    \|\widehat X_N-Z_\infty\|
  \Bigr].
\end{align*}
By Cauchy--Schwarz,
\begin{equation}\label{eq:excess-risk-w2-step}
\begin{aligned}
  \E[u(\widehat X_N)]-\E[u(Z_\infty)]
  &\le
  C_h
  \Bigl(
    1+\bigl(\E\|\widehat X_N\|^{4\ell}\bigr)^{1/2}
    +\bigl(\E\|Z_\infty\|^{4\ell}\bigr)^{1/2}
  \Bigr) \\
  &\qquad\qquad\times
  W_2\!\bigl(\Law(\widehat X_N),\pi_\beta\bigr).
\end{aligned}
\end{equation}

By Lemma~\ref{lem:moment-algorithm},
\[
  \sup_{N\ge0}\E\|\widehat X_N\|^{4\ell}<\infty,
\]
and since \(Z_\infty\sim\pi_\beta\), Lemma~\ref{lem:uniform-moments-sde} yields
\[
  \E\|Z_\infty\|^{4\ell}<\infty.
\]
Hence there exists a constant \(M>0\), independent of \(N\) and \(\lambda\), such that
\[
  1+\bigl(\E\|\widehat X_N\|^{4\ell}\bigr)^{1/2}
  +\bigl(\E\|Z_\infty\|^{4\ell}\bigr)^{1/2}
  \le M.
\]
Substituting this into \eqref{eq:excess-risk-w2-step} yields
\[
  \E[u(\widehat X_N)]-\E[u(Z_\infty)]
  \le
  C_h M\,W_2\!\bigl(\Law(\widehat X_N),\pi_\beta\bigr).
\]
By Corollary~\ref{cor:w2-mixing-ktula} (Talagrand's inequality combined
with~\eqref{eq:kl-ktula-to-pi-sameinit}),
\[
  W_2\!\bigl(\Law(\widehat X_N),\pi_\beta\bigr)
  \le
  C_1\,e^{-C_0 N\lambda}
  +
  C_2\,\lambda\,\sqrt{(N\lambda)\vee\log N},
\]
for constants $C_0,C_1,C_2>0$ depending only on $C_{\mathrm{LSI}}$, $R_0$, and
the constant of~\eqref{eq:kl-ktula-to-pi-sameinit}. Substituting into
$\E[u(\widehat X_N)]-\E[u(Z_\infty)]\le C_h M\,W_2(\Law(\widehat X_N),\pi_\beta)$
and absorbing $C_h M$ into the constants,
\begin{equation}\label{eq:first-term-excess-correct}
  \E[u(\widehat X_N)]-\E[u(Z_\infty)]
  \le
  C_1\,e^{-C_0 N\lambda}
  +
  C_2\,\lambda\,\sqrt{(N\lambda)\vee\log N}.
\end{equation}

For the second term in \eqref{eq:excess-splitting-correct}, a bound is obtained by following exactly the
same argument as in \cite[Lemma 4.9]{lim2023nonasymptoticestimatestuslaalgorithm}. In the present
setting, this yields a constant \(C_3>0\) such that
\begin{equation}\label{eq:second-term-excess-correct}
  \E[u(Z_\infty)]-\inf_{\theta\in\R^d}u(\theta)
  \le
  \frac{C_3 \log\beta}{\beta}.
\end{equation}

Combining \eqref{eq:excess-splitting-correct}, \eqref{eq:first-term-excess-correct}, and
\eqref{eq:second-term-excess-correct} yields
\[
  \E\!\bigl[u(\widehat X_N)\bigr]-\inf_{\theta\in\R^d}u(\theta)
  \le
  C_1\,e^{-C_0N\lambda}
  +
  C_2\,\lambda\,\sqrt{(N\lambda)\vee\log N}
  +
  \frac{C_3\log\beta}{\beta},
\]
which proves \eqref{eq:excess-risk-ktula-statement}.
\end{proof}

\subsection{Proofs of the Main Results for tRLMC}

\begin{proof}[Proof of Theorem~\ref{thm:tv-local-framework}]
Applying Lemma~\ref{lem:tv-shifted-bound}, it remains to bound the shifted-distance term.  
For this, we use the deterministic estimates \eqref{eq:dn-terminal-bound} and
\eqref{eq:dn-sum-bound}, where the quantities \(\overline a_0\) and \(\overline a_1\)
are controlled through the strong local error from Proposition~\ref{prop:strong-error-tamed-midpoint},
the weak local error from Proposition~\ref{prop:weak-error-tamed-midpoint}, and the uniform
moment bounds from Lemmas~\ref{lem:moment-tamed-rLMC} and~\ref{lem:tildeY-moment-tamed-rLMC}.
Substituting these bounds into Lemma~\ref{lem:tv-shifted-bound} yields the claimed estimate.
\end{proof}